\documentclass[12pt,reqno]{amsart}

\usepackage[a4paper, centering, left=3cm, right=3cm, top=3cm, bottom=3cm]{geometry}
\allowdisplaybreaks
\numberwithin{equation}{section}

\usepackage{amsmath,amsthm,amssymb,mathrsfs,cases}
\usepackage{mathtools}
\usepackage{nicematrix,tikz-cd}
\usetikzlibrary{positioning}
\usepackage{enumitem}

\usepackage{setspace}
\makeatletter
\renewcommand{\@secnumfont}{%
  \ifnum\value{subsection}>0 \bfseries \else \normalfont \fi
}
\renewcommand{\section}{\@startsection{section}{1}%
  \z@{-3.5ex \@plus -1ex \@minus -.2ex}{2.3ex \@plus.2ex}%
  {\normalfont\scshape\centering}}
\renewcommand{\subsection}{\@startsection{subsection}{2}%
  \z@{-3.25ex \@plus -1ex \@minus -.2ex}{1.5ex \@plus .2ex}%
  {\normalfont\bfseries}}
\renewcommand{\subsubsection}{\@startsection{subsubsection}{3}%
  \z@{-3.25ex \@plus -1ex \@minus -.2ex}{1.5ex \@plus .2ex}%
  {\normalfont\bfseries}}
\makeatother

\usepackage{xcolor}
\definecolor{tealblue}{rgb}{0.0, 0.45, 0.45}
\definecolor{brickred}{rgb}{0.6, 0.1, 0.1}
\definecolor{purnku}{rgb}{0.494,0.047,0.431}
\definecolor{winered}{rgb}{0.5,0,0}
\definecolor{softblue}{RGB}{25,25,112}
\definecolor{darkviolet}{rgb}{0.4, 0.0, 0.4}
\usepackage[hidelinks]{hyperref}
\hypersetup{colorlinks=true, linkcolor=brickred, citecolor=tealblue, urlcolor=darkviolet}
\usepackage[capitalize,noabbrev,sort&compress]{cleveref}

\newcommand*{\bbr}{\mathbb{R}}

\newcommand*{\bbc}{\mathbb{C}}
\newcommand*{\bbs}{\mathbb{S}}
\newcommand*{\bbt}{\mathbb{T}}

\newcommand*{\bbz}{\mathbb{Z}}

\newcommand*{\calb}{\mathcal{B}}
\newcommand*{\cale}{\mathcal{E}}

\newcommand*{\calf}{\mathcal{F}}

\newcommand*{\calr}{\mathcal{R}}

\newcommand*{\calw}{\mathcal{W}}
\newcommand*{\caly}{\mathcal{Y}}
\newcommand*{\call}{\mathcal{L}}
\newcommand*{\calm}{\mathcal{M}}

\newcommand*{\calk}{\mathcal{K}}
\newcommand*{\calq}{\mathcal{Q}}
\newcommand*{\cals}{\mathcal{S}}

\newcommand*{\scrr}{\mathscr{R}}

\newcommand*{\cm}{\mathrm{c}}

\newcommand*{\di}{\mathrm{d}}

\DeclareMathOperator{\ind}{ind}

\DeclareMathOperator{\tr}{{tr}}

\DeclareMathOperator{\inter}{Int}

\DeclareMathOperator{\supp}{supp}

\DeclareMathOperator{\rk}{rank}

\DeclareMathOperator{\ch}{ch}

\DeclareMathOperator{\id}{Id}
\DeclareMathOperator{\dist}{dist}

\DeclareMathOperator{\sat}{Sat}
\DeclareMathOperator{\re}{Re}

\newcommand*{\bv}{{\beta,\varepsilon}}
\newcommand*{\bvr}{{\beta,\varepsilon,r}}

\newcommand*{\f}{{\calf_1^\perp}}
\newcommand*{\ff}{{\calf_2^\perp}}

\newcommand*{\hfz}{{\widetilde{\calf}}}
\newcommand*{\hf}{{\widetilde{\calf}_{1}^\perp}}
\newcommand*{\hff}{{\widetilde{\calf}_{2}^\perp}}

\newcommand*{\rtwo}{\mathrm{\uppercase\expandafter{\romannumeral2}}}

\newcommand*{\ws}{{\calw_r}}

\newcommand*{\ms}{{{\calm}_r}}

\newcommand*{\kk}{{\widetilde{\calk}_{r}}}

\newcommand*{\mm}{{\widetilde{\calm}_{r}}}
\newcommand*{\pp}{{\widetilde{\pi}_{r}}}

\newcommand*{\ww}{{\widetilde{\calw}_{r}}}
\newcommand*{\ee}{{\widetilde{\cale}}}

\newcommand*{\we}{\bar{e}}

\newcommand*{\rmh}{\mathrm{H}}
\newcommand*{\rml}{\mathrm{L}} 

\newcommand*\abs[1]{\lvert#1\rvert}
\newcommand*\babs[1]{\big\lvert#1\big\rvert}
\newcommand*\bbabs[1]{\Big\lvert#1\Big\rvert}

\newcommand*\inn[1]{\langle#1\rangle}

\newcommand*\binn[1]{\big\langle#1\big\rangle}
\newcommand*\bbinn[1]{\Big\langle#1\Big\rangle}
\newcommand*\bbbinn[1]{\bigg\langle#1\bigg\rangle}
\newcommand*\ket[1]{(#1)}

\newcommand*\bket[1]{\big(#1\big)}
\newcommand*\bbket[1]{\Big(#1\Big)}
\newcommand*\bbbket[1]{\bigg(#1\bigg)}

\newcommand*\bcu[1]{\big[#1\big]}

\newcommand*\bac[1]{\big\{#1\big\}}
\newcommand*\bbac[1]{\Big\{#1\Big\}}
\newcommand*\bbbac[1]{\bigg\{#1\bigg\}}

\newcommand*\en[1]{\mathrm{End}(#1)}

\newtheorem{thm}{Theorem}[section]
\newtheorem{lemma}[thm]{Lemma}

\newtheorem{prop}[thm]{Proposition}
\newtheorem{prob}[thm]{Problem}

\theoremstyle{definition}
\newtheorem{rem}[thm]{Remark}

\newtheorem{defn}[thm]{Definition}

\crefname{thm}{Theorem}{Theorems}
\crefname{conj}{Conjecture}{Conjectures}
\crefname{prop}{Proposition}{Propositions}
\crefname{set}{Setting}{Settings}
\crefname{prob}{Problem}{Problems}
\crefname{cor}{Corollary}{Corollaries}
\crefname{lem}{Lemma}{Lemmas}
\crefname{defn}{Definition}{Definitions}
\crefname{rem}{Remark}{Remarks}
\crefname{exam}{Example}{Examples}

\begin{document}

\title[Positive scalar curvature on foliations: long neck problem]{Positive scalar curvature on foliations: \\ long neck problem}

\author{Hengyu Chen and Guangxiang Su}
 
\address{Chern Institute of Mathematics \& LPMC,\newline
\hspace*{18pt}Nankai University, Tianjin 300071, P. R. China}
\email{hichy@mail.nankai.edu.cn}

\address{Chern Institute of Mathematics \& LPMC,\newline
\hspace*{18pt}Nankai University, Tianjin 300071, P. R. China}
\email{guangxiangsu@nankai.edu.cn}
\date{\today}

\begin{abstract}
We prove estimates for Gromov's Long Neck Foliated Conjecture in the spin cases, which involves leafwise scalar curvature. The same methods yield a quantitative version of a theorem of Su--Wang--Zhang and a Llarull-type theorem for closed foliated manifolds. A refinement, which involves mean curvature and partially generalizes a theorem of Cecchini--Zeidler, is also obtained. 
\end{abstract}

\maketitle

\tableofcontents

\section{Introduction}

\subsection{Long neck problem}
The study of manifolds with positive scalar curvature metrics is an important subject in differential geometry. Using the Dirac operator, Lichnerowicz \cite{lic63} proved that if a closed spin manifold $M$ admits a positive scalar curvature metric, then $\widehat{A}(M)=0$. On compact manifolds with boundary, Gromov  proposes the following problem:

\begin{prob}[Long Neck Problem {\cite[p.~330]{g23}}]
Let $(M,g)$ be an $m$-dimensional, orientable, Riemannian manifold with boundary and $\Phi:M\to \bbs^m$ be a smooth area non-increasing map, where $\bbs^m$ is the standard sphere. 

What kind of a lower bound on the scalar curvature $k_g$ and a lower bound on the ``length of the neck" of $(M,\Phi)$, that is the distance between the support of the differential of $\Phi$ and the boundary of $M$, would make $\deg\Phi=0$?
\end{prob}

Cecchini \cite[Thm.~A]{ce20} settles the long neck problem for spin manifolds and strictly area decreasing maps using the Dirac operator. In odd dimensions, the map is further assumed to be constant near the boundary. 

In \cite[Thm.~1.4~and~Cor.~1.5]{cz24}, Cecchini and Zeidler refine \cite[Thm.~A]{ce20} for even-dimensional spin manifolds and area non-increasing maps using a lower bound on the mean curvature. Via the spectral flow method (cf.\ \cite{lsw}), Shi \cite{shi25-1} extends \cite[Thm.~1.4~and~Cor.~1.5]{cz24} to odd-dimensional spin manifolds, where the maps are locally constant near the boundary and area non-increasing.

For foliated manifolds with boundary, Gromov \cite[pp.~258--259]{g23} proposes the Long Neck Foliated Conjecture and suggests that combining \cite{ce20} and \cite{swz22} may lead to a proof for spin manifolds or spin foliations. Combining the methods from \cite{cz24,chen,swz22,wx25,zhang17,zhang20}, we will prove estimates for this conjecture in the spin cases (see \cref{t1.2}). We first recall some definitions. 

Throughout this paper, let $\bbs^m$ denote the unit sphere equipped with the metric $g_0$ induced from the standard flat metric on $\bbr^{m+1}$. Suppose $M$ is a smooth manifold of dimension $m$ with or without boundary. We call $(M,F)$ a foliated manifold, if $F\subseteq TM$ is a smooth integrable subbundle of rank $q$; that is, for any smooth sections $X,Y\in \Gamma(F)$, we have
\begin{equation}
    [X,Y]\in\Gamma(F).
\end{equation}
 Unless otherwise specified, we do not assume that $F$ is either transverse or tangent to the boundary $\partial M$ of $M$. Let $\inter M$ denote the interior of $M$. By the  Frobenius Theorem, $F|_{\inter M}$ determines a foliation on $\inter M$ in the classical sense.

Suppose $\Phi\colon M\to \bbs^m$ is a smooth map. For $x\in M$, let $\di_x\Phi\colon T_xM\to T_{\Phi(x)}\bbs^m$ denote the differential of $\Phi$ at $x$. The \emph{support} of $\di\Phi$ is defined to be 
\begin{equation}
    \supp(\di\Phi) = \overline{\{x\in M:\di_x\Phi \ne 0\}}.
\end{equation}

Given a smooth metric $g^F$ on $F$, we can simultaneously diagonalize $g^F$ and $(\Phi^*g_0)|_F$ at each point $x\in M$ to obtain a basis $\{e_j\}_{j=1}^q$ of $F_x$ that is orthonormal with respect to $g^F$ and orthogonal with respect to $(\Phi^*g_0)|_F$. For $1\leq j \leq q$, define
\begin{equation}
    \lambda_j(x) = |\di_x\Phi(e_j)|_{g_0}\ge 0,
\end{equation}
which are the singular values of the differential $\di_x\Phi|_F$. 
Let $|\cdot|_{\tr}$ denote the trace norm of a linear map. The differential $\di\Phi|_F\colon F\to T\bbs^m$ induces a linear map $\wedge^2 (\di\Phi|_F) \colon \wedge^2F \to \wedge^2 \,T\bbs^m$, where $\wedge^2 F$ and $\wedge^2 \,T\bbs^m$ are the second exterior powers of the corresponding bundles. At $x\in M$, the \emph{trace norm of} $\wedge^2 (\di\Phi|_F)$ is given by
\begin{equation}\label{deflam}
    |\wedge^2\!(\di\Phi|_F)\,|_{\tr}(x) = \sum_{1\leq i<j\leq q}\lambda_i(x)\,\lambda_j(x)
\end{equation}
and is a continuous function on $M$.

For subsets $K_0,K_1\subseteq M$, the leafwise distance ${\rm dist}_F(K_0,K_1)\in[0,\infty]$ is defined as the infimum of the length of smooth curves $\gamma\colon [0,1]\to M$ such that
\begin{equation}\label{dis00}
    \gamma(b)\in K_b\;(b=0,1)\quad\text{and}\quad
    \dot{\gamma}(t)\in F_{\gamma(t)}\;(0\le t\le 1).
\end{equation} 
If no smooth curve satisfies \eqref{dis00}, we define $\dist_F(K_0,K_1) = \infty$. For a subset $L\subseteq M$, we define the \emph{saturation} of $L$ on $(M,F)$ to be 
\begin{equation}\label{defsat}
    \sat_M(L) = \{x\in M:\dist_F(x,L)<\infty\}\subseteq M.
\end{equation}
When $\partial M =\varnothing$, we have $\inter M = M$, and then $\sat_M(L)\subseteq M$ is the union of the leaves that intersect $L$. For simplicity, write $\sat(L)$ for $\sat_M(L)$ if $\partial M=\varnothing$.

At each point $x\in \inter M$, let $\call_x$ denote the leaf of $(\inter M,F)$ that passes through $x$, which is a smooth immersed submanifold of $\inter M$ such that $F|_{\call_x} = T\call_x$. The \emph{leafwise scalar curvature} $k^F\in C^\infty(\inter M)$ is defined as the scalar curvature of $\call_x$ with respect to $g^F$. 

Our first theorem is about Gromov's Long Neck Foliated Conjecture \cite[pp.~258--259]{g23} in the spin cases. 

\begin{thm}\label{t1.2}
Let $M$ be a compact, connected, oriented manifold of dimension $m$ with $\partial M\ne\varnothing$. Let $F\subseteq TM$ be a smooth integrable subbundle of rank $q\ge 2$ equipped with a smooth metric $g^F$. Assume either $TM$ or $F$ is spin. Let $\Phi\colon M\to\bbs^m$ be a smooth map that is locally constant near $\partial M$ and of nonzero degree. Let $L\subseteq M$ be a compact subset with $\sat_M(L)\subseteq M$ defined by \eqref{defsat}. If
\begin{enumerate}[label=\textup{(\roman*)},ref=\textup{\roman*}]
    \item $L\subseteq\supp(\di\Phi)\subseteq \sat_M(L)$;\label{1.2i}
    \item $k^F|_L > 2\,|\wedge^2\!(\di\Phi|_F)\,|_{\tr}$ and $k^F|_{\supp(\di\Phi)\setminus L} \ge 2\,|\wedge^2\!(\di\Phi|_F)\,|_{\tr}$;\label{1.2ii}
    \item $k^F|_{\inter M \setminus \supp(\di\Phi)} \geq \alpha^2 q(q-1)$ \label{1.2iii}
\end{enumerate}
for some $\alpha>0$, then
\begin{equation}\label{dis0}
    \dist_F(\supp(\di\Phi),\partial M)\le \frac{\pi}{\alpha q}.
\end{equation}
\begin{figure}[htbp]
    \centering
    \includegraphics[width=0.75\textwidth]{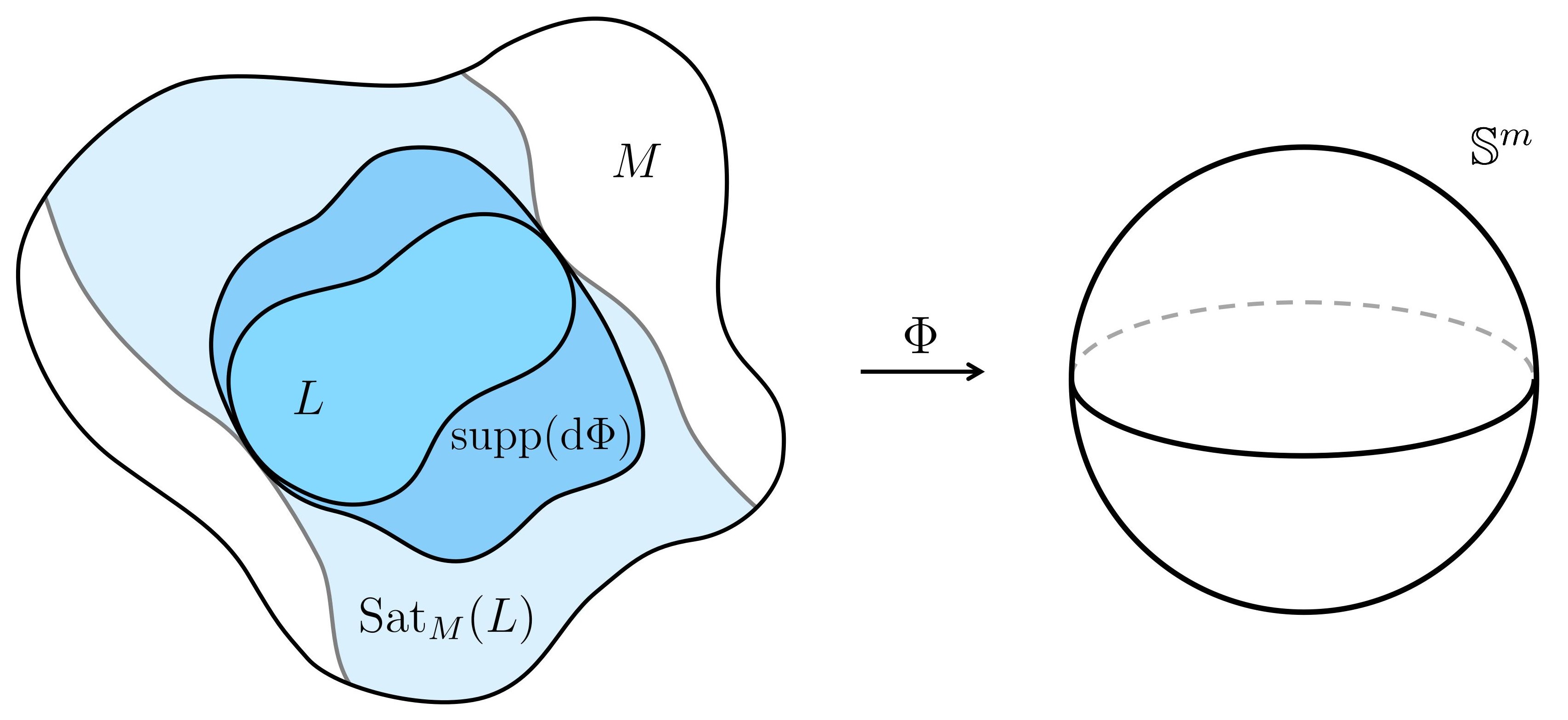}
    \caption{The map $\Phi\colon M\to\bbs^m$.}
\end{figure}
\end{thm}

\begin{rem}\label{rem1}
(1) If $F$ is transverse\footnote{That is, $F_x + T_x\partial M = T_xM$ for every $x\in\partial M$. } to $\partial M$ in \cref{t1.2}, then we have
\begin{equation}\label{1.5}
    \dist_F(\supp(\di\Phi),\partial M) < \frac{\pi}{\alpha q};
\end{equation}
see \cref{t1.2-0} and \cref{r4.6}. Note that $TM$ is always transverse to $\partial M$, so \eqref{1.5} generalizes \cite[Cor.~1.5]{cz24} and \cite[Cor.~1.6]{shi25-1} to foliations.

(2) For example, if $k^F|_{\supp(\di\Phi)} > 2\, |\wedge^2\!(\di\Phi|_F)\,|_{\tr}$, then $L = \supp(\di\Phi)$ satisfies the conditions \eqref{1.2i}--\eqref{1.2ii} in \cref{t1.2}. 

Assume $k^F|_{\supp(\di\Phi)} \ge 2\,|\wedge^2\!(\di\Phi|_F)\,|_{\tr}$ and $k^F|_{M \setminus \supp(\di\Phi)} \geq \alpha^2 q(q-1)$. If
\begin{equation}
    \supp(\di\Phi) \subseteq \sat_M\ket{\partial(\supp(\di\Phi))},
\end{equation}
where $\partial(\supp(\di\Phi))$ is the topological boundary of $\supp(\di\Phi)$, then $L = \partial(\supp(\di\Phi))$ satisfies $k^F|_L > 2\,|\wedge^2\!(\di\Phi|_F)\,|_{\tr}$.

(3) The condition ``$\supp(\di\Phi) \subseteq \sat_M(L)$'' is needed in \cref{t1.2}. For example, take $M = \bbs^{m-1}\times [-\pi,\pi]$ and $F = T\bbs^{m-1}$ ($m\ge 3$). The standard metric on $\bbs^{m-1}$ induces a  smooth metric $g^F$ on $F$. Take a smooth non-decreasing surjective function $f\colon [-\pi,\pi]\to [-\frac{\pi}{2},\frac{\pi}{2}]$ that is constant near $\pm\pi$. Define a smooth map
\begin{equation}
    \Phi\colon \bbs^{m-1}\times [-\pi,\pi]\to \bbs^m,\quad \Phi(x,t) = \bket{x\cdot \cos (f(t)), \sin (f(t))}.
\end{equation}
of nonzero degree, which is locally constant near the boundary. Moreover, we have
\begin{equation}
    k^F \equiv (m-1)(m-2) \ge 2\, |\wedge^2\!(\di\Phi|_F)\,|_{\tr}.
\end{equation}
However, $\dist_F(\supp(\di\Phi),\partial M) = \infty$ since the leaves are $\{\bbs^{m-1}\times\{t\}\}_{t\in[-\pi,\pi]}$.
\end{rem}

\begin{rem}
The estimate \eqref{1.5} is optimal in all dimensions. Suppose $2\le q\le m$ and $\alpha >0$. For a small positive number $\eta$, define $f_\eta:\bbr\to \bbr$ to be a smooth even function such that $f_\eta$ is monotone on $[0,\infty)$, $f_\eta \equiv 1$ on $[-\eta,\eta]$ and $f_\eta \equiv \alpha$ outside $[-2\eta,2\eta]$. Let $\varphi_\eta$ be the solution of the following ODE (cf.\ \cite[pp.~653--654]{g18}):
\begin{equation}
    \begin{cases}{}
        2 \frac{\varphi_\eta''}{\varphi_\eta} + (q-2)\frac{(\varphi_\eta')^2}{\varphi_\eta^2} = - f^2_\eta \,q ,&\\
        \varphi_\eta(0) = 1 \text{ and } \varphi'_\eta(0) = 0.&
    \end{cases}
\end{equation}
Let $(-T_\eta,T_\eta)$ denote the maximal interval on which $\varphi_\eta >0$. Note that $\lim_{\eta\to 0^+} T_\eta = \frac{\pi}{\alpha q}$. Proceed as in \cite[Prop.~5.2]{cz24}. For any $\delta>0$, take a finite covering $\widetilde{\bbt}^{q-1}$ of ${\bbt}^{q-1}$ and a smooth $\delta$-Lipschitz map $f_\delta\colon \widetilde{\bbt}^{q-1}\to \bbs^{q-1}$ of degree one. For sufficiently small $\epsilon>0$, take a smooth map $h_\epsilon\colon[-T_\eta +\epsilon, T_\eta-\epsilon]\to \bbs^1$ that is constant outside $[-2\eta,2\eta]$ and of degree one. For $q-1\le j\le m-1$, take smooth maps $\phi_j\colon \bbs^{\,j}\times\bbs^1\to\bbs^{\,j+1}$ of degree one, which factor through the corresponding smash products, such that the smooth map $\Phi_{\delta,\epsilon}\colon M\to \bbs^m$ defined by
\begin{align}
    \begin{split}
        \Phi_{\delta,\epsilon}\colon M \coloneqq \widetilde{\bbt}^{q-1}\times &[-T_\eta +\epsilon\,,\, T_\eta-\epsilon] \times \bbt^{m-q} \\
        &\xrightarrow{f_\delta\times h_\epsilon\times \id} \bbs^{q-1} \times \bbs^1 \times \bbt^{m-q} \xrightarrow{\phi_{m-1}\circ\dots\circ\phi_{q-1}} \bbs^{m}
    \end{split}
\end{align} 
is constant near boundary and of degree one. The maps $f_\delta, h_\epsilon,\phi_j$ can be chosen so that $\supp(\di\Phi_{\delta,\epsilon}) = \widetilde{\bbt}^{q-1}\times [-2\eta,2\eta] \times \bbt^{m-q}$.
Consider the integrable subbundle $F\subseteq TM$ and the metric $g^F$ defined by
\begin{equation}
F= T(\widetilde{\bbt}^{q-1}\times [-T_\eta +\epsilon, T_\eta-\epsilon]),\quad g^F_{\eta} = \varphi_\eta^2g_{\widetilde{\bbt}^{q-1}}\oplus \di t^2,
\end{equation} 
where $g_{\widetilde{\bbt}^{q-1}}$ is a flat metric on $\widetilde{\bbt}^{q-1}$. Then 
\begin{equation}
    k^{F}_{\eta} = f^2_\eta\, q(q-1)
    \quad\text{and}\quad
    \dist_F(\supp(\di\Phi_{\delta,\epsilon}),\partial M) = T_\eta -\epsilon -2\eta.
\end{equation}
Taking sufficiently small $\eta,\delta,\epsilon>0$, we obtain the optimality of \eqref{1.5} since
\begin{equation}
    \lim_{\delta\to 0^+}\sup_M|\wedge^2\!(\di\Phi_{\delta,\epsilon}|_F)\,|_{\tr} = 0.
\end{equation}

\end{rem}

For the next theorem, we temporarily assume the integrable subbundle $F$ is transverse to $\partial M$. In this case, we can talk about a leaf $\call'$ of $(M,F)$, which is a smooth, connected, immersed submanifold with smooth boundary $\partial \call' = \call' \cap \partial M$; see \cite[Ex.~1.2.21]{cc00}. Now, the leafwise scalar curvature $k^F\in C^\infty(M)$ is well-defined on $M$. Let $\{e_j\}_{j=1}^q$ be a local orthonormal frame of $(F,g^F)$ such that $e_1|_{\partial M}=\nu$ is the inward unit normal associated with $g^F$. Recall that the \emph{leafwise boundary mean curvature} $H^F\in C^\infty(\partial M)$ is defined as the mean curvature of the leaves at the boundary; that is, on a leaf $\call'$ of $(M,F)$,
\begin{equation}\label{h0}
    H^F|_{\partial \call'} = \frac{1}{q-1}\sum_{j=2}^q\inn{\nabla^{F}_{e_j}e_j,\nu},
\end{equation}
where $\nabla^{F}$ denotes the Levi-Civita connection on $\call'$ associated with $g^F|_{\call'}$. 

We also partially generalize \cite[Thm.~1.4]{cz24} and \cite[Thm.~1.5]{shi25-1} to foliations:
\begin{thm}\label{t1.2-0}
Let $M$ be an $m$-dimensional, compact, connected, oriented manifold with $\partial M\ne \varnothing$. Let $F\subseteq TM$ be an integrable subbundle of rank $q\ge 2$ that is transverse to $\partial M$. Assume either $TM$ or $F$ is spin. Let $g^F$ be a smooth metric on $F$. Suppose $\Phi\colon M\to\bbs^m$ is a smooth map and $L\subseteq M$ is a compact subset such that
\begin{enumerate}[label=\textup{(\roman*)},ref=\textup{\roman*}]
    \item $L\subseteq \supp(\di\Phi) \subseteq \sat_M(L)$ and $\dist_F(\supp(\di\Phi),\partial M) > l>0$;\label{t1.20i}
    \item $k^F|_L > 2\,|\wedge^2\!(\di\Phi|_F)\,|_{\tr}$ and $k^F|_{\supp(\di\Phi)\setminus L} \ge 2\,|\wedge^2\!(\di\Phi|_F)\,|_{\tr}$;\label{t1.20ii}
    \item $k^F|_{M \setminus \supp(\di\Phi)} \geq \alpha^2 q(q-1)$ and $H^F|_{\partial M} \geq -\alpha\tan\ket{\frac{\alpha q l}{2}}$\label{t1.20iii}
\end{enumerate}
for some $\alpha>0$ and $l\in (0,\frac{\pi}{\alpha q})$. Then $\deg\Phi = 0$.
\end{thm}

Similar to \cite[Cor.~1.5]{cz24}, \cref{t1.2-0} implies \cref{rem1}\,(1) since $ - \alpha\tan\ket{\frac{\alpha q l}{2}}$ tends to $-\infty$ as $l\to {\frac{\pi}{\alpha q}}$.

\subsection{A quantitative version of a theorem of Su-Wang-Zhang} In \cite{swz22}, Su, Wang and Zhang proved the following theorem which resolves Gromov's sharp foliated $\otimes_\varepsilon$-twisting conjecture. 

\begin{thm}[{\cite[Thm.~1.2]{swz22}}]\label{t1.3}
Let $(M,g)$ be an $m$-dimensional, noncompact, connected, oriented, complete Riemannian manifold without boundary. Let $F\subseteq TM$ be a smooth integrable subbundle of rank $q \ge 2$. Define $g^F = g|_F$. Let $\Phi: M\to \bbs^m$ be a smooth map that is area non-increasing along $F$, locally constant near infinity and of nonzero degree. If either $TM$ or $F$ is spin and 
$$k^F>q(q-1)\quad {\rm on}\ {\rm supp}({\rm d}\Phi),$$
then we have $\inf_M k^F<0$.
\end{thm}

We will prove the following quantitative version of \cref{t1.3}, which is also a foliated version of \cite[Thm.~1.9]{shi25-1}; cf.\ \cite[Thm.~3.1]{z20}. As in \cite[Prop.~5.5]{z20}, \cref{t1.2} is a corollary of the following \cref{t1.4}.

\begin{thm}\label{t1.4}
Let $(M,g)$ be an $m$-dimensional, noncompact, connected, oriented, complete Riemannian manifold without boundary. Let $F\subseteq TM$ be a smooth integrable subbundle of rank $q \ge 2$. Define $g^F = g|_F$. Assume either $TM$ or $F$ is spin. Let $\Phi \colon M\to\bbs^m$ be a smooth map. Suppose $W\subseteq M$ is a codimension-$0$, compact, embedded  submanifold with boundary and $L\subseteq M$ is a compact subset such that 
\begin{enumerate}[label=\textup{(\roman*)},ref=\textup{\roman*}]
    \item $L\subseteq \supp(\di\Phi) \subseteq W\cap \sat(L)$ and $\deg\Phi \ne 0$;\label{t1.4i}
    \item $k^F|_L > 2\,|\wedge^2\!(\di\Phi|_F)\,|_{\tr}$ and $k^F|_{\supp(\di\Phi)\setminus L} \ge 2\,|\wedge^2\!(\di\Phi|_F)\,|_{\tr}$;\label{t1.4ii}
    \item $k^F|_{W \setminus \supp(\di\Phi)} \geq \alpha^2 q(q-1)$ and $\dist_F(\supp(\di\Phi),\partial W) \ge l$\label{t1.4iii}
\end{enumerate}
    for some $\alpha>0$ and $l\in (0,\frac{\pi}{\alpha q})$. Then
    \begin{equation}
        \inf_{M} k^F < -\alpha^2 q(q-1)\tan^2\!\bbket{\frac{\alpha q l }{2}}.
    \end{equation}
\end{thm}

Then \cref{t1.4} implies the following theorem that generalizes \cref{t1.3}. 

\begin{thm}\label{t1.5}
    Let $(M,g)$ be an $m$-dimensional, noncompact, connected, oriented, complete Riemannian manifold without boundary. Let $F\subseteq TM$ be an integrable subbundle of rank $q\ge 2$ equipped with $g^F = g|_F$. Assume either $TM$ or $F$ is spin. Let $\Phi\colon M\to\bbs^m$ be a smooth map that is locally constant at infinity and of nonzero degree. Suppose $L$ is a compact subset such that
    \begin{enumerate}[label=\textup{(\roman*)},ref=\textup{\roman*}]
        \item $\partial(\supp(\di\Phi))\subseteq L \subseteq \supp(\di\Phi)\subseteq \sat(L)$;\label{t1.8i}
        \item $k^F|_L > 2\,|\wedge^2\!(\di\Phi|_F)\,|_{\tr}$ and $k^F|_{\supp(\di\Phi)\setminus L} \ge 2\,|\wedge^2\!(\di\Phi|_F)\,|_{\tr}$.\label{t1.8ii}
    \end{enumerate}
    Then we have $\inf_{M} k^F < 0$.
\end{thm}
For example, we can take $L= \supp(\di\Phi)$ if $k^F|_{\supp(\di\Phi)} > 2\,|\wedge^2\!(\di\Phi|_F)\,|_{\tr}$. We can take $L= \partial(\supp(\di\Phi))$ if $k^F|_{\supp(\di\Phi)} \ge 2\,|\wedge^2\!(\di\Phi|_F)\,|_{\tr}$, $k^F|_{\partial(\supp(\di\Phi))} > 0$ and $\supp(\di\Phi) \subseteq \sat\ket{\partial(\supp(\di\Phi))}$.

\subsection{A Llarull-type theorem on foliations}

In \cite{su}, Su partially generalizes  Llarull's theorem \cite{ll98} to foliations. For more on Llarull's theorem, see, e.g., \cite{cwxz26,g23,lsw,li10,wwxz26, wx25} and the references therein. 

Using the techniques from the proof of \cref{t1.2}, we have the following theorem which, in some sense, proves the rigidity part of the Llarull's theorem on closed foliated manifolds in the spin cases.  

\begin{thm}\label{t1.7}
    Let $(M,F)$ be a closed, oriented, foliated manifold. Suppose $m=\dim M$ and $q=\rk F \ge 3$. Assume either $TM$ or $F$ is spin. Let $g^F$ be a smooth metric on $F$ such that $k^F\ge q(q-1)$. If there exists a smooth map $\Phi\colon M\to \bbs^m$ of nonzero degree that is area non-increasing along $F$, then there exists a leaf $\call$ associated with $F$ such that $(\Phi|_\call)^*g_0 = g^F|_\call$, where $g_0$ denotes the standard metric on the unit sphere $\bbs^m$.
\end{thm}

The rest of this paper is organized as follows. Following \cite[\S5]{connes} (cf.\ \cite[\S2.1]{zhang17} and \cite[\S5.1]{chen}), we first recall the constructions of the Connes fibration in \cref{2.1}. Inspired by \cite[Lem.~2.7--2.8]{wx25}, we prove a Poincaré-type inequality for foliations in \cref{2.2}, which is essential when we only have the condition $k^F|_{\supp(\di\Phi)} \ge 2\,|\wedge^2\!(\di\Phi|_F)\,|_{\tr}$. Then we study closed, foliated manifolds and prove \cref{t1.7} in \cref{3} using deformed sub-Dirac operators \cite{lz}.

For better estimates in the long neck problems on foliated manifolds, we need the Lichnerowicz-type formula with coefficient $\frac{q}{q-1}$ for almost isometric foliations developed in \cite[Thm.~4.11--4.12]{chen}. Moreover, we need to refine the Poincaré-type inequality for foliations using the Penrose operator along the leaves (see \cref{4.1}). Then we prove \cref{t1.2,t1.4,t1.5} in \cref{4.2,4.3,4.4}. 

Finally, by considering the local boundary value problems as in \cite[\S5]{cz24} (cf.\ \cite{bb12,bb16}), we prove \cref{t1.2-0} in \cref{5} using leafwise boundary mean curvature bounds. The following diagram shows the dependency of results in this paper.
\begin{figure}[htbp]
    \centering
    \includegraphics[width=0.9\textwidth]{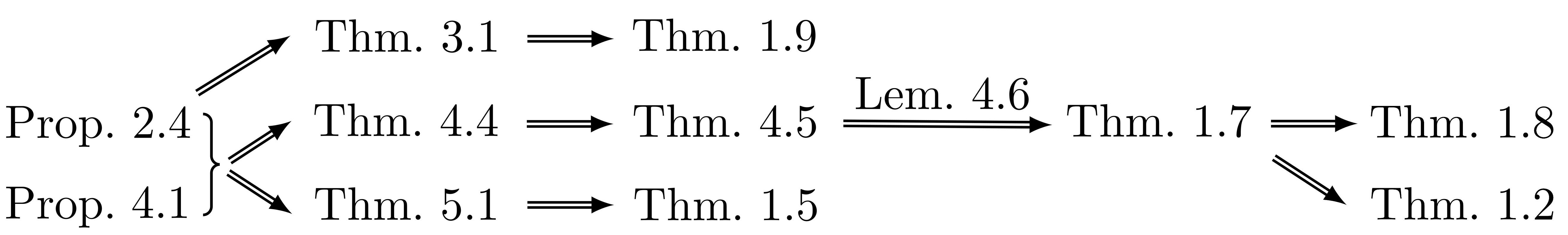}
\end{figure}

\section{Connes Fibration and Poincaré-type Inequality}\label{2}

\subsection{Connes Fibration}\label{2.1}

Throughout this paper, all manifolds and foliations are smooth. 
Let $(M,F)$ be a foliated manifold without boundary such that $\rk F<\dim M$. We assume $\dim M\ge 2$ and $\rk F \ge 1$ in this paper. Let $g$ be a complete metric on $M$, and define $g^F = g|_F$. Following \cite[\S5]{connes} (cf.\ \cite[\S2.1]{zhang17} \cite[\S5.1]{chen}), let
\begin{equation}\label{cf}
    \bar{\pi}\colon \calm\to M
\end{equation}
be the Connes fibration, whose fiber $\calm_x=\bar{\pi}^{-1}(x)$ at $x\in M$ is a product of four copies of the space of Euclidean metrics on the vector space $T_xM/{{F}}_x$. 

The vertical tangent bundle $T^V \calm$ is spin and of even rank since it is a direct sum of four isomorphic subbundles. Recall that there is a natural metric $g^{T^V\calm}$ on $T^V\calm$ such that any two points $v_1,v_2\in \calm_x$ $(x\in M)$ can be joined by a unique minimizing geodesic with respect to $g^{T^V\calm}|_{\calm_x}$ in the fiber $\calm_x$.

By using the Bott connection on $TM/F$, which is leafwise flat, we obtain a horizontal tangent bundle $T^H \calm\subseteq T\calm$. Fix a splitting
\begin{equation}\label{split}
    TM =F \oplus F^\perp.
\end{equation}
Then $F^\perp$ is isomorphic to $TM/F$, and \eqref{split} lifts to a splitting $T^H \calm={\mathcal{F}}\oplus {\mathcal{F}}^\perp_1$. Here $\calf$ is isomorphic to $\bar{\pi}^*F$, and it is also integrable since the Bott connection is leafwise flat. Let $g^\calf=\bar{\pi}^* g^F$ be the pullback metric on $\calf$. 

For any $v\in \calm$, the fiber $\mathcal{F}^\perp_{1,v}$ is identified with $F^\perp_{\bar{\pi}(v)}$ under the projection \eqref{cf}. By definition, each point $v\in \calm$ is a quadruple of Euclidean metrics on $F^\perp_{\bar{\pi}(v)}$. The first metric in this quadruple then determines a metric on $\calf^\perp_{1,v}$. In this way, the subbundle $\calf^\perp_1$ carries a canonically induced smooth metric $g^{\calf^\perp_1}$.

From now on, set $\calf^\perp_2=T^V\calm$. Set $q=\rk F$, $m=\dim M$ and $n=\dim\calm$. Let $g^{T\calm}$ be the Riemannian metric on $\calm$ defined by the orthogonal splitting 
\begin{equation}\label{met1}
T\calm=\calf\oplus \f\oplus \ff,\quad g^{T\calm}=g^\calf\oplus g^\f\oplus g^\ff.
\end{equation}
By \cite[Lem.~5.2]{connes}, the foliation $(\calm,\calf)$ is almost isometric with respect to \eqref{met1} in the sense of \cite[\S3]{connes}. That is, the action of the holonomy groupoid of $(\calm,\calf)$ on $\f\oplus\ff$ takes the form
\begin{equation}\label{almostiso}
    \begin{pmatrix}
        O_1 &0\\
        Q & O_2
    \end{pmatrix},
\end{equation}
where $O_1$ and $O_2$ are orthogonal matrices of ranks $m-q$ and $n-m$, respectively, and $Q$ is an $(n-m)\times (m-q)$ matrix.

Define $g^{F^\perp} = g|_{F^\perp}$. By definition, the quadruple $(g^{F^\perp},g^{F^\perp},g^{F^\perp},g^{F^\perp})$ determines an embedding $\iota\colon M\hookrightarrow \calm$. For any $v\in \calm\setminus \iota(M)$, recall there exists a unique geodesic from $v$ to $\iota(\bar{\pi}(v))$ in $\calm_{\bar{\pi}(v)}$. Let $\sigma(v)\in \ff|_{v}$ be the unit vector tangent to this geodesic at the point $v$ (pointing to the direction of $\iota(\bar{\pi}(v))$). Let 
\begin{align}
\rho(v)=\dist_{\calm_{\bar{\pi}(v)}}\bket{v,\iota(\bar{\pi}(v))}
\end{align}
denote
the length of this geodesic and define $\rho|_{\iota(M)} = 0$. Then $\sigma$ is a smooth vector field on $\calm\setminus \iota(M)$, and the function $\rho$ is smooth on $\calm\setminus \iota({M})$. 

Following \cite{zhang17,chen}, for any number $r>0$, we define
\begin{align}\label{ms}
\ms=\calm_r'=\{v\in \calm: \rho(v)\leq r\}\quad \text{and}\quad \mm = \ms \cup_{\rho^{-1}(r)} \calm_r'.
\end{align}
Here $\ms$ is a smooth manifold with boundary $\rho^{-1}(r)$, and $\mm$ is the double of $\ms$. Using a foliated collar neighborhood around $\rho^{-1}(r)$ with respect to the integrable subbundle $T^V\calm|_{\calm_{r}}$, the projection \eqref{cf} extends to a proper smooth submersion 
\begin{equation}\label{mm}
\pp\colon\mm\to M,
\end{equation}
which is a smooth (locally trivial) fibration by \cite{er51}. Note that the fibers of $\mm$ are the double of the corresponding fibers of $\ms$. View $\ms$ as a subset of $\mm$, and then we obtain an embedding $\iota\colon M\to\mm$.

The connection on $\ms$ can be extended to a connection on $\mm$, which determines a horizontal subbundle $T^H\mm\subseteq T\mm$. Then $TM =F \oplus F^\perp$ lifts to a splitting $T^H\mm = \hfz\oplus \hf$. Let $\widetilde{\mathcal{F}}^\perp_2$ denote the vertical tangent bundle $T^V \mm$. Let $g^{\widetilde{\mathcal{F}}} = \widetilde{\pi}^*_r g^{{F}}$ be the pullback metric on $\widetilde{\mathcal{F}}$. Take an arbitrary metric $g^{\widetilde{\mathcal{F}}^\perp_2}$ on $\mm$ that extends $g^{{\mathcal{F}}^\perp_2}|_\ms$. Moreover, we can take a metric $g^{\widetilde{\mathcal{F}}^\perp_1}$ which extends $g^{{\mathcal{F}}^\perp_1}|_\ms$ and satisfies\footnote{This is possible because \eqref{control} holds on $\ms$ by the construction of $\iota\colon M\to \calm$.}
\begin{equation}\label{control}
    g^{\widetilde{\mathcal{F}}^\perp_1} \ge C_r \cdot \pp^*g^{F^\perp}\quad\text{on }\mm
\end{equation}
for some $C_r>0$. We obtain a metric $g^{T\mm}$ defined by the orthogonal splitting
\begin{align}\label{met2}
T\mm=\widetilde{\mathcal{F}}\oplus \widetilde{\mathcal{F}}^\perp_1\oplus \widetilde{\mathcal{F}}^\perp_2, \quad g^{T\mm}=g^{\widetilde{\mathcal{F}}}\oplus g^{\widetilde{\mathcal{F}}^\perp_1}\oplus g^{\widetilde{\mathcal{F}}^\perp_2}. 
\end{align}
By construction, we have $g^{T\mm}|_{\ms} = g^{T\calm}|_{\ms}$. 

Moreover, the projection $\pp\colon (\mm,g^{T\mm}) \to (M,g)$ is Lipschitz due to \eqref{control} and $g^{\widetilde{\mathcal{F}}} = \widetilde{\pi}^*_r g^{{F}}$. Then $g^{T\mm}$ is complete since $g$ is complete and $\pp$ is a proper map.

Let $f_0\colon[0,1]\to [0,1]$ be a smooth function such that $f_0\equiv 0$ on $[0,\frac{1}{4}]$, while $f_0\equiv 1$ on $[\frac{1}{2},1]$. The function $f_0(\frac{\rho}{r})\colon\ms\to [0,1]$ extends to a smooth function
\begin{equation}
    f\colon\mm\to [0,1]
\end{equation}
such that $f\equiv 1$ on $\mm\setminus \calm_{\frac{r}{2}}$. 

Let $p_2^\perp\colon T\calm\to {\mathcal{F}}^\perp_2$ be the orthogonal projection. For the Levi-Civita connection $\nabla^{T\calm}$ of $g^{T\calm}$ on $\calm$, the Euclidean connection 
\begin{equation}
    \nabla^{{\mathcal{F}}^\perp_2}\coloneqq p^\perp_2 \nabla^{T\calm}p^\perp_2
\end{equation} 
on ${\mathcal{F}}^\perp_2$ is independent of $g^\calf$ and $g^\f$. Since ${\mathcal{F}}^\perp_2$ is spin and of even rank, we have a $\mathbb{Z}_2$-graded Hermitian (dual) bundle of spinors
\begin{equation}
S({\mathcal{F}}^\perp_2)^*=S_+({\mathcal{F}}^\perp_2)^*\oplus S_-({\mathcal{F}}^\perp_2)^*
\end{equation}
equipped with a Hermitian connection $\nabla^{S_{\pm}({\mathcal{F}}^\perp_2)^*}$ induced by $\nabla^{{\mathcal{F}}^\perp_2}$. There is a smooth bundle homomorphism $\widehat{c}\colon{\mathcal{F}}^\perp_2\to {\rm End}(S({\mathcal{F}}^\perp_2)^*)$, called Clifford action, such that $\widehat{c}(X)$ is self-adjoint and $\widehat{c}(X)^2 = |X|^2$ for all $X\in\ff$.

Following \cite[\S1b]{bz93} (cf.\ \cite[\S2.3]{zhang17}, \cite[\S3]{zhang23}), we take a complex vector bundle $\mathcal{Y}$ over $\calm$ such that $S_-({\mathcal{F}}^\perp_2)^*\oplus \mathcal{Y}$ is a trivial vector bundle over $\calm$. Then $S_+(\ff)^*\oplus \mathcal{Y}$ is a trivial vector bundle on $\calm\setminus \iota(M)$, under the identification $\widehat{c}(\sigma)\oplus\id$. By obviously extending the above trivial vector bundles from $\ms$ to $\mm$, we construct a $\mathbb{Z}_2$-graded Hermitian vector bundle
\begin{equation}\label{xi}
    \xi=\xi_+\oplus \xi_-
\end{equation}
over $\mm$ and an odd self-adjoint endomorphism $V^\xi\in \Gamma({\rm End}(\xi))$ such that 
\begin{align}
\xi_\pm &=S_\pm({\mathcal{F}}^\perp_2)^*\oplus \mathcal{Y}\quad\text{on } \ms,\\
V^\xi&=f\,\widehat{c}(\sigma)\oplus \id \quad \text{on } \ms\label{vxi}
\end{align}
and $(V^\xi)^2 = \id$ on $\mm\setminus \calm_{\frac{r}{2}}$; cf. \cite[\S2.3]{zhang17}. Take a Hermitian connection $\nabla^\caly$ on $\caly$, and then take an even Hermitian connection $\nabla^\xi = \nabla^{\xi_+}\oplus\nabla^{\xi_-}$ on $\xi\to \mm$ such that 
\begin{equation}\label{xicon}
    \nabla^{\xi_\pm} = \nabla^{S_\pm(\ff)^*}\oplus \nabla^\caly \quad\text{on }\ms.
\end{equation}

\begin{rem}\label{nonfoliated}
    When $F=TM$, we adopt the convention that $\calm = \mm = M$ and $\xi = M\times\bbc$ is the trivial line bundle. Meanwhile, we will take $V^\xi = 0 \in \Gamma(\en{\xi})$ and $\rho\equiv 0$.
\end{rem}

\subsection{Poincaré Inequality on Foliations}\label{2.2} 
Inspired by Gromov's proof of the isoperimetric inequality for Carnot--Carathéodory manifolds \cite[\S2.3--2.4]{g96b} (cf.\ \cite[Prop.~18]{pansu16}) and the Poincaré-type inequality on manifolds \cite[Lem.~2.7--2.8]{wx25}, we will prove a Poincaré-type inequality on foliations. 

The notation $|\cdot|$ will be used for the fiberwise norm on a vector bundle and the fiberwise operator norm of a bundle map, sometimes with subscripts indicating the relevant metrics. Let $\inter Q$ denote the interior of a topological set $Q$.

\begin{prop}\label{lem1}
    Let $(M,g)$ be a Riemannian manifold without boundary, equipped with a smooth integrable subbundle $F\subseteq TM$ of rank $q$. Suppose $K\subseteq M$ is a compact subset and $U\subseteq M$ is an open subset such that
    \begin{equation}\label{sat}
        K\subseteq \sat(U).
    \end{equation} 
    Given $a_0>0$, there exist a compact subset $L_0\subseteq M$ and a uniform constant $C>0$ such that for any Hermitian vector bundle $E\to M$, any $u\in\Gamma(E)$, any Hermitian connection $\nabla$ on $E$, and any smooth bundle map $A\colon E\to F^*\otimes E$ satisfying
    \begin{equation}\label{acond1}
        \sup_{x\in L_0}\bac{\sup_{X \in F_x,\;|X|_g=1}|A_X|} \leq a_0,
    \end{equation}
    the following inequality holds:
    \begin{equation}\label{result1}
        \int_{K} |u|^2 \le C\bbbac{\sum_{j=1}^q\int_{L_0} \abs{(\nabla_{e_j}+ A_{e_j})u}^2 + \int_{U} |u|^2},
    \end{equation}
    where $\{e_j\}_{j=1}^q$ is any local orthonormal frame of $(F,g|_F)$ and the integrations are taken with respect to the Riemannian measure induced by $g$.
\end{prop}

\begin{rem}
    Define $g^F = g|_F$. Indeed, for compact subset $K\subseteq M$ and open subset $U\subseteq M$ such that $K\subseteq \sat(U)$, we have $\sup_{x\in K}\bac{\dist_F(x,U)} < \infty$. 
\end{rem}

\begin{proof}[Proof of \cref{lem1}]
    For any $x\in K$, by \eqref{sat}, there exist $y\in U$ and an embedded smooth curve $\gamma\colon [0,1] \to M$ in a leaf such that 
    \begin{equation}\label{eq1}
        \gamma(0) = x, \quad \gamma(1) = y
        \quad\text{and}\quad
        0\ne \dot{\gamma}(t)\in F_{\gamma(t)}.
    \end{equation}
    Cover $\gamma([0,1])$ by finitely many foliation charts. Then we can extend $\dot{\gamma}$ to a smooth vector field $Y$ tangent to $F$, with compact support contained in an open neighborhood of $\gamma([0,1])$ in $M$. Let 
    \begin{equation}
        \Psi\colon M\times\bbr\to M
    \end{equation}
    denote the flow generated by $Y$. Then $\Psi(\cdot,t)$ is a diffeomorphism for each $t\in\bbr$. We have $\partial_t\Psi(z,t)= Y_{\Psi(z,t)}$ and $\Psi(\cdot,0) = \id$. Moreover, we can take a compact subset $Q\subseteq M$ such that
    \begin{equation}\label{eq3}
        x\in\inter Q
        \quad\text{and}\quad
        \Psi(Q\times\{1\})\subseteq U.
    \end{equation}
    
    Since $K$ is compact, we can take finitely many such flows $\Psi_i$ $(i=1,\dots,N)$ and such compact subsets $Q_i$ so that
    \begin{equation}\label{domain1}
        K\subseteq \cup_iQ_i
        \quad\text{and}\quad
        \Psi_i(Q_i\times\{1\}) \subseteq U.
    \end{equation} 
    Since $Q_i$ is compact, the number $a_1>0$ defined as follows is finite:
    \begin{equation}\label{req}
        a_1 = \max_i \bbac{\max_{z\in Q_i,\;t\in [0,1]} |\partial_t\Psi_i(z,t)|_g}<\infty.
    \end{equation}

    Let $J_i(z,t)$ denote the Jacobian of the diffeomorphism $\Psi_i(\cdot,t)$ at the point $z$. Note that there exists a constant $a_2>0$ such that
\begin{equation}\label{esti2}
    J_i(z,t) \ge a_2
\end{equation}
for all $i\in \{1,\dots,N\}$ and all $(z,t)\in Q_i\times[0,1]$.
    
    For $z\in Q_i$, define
\begin{equation}\label{gammaz}
    \gamma_{i,z}(t)=\Psi_i(z,t)
    \quad\text{and}\quad
    \dot{\gamma}_{i,z}(t) = \partial_t\Psi_i(z,t)\in F_{\gamma_{i,z}(t)}.
\end{equation}
For any $u\in \Gamma(E)$, let $|u|^2_{\gamma_{i,z}(t)}$ denote the value of $|u|^2$ at the point $\gamma_{i,z}(t)$. Then 
\begin{align}\label{diff}
    \begin{split}
        \partial_t\bket{\abs{u}^2_{\gamma_{i,z}(t)}}
        &= 2\re\binn{u,{\nabla}_{\dot{\gamma}_{i,z}(t)}u}\\
        &= 2\re\binn{u,\bket{\nabla_{\dot{\gamma}_{i,z}(t)}+A_{\dot{\gamma}_{i,z}(t)}}u} - 2\re\binn{u,A_{\dot{\gamma}_{i,z}(t)}u}\\
        &\ge -(2a_0a_1 + 1)|u|^2_{\gamma_{i,z}(t)} - \babs{\bket{\nabla_{\dot{\gamma}_{i,z}(t)}+A_{\dot{\gamma}_{i,z}(t)}}u}^2
    \end{split}
\end{align}
by \eqref{acond1} and \eqref{req}. By Gronwall's inequality and \eqref{diff}, we have
\begin{align}\label{uz}
    \begin{split}
        |u|_{z}^2 &\le e^{2a_0a_1+1}\bbbac{|u|_{\gamma_{i,z}(1)}^2 + \int_{0}^{1}\babs{\bket{\nabla_{\dot{\gamma}_{i,z}(t)}+A_{\dot{\gamma}_{i,z}(t)}}u}^2\,\di t}\\
    &\le e^{2a_0a_1+1}\bbbac{|u|_{\gamma_{i,z}(1)}^2 + a_1^2\sum_{j=1}^q\int_{0}^{1}\babs{\bket{\nabla_{e_j}+A_{e_j}}u}^2\,\di t}
    \end{split}
\end{align}
for any $z\in Q_i$ and any local orthonormal frame $\{e_j\}_{j=1}^q$ of $(F,g^F)$. We define 
\begin{equation}\label{defl}
    L_0 = \cup_i \bac{\Psi_i(Q_i\times[0,1])}.
\end{equation}

Integrate \eqref{uz} over $Q_i$ and then sum over $i=1,\dots,N$. By \eqref{esti2}, we have
\begin{equation}\label{esti3}
    \int_K|u|^2 \le \sum_{i=1}^N\int_{Q_i} |u|^2
    \le 
    \frac{N e^{(2a_0a_1+1)}}{a_2}\bbbac{\int_{U}|u|^2 + a_1^2\sum_{j=1}^{q}\int_{L_0}\babs{\bket{\nabla_{e_j}+A_{e_j}}u}^2}.
\end{equation}
So we obtain \eqref{result1} for $C = N e^{2a_0a_1+1}a^{-1}_2\max\{1,a_1^2\}$.
\end{proof}

Recall the Connes fibration $\bar{\pi}\colon \calm\to M$ in \eqref{cf}. Let $g^F$ be a smooth metric on $F$. Following \cite[(1.11)]{zhang17}, we rescale the metric $g^{T\calm}$ defined in \eqref{met1} as follows:
\begin{align}\label{met3}
    g_{\beta,\varepsilon}^{T\calm}=\beta^2 g^{\calf}\oplus \frac{g^{\f}}{ \varepsilon^2}\oplus g^{\ff}\quad (0<\beta,\varepsilon\leq 1).
\end{align}

\begin{prop}\label{lem2}
    Let $M$ be a smooth manifold without boundary. Let $F\subseteq TM$ be a smooth integrable subbundle equipped with a smooth metric $g^F$. Suppose $K\subseteq M$ is a compact subset and $U\subseteq  M$ is an open subset such that 
    \begin{equation}
        K\subseteq \sat(U).
    \end{equation}  
    Given $a_0>0$, there exist a compact subset $L_0\subseteq  M$ and a uniform constant $C>0$ such that for any Hermitian vector bundle $E\to \calm$, any $u\in \Gamma(E)$, any Hermitian connection $\nabla$ on $E$, any $\bv>0$, and any smooth bundle map $A\colon E\to \calf^*\otimes E$ with
    \begin{equation}\label{acond2}
        \sup_{v\in \bar{\pi}^{-1}(L_0)}\bbac{\sup_{X \in \calf_v,\;|X|_\bv=1}|A_X|} \leq \frac{a_0}{\beta},
    \end{equation}
    the following inequality holds:
    \begin{equation}\label{result2}
            \int_{\bar{\pi}^{-1}(K)}\frac{|u|^2}{\beta^2}
            \le C\bbbac{\sum_{j=1}^q\int_{\bar{\pi}^{-1}(L_0)} \babs{\bket{\nabla_{e_j}+A_{e_j}}u}^2 + \int_{\bar{\pi}^{-1}(U)} \frac{|u|^2}{\beta^2}},
    \end{equation}
    where $\{e_j\}_{j=1}^q$ is any local orthonormal frame of $(\calf,\beta^2g^\calf)$ and the integrations are taken with respect to the Riemannian measure induced by the metric $g^{T\calm}_\bv$ in \eqref{met3}.
\end{prop}
\begin{proof}
    As in \eqref{eq1}--\eqref{eq3}, we obtain finitely many flows $\Psi_i\colon M\times \bbr\to M$ ($i=1,\cdots, N$) generated by the vector fields $Y_i\in\Gamma(F)$ with compact support. Moreover, we obtain compact subsets $Q_i\subseteq M$ such that \eqref{domain1} holds.    
    
    Recall the horizontal subbundle $T^H\calm$ determined by the Bott connection in \cref{2.1}. The horizontal lifting of the vector field $Y_i$ generates a horizontal flow
    \begin{equation}
        \widetilde{\Psi}_i\colon \calm\times \bbr\to\calm.
    \end{equation} 
    Define $\widetilde{Q}_i = \bar{\pi}^{-1}(Q_i)$, and define $a_1$ by \eqref{req}. Then we have
    \begin{equation}\label{req2}
        \max_i \bbac{\max_{z\in \widetilde{Q}_i,\;t\in [0,1]} |\partial_t\widetilde{\Psi}_i(z,t)|_\bv} = \beta a_1<\infty.
    \end{equation}

    Define $\widetilde{\gamma}_{i,z}(t) = \widetilde{\Psi}_i(z,t)$. As in \eqref{uz}, it follows from \eqref{acond2} and \eqref{req2} that
    \begin{equation}\label{uz1}
        |u|_{z}^2 
        \le e^{2a_0a_1+1}\bbbac{|u|_{\widetilde{\gamma}_{i,z}(1)}^2 + \beta^2a_1^2\sum_{j=1}^q\int_{0}^{1}\babs{\bket{\nabla_{e_j}+A_{e_j}}u}^2\,\di t}
    \end{equation}
    for any $z\in\widetilde{Q}_i$ and any local orthonormal frame $\{e_j\}_{j=1}^q$ of $(\calf,\beta^2g^\calf)$.

    Let $\widetilde{J}_i(z,t)$ denote the Jacobian of the diffeomorphism $\widetilde{\Psi}_i(\cdot,t)$ at the point $z$. Recall that $(\calm,\calf)$ is almost isometric with respect to \eqref{met3}. By \eqref{almostiso}, there exists a constant $a_2>0$ independent of $\bv$ such that
    \begin{equation}\label{esti22}
        \widetilde{J}_i(z,t) \ge a_2
    \end{equation}
    for all $i\in \{1,\dots,N\}$ and all $(z,t)\in \widetilde{Q}_i\times[0,1]$. As in \eqref{esti3}, we have
\begin{equation}\label{esti32}
    \int_{\bar{\pi}^{-1}(K)}|u|^2 \le 
    \frac{N e^{(2a_0a_1+1)}}{a_2}\bbbac{\int_{\bar{\pi}^{-1}(U)}|u|^2 + \beta^2a_1^2\sum_{j=1}^{q}\int_{\bar{\pi}^{-1}(L_0)}\babs{\bket{\nabla_{e_j}+A_{e_j}}u}^2}.
\end{equation}
for the compact subset $L_0\subseteq M$ defined in \eqref{defl}. Then \eqref{result2} also holds for $C = N e^{2a_0a_1+1}a^{-1}_2\max\{1,a_1^2\}$.
\end{proof}
\begin{rem}\label{rem2}
    For every positive integer $k$, \cref{lem2} still holds if we replace $\calm$ by $\calm\times \bbt^k$, where we take $g^{T(\calm\times \bbt^k)}_\bv = g^{T\calm}_\bv \oplus \epsilon_1^2\di\theta_1^2 \oplus \cdots \epsilon_k^2\di\theta_k^2$ for $\theta_i\in\bbs^1 = \bbr/2\pi\bbz$ and $\epsilon_i>0$ $(i=1,\dots,k)$.
\end{rem}

\section{Llarull-type Theorem on Foliations}\label{3}
As an application of \cref{lem2}, we prove a leafwise scalar curvature extremality theorem on the closed foliated manifold $(M,F)$ when $TM$ or $F$ is spin; see \cref{main1}. As a corollary, we obtain the Llarull-type \cref{t1.7} on foliations. 

\subsection{Leafwise Scalar Curvature Extremality}\label{3.1}
Let $(M,F)$ be a foliated manifold. Assume either $TM$ or $F$ is spin. We construct an auxiliary even-dimensional manifold $M^\lozenge$ as follows. When $TM$ is spin, define
\begin{equation}\label{tmspin}
    M^\lozenge =
    \begin{cases}
        M,& \text{if } \dim M \text{ is even},\\
        M \times \bbs^1,& \text{if }\dim M \text{ is odd}.
    \end{cases}
\end{equation}
When $TM$ is nonspin but $F$ is spin, define
\begin{equation}\label{fspin}
    M^\lozenge =
    \begin{cases}
        M,& \text{if } \rk F \text{ and } \dim M\text{ are even},\\
        M \times \bbs^1\times \bbs^1,& \text{if } \rk F \text{ is odd and } \dim M\text{ is even},\\
        M\times\bbs^1,& \text{otherwise}.
    \end{cases}
\end{equation}

Let $F$ and $TM$ still denote their pullbacks to $M^\lozenge$. Let $g^F$ be a smooth metric on $F$. Set $q = \rk F$. Let $E\to M^\lozenge$ be a Hermitian vector bundle equipped with a Hermitian connection $\nabla^E$. Let $R^E = (\nabla^E)^2$ denote its curvature.

When $TM$ is nonspin but $F$ is spin, let $S(F) \to M^\lozenge$ be the Hermitian bundle of spinors on $(F,g^F)$ that is equipped with the Clifford action $c\colon F\to \en{S(F)}$. Define a self-adjoint bundle endomorphism $\calr^{E}_q\in \Gamma(\en{S(F) \otimes E})$ on $M^\lozenge$ by
\begin{equation}\label{end1}
    \calr^{E}_q = \frac{1}{2}\sum_{i,j=1}^{q} c(\we_i) \, c(\we_j) \otimes R^E(\we_i,\we_j)
\end{equation}
for any local orthonormal frame $\{\we_j\}_{j=1}^q$ of $(F,g^F)$. 

When $TM$ is spin, we extend $g^F$ to an arbitrary metric $g^{TM}$ on $TM\subseteq TM^\lozenge$. Let $S(TM)\to M^\lozenge$ be the Hermitian bundle of spinors on $(TM,g^{TM})$ that is equipped with the Clifford action $c\colon TM\to\en{S(TM)}$. Then \eqref{end1} defines a self-adjoint bundle endomorphism $\calr^{E}_q\in \Gamma(\en{S(TM) \otimes E})$ on $M^\lozenge$ instead. Let $|\calr^{E}_q|$ denote the pointwise operator norm of $\calr^{E}_q$.

\begin{thm}\label{main1}
Let $M$ be a closed, connected, oriented manifold equipped with an integrable subbundle $F\subseteq TM$ of rank $q \ge 2$. Assume either $TM$ or $F$ is spin. Let $g^F$ be a smooth metric on $F$. If there exists a $\bbz_2$-graded Hermitian vector bundle $E = E_+\oplus E_-$ over $M^\lozenge$ equipped with an even Hermitian connection $\nabla^E$ such that
\begin{equation}\label{cond1}
    k^F \ge 4|\calr^{E}_q|
    \text{ on } M^\lozenge
    \quad\text{and}\quad
    \binn{\widehat{A}\ket{TM^\lozenge}\ch (E_+ - E_-),[M^\lozenge]} \ne 0,
\end{equation}
then there exists a leaf $\call\subseteq M^\lozenge$ associated with $F$ such that $k^F = 4|\calr^{E}_q|$ on $\call$.
\end{thm}

\begin{rem}
    The vanishing theorems of Connes \cite{connes} and Zhang \cite{zhang17} generalize Lichnerowicz theorem to foliations: if a closed, oriented, foliated manifold $(M,F)$ admits a metric of positive leafwise scalar curvature and either $F$ or $TM$ is spin, then $\widehat{A}(M)= 0$. Applying \cref{main1} to $M\times M\times M\times M$ and $E = \bbc$, we can show that the vanishing theorems of Connes and Zhang still hold when $k^F\ge 0$ and there exists a point $x_\mathcal{L}$ on each leaf $\call$ of $(M,F)$ such that $k^F(x_\mathcal{L})>0$.
\end{rem}

\begin{rem}
    Using similar techniques as in the proof of \cref{main1} (cf.\ \cite[Thm.~1.8]{swz22}), we can also prove the following result. Let $(M,F)$ be a closed, oriented, foliated manifold such that $M$ is area enlargeable along $F$ (see \cite[Def.~1.3--1.4]{swz22}). Assume either $TM$ or $F$ is spin. If $g^F$ is a smooth metric on $F$ with $k^F\ge 0$ on $M$, then there exists a leaf $\call$ associated with $F$ such that $k^F\equiv 0$ on $\call$.
\end{rem}

\begin{proof}[Proof of \cref{main1}]
    Suppose to the contrary that for every leaf $\call\subseteq M^\lozenge$ associated with $F$, there exists a point $x_{\call}\in\call$ such that $k^F - 4|\calr^E_q| \ge 2a_\call$ at $x_\call$ for some $a_\call>0$. Then there exists an open neighborhood $U_\call\subseteq M^\lozenge$ around $x_\call$ such that
    \begin{equation}
        k^F - 4|\calr^E_q| \ge a_\call>0 \quad \text{on }U_\call.
    \end{equation} 
    
    As a subset of $M^\lozenge$, the saturation $\sat_{M^\lozenge}(U_\call)$ of $U_\call$ is defined to be the union of all leaves on $(M^\lozenge,F)$ that intersect $U_\call$. Note that $\sat_{M^\lozenge}(U_\call)\subseteq M^\lozenge$ is open since $U_\call$ is. Since $M^\lozenge$ is compact, there exist finitely many leaves $\{\call_i\}_{i=1}^N$ such that
    \begin{equation}
        M^\lozenge = \cup_i\sat_{M^\lozenge}(U_{\call_i}).
    \end{equation}
    Define $U = \cup_i \, U_{\call_i}$ and $a = \min_{i}\{a_{\call_i}\}$. Then we have
    \begin{equation}\label{alpha}
        M^\lozenge = \sat_{M^\lozenge}(U)
        \quad\text{and}\quad
        k^F - 4|\calr^E_q| \ge a > 0
        \text{ on } U.
    \end{equation}
    
    Next, we work case by case to deduce a contradiction.

\subsection{Proof of Theorem \ref{main1} when \texorpdfstring{$TM$}{TM} is Spin}\label{3.2}

Suppose $TM$ is spin. We can assume $\dim M$ is even without loss of generality. In this case, we have $M^\lozenge = M$. The manifold $\mm$ defined in \eqref{ms} is closed since $M$ is. Following \cite{zhang17}, we rescale the metric $g^{T\mm}$ defined in \eqref{met2} as follows:
\begin{align}\label{met4}
    g_{\beta,\varepsilon}^{T\mm}=\beta^2 g^{\widetilde{\mathcal{F}}}\oplus \frac{g^{\widetilde{\mathcal{F}}^\perp_1}}{ \varepsilon^2}\oplus g^{\widetilde{\mathcal{F}}^\perp_2}\quad (0<\beta,\varepsilon\leq 1).
\end{align}

Since $TM$ and $\widetilde{\mathcal{F}}^\perp_2$ are spin and of even rank, so is  $T\mm$. Let
\begin{equation}
    S_{\beta,\varepsilon}(T\mm)=S_{\beta,\varepsilon}^+(T\mm)\oplus S^-_{\beta,\varepsilon}(T\mm)
\end{equation} 
be the $\mathbb{Z}_2$-graded Hermitian bundle of spinors associated with $(\mm,g_{\beta,\varepsilon}^{T\mm})$. Let $c_\bv\colon T\mm\to \en{S_{\beta,\varepsilon}(T\mm)}$ denote the Clifford action.

Let $\ee = \ee_+\oplus\ee_-$ denote the pullback Hermitian vector bundle over $\mm$ defined by $ \ee_\pm = \pp^*E_\pm$. Define a $\bbz_2$-graded Hermitian vector bundle $\cals_\bv = \cals_\bv^+ \oplus \cals_\bv^-$ by
\begin{align}\label{bun1}
\cals_\bv = S_{\beta,\varepsilon}(T\mm)\,\widehat{\otimes}\,\xi \,\widehat{\otimes}\, \ee
\end{align}
over $\mm$, where $\xi$ is defined in \eqref{xi}. Let $c_\bv$ still denote the induced Clifford action $c_\bv \,\widehat{\otimes}\, \id \otimes \id$ on $\cals_\bv$. Let $\nabla^\bv$ denote the tensor product Hermitian connection on $\cals_\bv$, and let $D_\bv$ denote the twisted Dirac operator on $\cals_\bv$. 

Recall the bundle endomorphism $V^\xi\in\Gamma(\en{\xi})$ defined around \eqref{vxi}. We define an odd self-adjoint bundle endomorphism 
\begin{equation}\label{end2}
    V = \id
    \,\widehat{\otimes}\, V^\xi \,\widehat{\otimes}\, \id \in \Gamma(\en{\cals_\bv}).
\end{equation}
Then $V$ anticommutes with $c_\bv$. 

For $0<\delta\le 1$, let $h_{\delta,0}\colon[0,1]\to [\delta,1]$ be a smooth function such that $h_{\delta,0}\equiv 1$ on $[0,\frac{3}{4}]$, while $h_{\delta,0}\equiv \delta$ on $[\frac{7}{8},1]$. The function $h_{\delta,0}(\frac{\rho}{r})\colon\ms\to [0,1]$ extends to a globally defined smooth function 
\begin{equation}\label{hdelta}
    h_\delta\colon\mm\to [0,1]
\end{equation}
such that $h_\delta\equiv \delta$ on $\mm\setminus \calm_{\frac{7r}{8}}$. 

Following \cite[(2.21)]{zhang17}, for $\delta,\bv>0$, we define the deformed Dirac operator
\begin{align}\label{dd1}
    \calb_{\delta,\bv}=h_\delta \,D_\bv\, h_\delta + \beta^{-1}V\colon \Gamma(\cals_\bv)\to \Gamma(\cals_\bv)
\end{align}
on the closed manifold $\mm$. For simplicity, we write $\calb = \calb_{\delta,\bv}$. Let $\calb_+\colon\Gamma(\cals_\bv^+)\to \Gamma(\cals_\bv^-)$ denote the natural restriction. As in \cite[(2.34)]{zhang17}, it follows from the Atiyah--Singer index theorem \cite{as68} (cf. \cite[III.~Thm.~13.8]{lm89}) and \eqref{cond1} that
\begin{align}\label{contradict}
    \begin{split}
        \ind\calb_+ &= \binn{\widehat{A}(T\mm)\ch(\xi_+-\xi_-)\ch(\widetilde{\cale}_+-\widetilde{\cale}_-),[\mm]}\\
        &=\binn{\widehat{A}(TM)\ch (E_+-E_-),[M]}\ne 0.
    \end{split}
\end{align}

Following \cite[p.~115]{bl91}, let $\varphi_0:[0,1]\to [0,1]$ be a smooth function such that ${\varphi}_0\equiv 0$ on $[0,\frac{1}{2}]$ and ${\varphi}_0\equiv 1$ on $[\frac{3}{4},1]$. Define $\varphi_1,\varphi_2:\mm\to [0,1]$ by
\begin{equation}\label{pou2}
        \varphi_{1} = \frac{1 - \varphi_0\bket{\frac{\rho}{r}}}{\bbket{\varphi_0\bket{\frac{\rho}{r}}^2 + \bket{1- \varphi_0\bket{\frac{\rho}{r}}}^2}^{\frac{1}{2}}},
        \quad
        \varphi_{2} = \frac{\varphi_0\bket{\frac{\rho}{r}}}{\bbket{\varphi_0\bket{\frac{\rho}{r}}^2 + \bket{1- \varphi_0\bket{\frac{\rho}{r}}}^2}^{\frac{1}{2}}},
    \end{equation}
on $\ms$, and require that they are constant outside $\calm_{\frac{3r}{4}}$.

For simplicity, let $\|\cdot\|$ denote the $\mathrm{L}^2$-norm on $\Gamma(\cals_\bv)$ with respect to $g^{T\mm}_\bv$. Throughout this paper, the differential $\di \varphi$ of a smooth function $\varphi$ on a Riemannian manifold will be identified with the gradient of $\varphi$. For $u\in \Gamma(\cals_\bv)$, we have
\begin{equation}
    \calb(\varphi_{b} u) = \varphi_{b} \calb u + h^2_\delta\, c_\bv({\rm d}\varphi_{b})u,\quad b=1,2.
\end{equation}
Note that  $\varphi_{1}^2 + \varphi_{2}^2 = 1$ on $\mm$. Then it follows that
\begin{align}\label{spe1}
\begin{split}
    \|\calb u\|^2
    &= \|\varphi_{1} \calb u\|^2+ \|\varphi_{2} \calb u\|^2\\
    &\geq 
    \frac{1}{2}\|\calb(\varphi_{1} u)\|^2+ \frac{1}{2}\|\calb(\varphi_{2}u)\|^2-\|h^2_\delta\, c_\bv({\rm d}\varphi_{1})u\|^2 -\|h^2_\delta\,c_\bv({\rm d}\varphi_{2})u\|^2.
\end{split}
\end{align}
In this paper, the bracket $[\cdot,\cdot]$ is superbracket for $\bbz_2$-graded vector bundles. For example, we have $[D_\bv,V] = D_\bv V + V D_\bv$. By \cite[Lem.~2.1]{zhang17}, there exists a constant $C_1>0$ independent of $\bvr>0$ so that
\begin{equation}\label{spe2}
|c_\bv({\rm d}\varphi_b)|\leq {\frac{C_1}{\beta r}}+O_{r}(1) \quad\text{and}\quad
\abs{[D_\bv,V]}\leq {\frac{C_1}{\beta r}}+O_{r}(1)
\end{equation}
on $\ms$; cf.\ \cite[(5.22)--(5.23)]{chen}. Here, the bounding constants in the $O_r(1)$ terms may depend on $r$. Recall $\supp(\varphi_2u)\subseteq \mm\setminus\calm_{\frac{r}{2}}$. By \eqref{spe2}, we have
\begin{align}\label{spe4}
    \begin{split}
        \|\calb(\varphi_2 u)\|^2
        &= \int_\mm\babs{h_\delta D_\bv (h_\delta\varphi_2 u)}^2+ \int_{\mm}\bbbinn{\varphi_2 u, \bbbket{\frac{V^2}{\beta^2} +\frac{h^2_\delta}{\beta}[D_\bv,V]}\varphi_2 u}\\
        &\geq \bbbac{\frac{1}{\beta^2}+O\bbbket{\frac{1}{\beta^2 r}}+O_r\bbbket{\frac{1}{\beta}}}\int_{\ms}|\varphi_2 u|^2\\
        &\qquad\qquad+\bbbac{\frac{1}{\beta^2}+\delta^2 \, O_{\beta,\varepsilon,r}(1)}\int_{\mm \setminus \ms}|\varphi_2 u|^2. 
    \end{split}
\end{align}

Let $\cale$ denote the pullback Hermitian vector bundle $\bar{\pi}^*E \to \calm$. We define the Hermitian vector bundles $\cals_1$ and $\cals_2$ over $\calm$ by 
\begin{align}
    \cals_1 &=S_\bv(T\calm)\,\widehat{\otimes}\,S(\ff)^* \,\widehat{\otimes}\, \cale\label{cals1},\\
    \cals_2 &=S_\bv(T\calm)\,\widehat{\otimes}\,(\caly \oplus \caly) \,\widehat{\otimes}\, \cale.\label{cals2}
\end{align}
Then $\calb$ preserves the orthogonal splitting $\cals_\bv|_{\ms} = \cals_1|_{\ms} \oplus \cals_2|_{\ms}$. Suppose
\begin{equation}
    \varphi_1 u = u_1 + u_2,
\end{equation}
where $u_b\in \Gamma(\cals_b|_{\ms})$ for $b=1,2$. Note that $h_\delta\equiv 1$ on ${\rm supp}(u_2)$. By \eqref{vxi}--\eqref{xicon}, we have $[D_\bv,V]u_2=0$ and $V^2u_2=u_2$. It follows that
\begin{align}\label{spe5}
\|\calb u_2\|^2=\int_{\ms}\abs{D_\bv\, u_2}^2+\beta^{-2}|u_2|^2\geq \int_{\ms}\beta^{-2}|u_2|^2.
\end{align}

Set $m=\dim M$ and $n=\dim\calm$. Let $\{\we_j\}_{j=1}^n$ be a local orthonormal frame on $(\calm,g^{T\calm})$ such that $\{\we_j\}_{j=1}^q$, $\{\we_j\}_{j=q+1}^m$ and $\{\we_j\}_{j=m+1}^n$ span $\calf,\f$ and $\ff$, respectively. We define
\begin{equation}\label{scale}
    e_j =
    \begin{cases}
        \beta^{-1}\we_j, & \text{if } 1\le j \le q,\\
        \varepsilon\we_j, & \text{if } q+1\le j \le m,\\
        \we_j, & \text{if } m+1\le j \le n.
    \end{cases}
\end{equation}
Then $\{e_j\}_{j=1}^n$ is a local orthonormal frame on $(\calm,g^{T\calm}_\bv)$. 

Note that ${\rm supp}(u_1)\subseteq \mathcal{M}_{\frac{3r}{4}}$. Taking $A = 0$ and $E = \cals_\bv$ with the Hermitian connection $\nabla^\bv$ in \cref{lem2}, we obtain a constant $C>0$ independent of $\bvr$ such that for the open subset $U$ and the constant $a>0$ defined before \eqref{alpha},
\begin{equation}\label{keyesti}
    \int_{\calm_r} \sum_{j=1}^{q}|\nabla^\bv_{e_j} u_1|^2 + \frac{a}{4\beta^2}\int_{\calm_r\,\cap\,\bar{\pi}^{-1}(U)}|u_1|^2 \ge \frac{C}{\beta^2}\int_{\calm_r}|u_1|^2.
\end{equation}
Define the bundle endomorphisms $\calr^\bv\in\Gamma(\en{\cals_\bv})$ and $\calr^{\cale,\bv}_q\in\Gamma(\en{\cals_1})$ by
\begin{align}
    \calr^\bv &= \frac{1}{2}\sum_{i,j=1}^n c_\bv(e_i)\,c_\bv(e_j)\,R^{\cals_\bv}(e_i,e_j),\label{calr}\\
    \calr^{\cale,\bv}_q &= \frac{1}{2}\sum_{i,j=1}^q c_\bv(e_i)\,c_\bv(e_j)\,\widehat{\otimes}\,\id \,\widehat{\otimes}\, R^{\cale}(e_i,e_j).\label{calre}
\end{align}
Recall that $\ee|_\ms = \cale|_\ms$. Since $(\ms,\calf)$ is almost isometric with respect to \eqref{met4}, by \eqref{cals1}, the Lichnerowicz formula, \cite[Prop.~1.4]{zhang17} and \cite[(1.45)--(1.46)]{zhang17}, we have 
\begin{align}
    \begin{split}\label{spe6}
        \int_\ms|D_\bv u_1|^2 &= \int_\ms\sum_{j=1}^{n}|\nabla^\bv_{e_j} u_1|^2 + \int_\ms\binn{u_1,\calr^\bv u_1}\\
        &= \int_\ms\sum_{j=1}^{n}|\nabla^\bv_{e_j} u_1|^2 + \int_\ms \bbbinn{u_1, \frac{\ket{k^\calf + 4\beta^2\calr^{\cale,\bv}_q}u_1}{4\beta^2}}\\
        &\qquad\qquad+ \int_\ms  O_r\bbbket{\frac{1}{\beta}+ \frac{\varepsilon^2}{\beta^2}}|u_1|^2;
    \end{split}
\end{align}
cf.\ \cite[(1.71)]{zhang17}. Recall that $\calf\oplus\f = \bar{\pi}^*TM$ and $\cale=\bar{\pi}^*E$. Note that
\begin{equation}
    S_\bv(T\calm) = S_\bv(\calf\oplus\f)\,\widehat{\otimes}\,S(\ff)
    \quad\text{and}\quad
    \beta^2|\calr^{\mathcal{E},\bv}_q| = \bar{\pi}^*|\mathcal{R}^E_q|.
\end{equation} 
Then it follows from \eqref{cond1}, \eqref{alpha} and $k^\calf = \bar{\pi}^*k^F$ that 
\begin{equation}\label{kf5}
    k^\calf - 4\beta^2|\calr^{\cale,\bv}_q| \ge 0 \text{ on } \calm
    \quad\text{and}\quad
    k^\calf - 4\beta^2|\calr^{\cale,\bv}_q| \ge a>0 \text{ on } \bar{\pi}^{-1}(U).
\end{equation}
By \eqref{spe2}, \eqref{keyesti}, \eqref{spe6},  \eqref{kf5} and $h_\delta \equiv 1$ on $\supp(u_1)$, we have
\begin{align}\label{spe7}
    \begin{split}
        \|\calb u_1\|^2&=\int_{\ms}|D_\bv \,u_1|^2 + \int_{\ms}\bbbinn{u_1,\bbbket{\frac{V^2}{\beta^2}+\frac{1}{\beta}[D_\bv,V]}u_1}\\
        &\ge \int_\ms\sum_{j=1}^{n}|\nabla^\bv_{e_j} u_1|^2 + \frac{a}{4\beta^2}\int_{\ms\,\cap\,\bar{\pi}^{-1}(U)}|u_1|^2 \\
        &\qquad\qquad+ \int_{\ms} \bbbac{O\bbbket{\frac{1}{\beta^2r}} + O_{r}\bbbket{\frac{1}{\beta} + \frac{\varepsilon^2}{\beta^2}} }|u_1|^2\\
        &\ge \int_{\ms}\bbbac{\frac{C}{\beta^2} + O\bbbket{\frac{1}{\beta^2r}} + O_{r}\bbbket{\frac{1}{\beta} + \frac{\varepsilon^2}{\beta^2}} } |u_1|^2.\\
    \end{split}
\end{align}

By \eqref{spe1}--\eqref{spe4}, \eqref{spe5}, \eqref{spe7}, $h_\delta\le 1$ and $\varphi_{1}^2 + \varphi_{2}^2 = 1$, there exist sufficiently large $r>0$, sufficiently small $\beta =\varepsilon>0$, $\delta>0$ and $C_0>0$ such that
\begin{align}\label{spe8}
\|\calb u\|\geq C_0\|u\| 
\end{align}
holds for all $u\in \Gamma(\cals_\bv)$. It follows that $\ind\calb_+=0$, which contradicts \eqref{contradict}.

\subsection{Proof of Theorem \ref{main1} when \texorpdfstring{$F$}{F} is Spin}\label{3.3}

In what follows, we assume $TM$ is nonspin but $F$ is spin. Then $\rk F < \dim M$.

\textbf{Case 1: $\rk F$ and $\dim M$ are even.}

In this case, we have $M^\lozenge = M$. We follow \cite[\S2.5]{zhang17} closely. Let $\widetilde{p}\colon T\mm \to \hfz$, $\widetilde{p}_1^\perp\colon T\mm\to \widetilde{\mathcal{F}}^\perp_1$ and $\widetilde{p}_2^\perp\colon T\mm\to \widetilde{\mathcal{F}}^\perp_2$ be the orthogonal projections. Let $\nabla^{T\mm,\bv}$ denote the Levi-Civita connection of $g^{T\mm}_\bv$. We define the Euclidean connections
\begin{equation}
    \nabla^{\hfz,\bv}= \widetilde{p}\, \nabla^{T\mm,\bv} \widetilde{p},
    \;\;
    \nabla^{\hf,\bv}=  \widetilde{p}^\perp_1\nabla^{T\mm,\bv}\widetilde{p}^\perp_1,
    \;\;
    \nabla^{\hff}=  \widetilde{p}^\perp_2 \nabla^{T\mm,\bv} \widetilde{p}^\perp_2
\end{equation}
on $\hfz,\hf,\hff$ respectively. Here $\nabla^{\hff}$ is independent of $\bv$ since $\hff$ is integrable. 

The subbundle $\hfz = \pp^*F$ is spin since $F$ is. Let 
\begin{equation}
    S_\bv(\hfz) = S^+_\bv(\hfz) \oplus S^-_\bv(\hfz)
    \quad \text{and}\quad 
    S(\hff) = S_+(\hff) \oplus S_-(\hff)
\end{equation} 
be the $\bbz_2$-graded Hermitian vector bundles of spinors on $(\hfz,\beta^2g^\hfz)$ and $(\hff,g^\hff)$ with Hermitian connections induced by $\nabla^{\hfz,\bv}$ and $\nabla^{\hff}$, respectively. For $X\in\hfz$ and $Y\in\hff$, the Clifford actions $c_\bv(X)\in\en{S_\bv(\hfz)}$ and $c(Y)\in\en{S(\hff)}$ are odd  endomorphisms.

Let $\Lambda^*(\hf)$ be the complexified exterior algebra bundle of the dual bundle $\widetilde{\calf}^{\perp,*}_1$, which carries a canonical Hermitian metric induced by $g^\f$. Let $\Lambda^*_\bv(\hf)$ denote the bundle $\Lambda^*(\hf)$ equipped with the Hermitian connection induced by $\nabla^{\hf,\bv}$. For $X\in \hf$, the Clifford actions $c_\bv,\widehat{c}_\bv\colon \hf\to \en{\Lambda_\bv^*(\hf)}$ are given by
\begin{equation}
    c_\bv(X) =\varepsilon^{-1} \bac{(X^*\wedge) - i_X}, 
    \quad 
    \widehat{c}_\bv(X) = \varepsilon^{-1} \bac{(X^*\wedge) + i_X}, 
\end{equation}
where $X^*\wedge$ and $i_X$ denote the exterior multiplication by $X^*$ and the interior multiplication by $X$ on $\Lambda_\bv^*(\hf)$, respectively. In this case, $\rk F^\perp = \rk \hf$ are even. Then there is a $\bbz_2$-grading on $\Lambda_\bv^*(\hf)$ determined by
\begin{equation}
    \bket{\sqrt{-1}}^{\frac{m-q}{2}}c_\bv(e_{q+1}) \cdots c_\bv(e_m),
\end{equation}
where $\{e_j\}_{j=q+1}^m$ is an orthonormal basis of $\hf$ with respect to \eqref{met4}.

Let $\phi$ be a polynomial with integer coefficients. Let $\phi(F^\perp)$ be the corresponding polynomial in the exterior and symmetric powers of $F^\perp$, and use the same notation for its complexification. Similarly, we have a $\bbz_2$-graded Hermitian vector bundle $\phi(\hf)$ with Hermitian metric induced by $g^\hf$. Let $\phi_\bv(\hf)$ denote the bundle $\phi(\hf)$ equipped with the Hermitian connection induced by $\nabla^{\hf,\bv}$. 

Instead of \eqref{bun1}, we consider
\begin{align}\label{bun2}
    \cals_\bv = S_{\beta,\varepsilon}(\widetilde{\mathcal{F}})
    \,\widehat{\otimes}\,
    \Lambda^*_{\beta,\varepsilon}(\widetilde{\mathcal{F}}^\perp_1)
    \,\widehat{\otimes}\,
    S(\widetilde{\mathcal{F}}^\perp_2)\,\widehat{\otimes }\,\xi  
     \,\widehat{\otimes}\, \phi_\bv(\hf) \,\widehat{\otimes}\, \ee,
\end{align}
as in \cite[\S2.5]{zhang17}, where $\xi$ is defined in \eqref{xi} and $\ee = \pp^*E$.

The bundle endomorphism $V^\xi$ defined around  \eqref{vxi} extends to $V\in\Gamma(\en{\cals_\bv})$ as in \eqref{end2}. Let $D_\bv$ denote the twisted sub-Dirac operator on $\cals_\bv$ defined as in \cite[\S2]{lz} and \cite[(1.61)]{zhang17}. Recall $\hf = \pp^*F^\perp$. We have $\phi(\hf) = \pp^*(\phi(F^\perp))$. For the self-adjoint Fredholm operator $\mathcal{B}_{\delta,\bv}$ defined as in \eqref{dd1}, we have (compare \eqref{contradict})
\begin{equation}\label{ind1}
    \ind\mathcal{B}_{\delta,\bv,+} = 2^{\frac{m-q}{2}}\binn{\widehat{A}({F})\,\widehat{L}({F}^\perp)\ch(\phi(F^\perp))\ch(E_+-E_-),[M]};
\end{equation}
cf.\ \cite[(2.59)]{zhang17}. By \cite[(1.45)--(1.46)]{zhang17}, we have $R^{\f,\bv}(e_i,e_j) = O_r(\beta^{-1})$ on $\calm_r$ for  $1\leq i,j\leq n$. Then \eqref{spe6} still holds. 

As in \eqref{spe8}, there exist sufficiently large $r>0$, sufficiently small $\beta=\varepsilon>0$ and sufficiently small $\delta>0$ such that
$\ind\mathcal{B}_{\delta,\bv,+} = 0$. 
By \eqref{ind1}, we have
\begin{equation}
    \binn{\widehat{A}({F})\,\widehat{L}({F}^\perp)\ch(\phi(F^\perp))\ch(E_+-E_-),[M]} = 0
\end{equation}
for any polynomial $\phi$ with integral coefficients. 
As any rational Pontryagin class $p(F^\perp)$ of $F^\perp$ can be expressed as a rational linear combination of classes of the form $\widehat{L}(F^\perp)\ch \ket{\phi(F^\perp)}$, we have
\begin{equation}
    \binn{\widehat{A}(F)\,p(F^\perp)\ch(E_+-E_-),[M]} = 0;
\end{equation}
cf.\ \cite[(2.61)]{zhang17}. Taking $p(F^\perp) = \widehat{A}(F^\perp)$, it follows from $\widehat{A}(TM)= \widehat{A}(F)\widehat{A}(F^\perp)$ that
\begin{equation}
    \binn{\widehat{A}(TM)\ch(E_+-E_-),[M]}=0,
\end{equation}
which contradicts \eqref{cond1}.

\textbf{Case 2: $\rk F$ is even and $\dim M$ is odd.}

In this case, we have $M^\lozenge = M\times \bbs^1$ and  $\mm^\lozenge = \mm\times\bbs^1$. Note that $\rk \hf$ is odd. For $0<\beta,\varepsilon\leq 1$, consider the following orthogonal splitting instead of \eqref{met4}:
\begin{equation}\label{met5}
T\mm^\lozenge=\widetilde{\mathcal{F}} \oplus \bket{\widetilde{\mathcal{F}}^\perp_1\oplus T\bbs^1}\oplus \widetilde{\mathcal{F}}^\perp_2,
\quad
g_{\beta,\varepsilon}^{T\mm^\lozenge}=\beta^2 g^{\widetilde{\mathcal{F}}}\oplus \frac{g^{\widetilde{\mathcal{F}}^\perp_1}\oplus{\di\theta^2}}{\varepsilon^2}\oplus g^{\widetilde{\mathcal{F}}^\perp_2},
\end{equation}
where $\theta\in \bbs^1\cong \bbr/2\pi\bbz$. Instead of \eqref{bun2}, we now consider
\begin{align}
S_\bv(\widetilde{\mathcal{F}})\,\widehat {\otimes }\,\Lambda^*_\bv(\widetilde{\mathcal{F}}_1^\perp\oplus T\bbs^1)\,\widehat{\otimes}\,S(\widetilde{\mathcal{F}}^\perp_2)\,\widehat{\otimes}\,\xi \,\widehat{\otimes}\, \phi_\bv(\hf \oplus T\bbs^1) \,\widehat{\otimes}\, \ee 
\end{align}
over $\mm^\lozenge$. Define  $\calm^\lozenge_{r} = \ms\times \bbs^1$. The foliated manifold $(\calm^\lozenge_r,\calf)$ is still almost isometric with respect to \eqref{met5}; cf.\ \cite[Ex.~4.6]{chen}. Proceeding as in \cref{3.1} and Case 1, we arrive at a contradiction.

\textbf{Case 3: $\rk F$ and $\dim M$ are odd.}

In this case, we have $M^\lozenge = M\times \bbs^1$ and $\mm^\lozenge = \mm\times\bbs^1$. Note that $\rk \hf$ is even. For $0<\beta,\varepsilon\leq 1$, consider the following orthogonal splitting instead of \eqref{met4}:
\begin{equation}\label{met6}
T\mm^\lozenge=\bket{\widetilde{\mathcal{F}}\oplus T\bbs^1} \oplus \widetilde{\mathcal{F}}^\perp_1\oplus \widetilde{\mathcal{F}}^\perp_2,
\quad
g_{\beta,\varepsilon}^{T\mm^\lozenge}=\bket{\beta^2 g^{\widetilde{\mathcal{F}}}\oplus{\di\theta^2}} \oplus \frac{g^{\widetilde{\mathcal{F}}^\perp_1}}{\varepsilon^2}\oplus g^{\widetilde{\mathcal{F}}^\perp_2},
\end{equation}
where $\theta\in \bbs^1\cong \bbr/2\pi\bbz$.
Instead of \eqref{bun2}, we now consider
\begin{align}
S_\bv(\widetilde{\mathcal{F}}\oplus T\bbs^1)\,\widehat {\otimes }\,\Lambda^*_\bv(\widetilde{\mathcal{F}}_1^\perp)\,\widehat{\otimes}\,S(\widetilde{\mathcal{F}}^\perp_2)\,\widehat{\otimes}\,\xi \,\widehat{\otimes}\, \phi_\bv(\hf) \,\widehat{\otimes}\, \ee 
\end{align}
over $\mm^\lozenge$. The foliated manifold $(\calm^\lozenge_r,\calf)$ is almost isometric with respect to
\begin{equation}
\hf \oplus (\hff\oplus T\bbs^1)
\quad\text{and}\quad
\frac{g^\hf}{\varepsilon^2}\oplus (g^\hff\oplus\di\theta^2),
\end{equation}
where $\calm^\lozenge_{r} = \ms\times \bbs^1$; cf.\ \cite[Ex.~4.6]{chen}. Proceeding as in \cref{3.1} and Case 1, we arrive at a contradiction; cf.\ \cite[\S6.2]{chen}.

\textbf{Case 4: $\rk F$ is odd and $\dim M$ is even.}

In this case, we have $M^\lozenge = M \times \bbs^1\times \bbs^1$ and  $\mm^\lozenge = \mm\times \bbs^1\times\bbs^1$. For $0<\beta,\varepsilon\leq 1$, we consider the following orthogonal splitting instead of \eqref{met4}:
\begin{align}
&T\mm^\lozenge=\bket{\widetilde{\mathcal{F}}\oplus T\bbs^1}\oplus \bket{\widetilde{\mathcal{F}}^\perp_1\oplus T\bbs^1}\oplus \widetilde{\mathcal{F}}^\perp_2\\
\text{and\;\;}
&g_{\beta,\varepsilon}^{T\mm^\lozenge}=\bket{\beta^2 g^{\widetilde{\mathcal{F}}}\oplus {\rm d}\theta^2_1}\oplus \frac{g^{\widetilde{\mathcal{F}}^\perp_1}\oplus{\di\theta^2_2}}{\varepsilon^2}\oplus g^{\widetilde{\mathcal{F}}^\perp_2},\label{met70}
\end{align}
where $\theta_1,\theta_2\in\bbs^1\cong \bbr/2\pi\bbz$. Instead of \eqref{bun2}, we now consider
\begin{align}
S_\bv(\widetilde{\mathcal{F}}\oplus T\bbs^1)\,\widehat {\otimes }\,\Lambda^*_\bv(\widetilde{\mathcal{F}}_1^\perp\oplus T\bbs^1)\,\widehat{\otimes}\,S(\widetilde{\mathcal{F}}^\perp_2)\,\widehat{\otimes}\,\xi  \,\widehat{\otimes}\, \phi_\bv(\hf \oplus T\bbs^1) \,\widehat{\otimes}\, \ee
\end{align}
over $\mm^\lozenge$. The foliated manifold $(\calm^\lozenge_r,\calf)$ is almost isometric with respect to
\begin{equation}\label{met71}
(\hf\oplus T\bbs^1) \oplus (\hff\oplus T\bbs^1)
\quad\text{and}\quad
\frac{g^\hf\oplus \di\theta^2_2}{\varepsilon^2}\oplus (g^\hff\oplus\di\theta^2_1),
\end{equation}
where $\calm^\lozenge_{r} = \ms\times \bbs^1\times\bbs^1$; cf.\ \cite[Ex.~4.6]{chen}. Proceeding as in \cref{3.1} and Case 1, we arrive at a contradiction; cf.\ \cite[\S6.2]{chen}. 

We complete the proof of \cref{main1}.
\end{proof}

\subsection{Proof of Theorem \ref{t1.7}}
Now we are ready to prove the Llarull-type \cref{t1.7} on foliations. 

\begin{thm}\label{main11}
    Let $(M,F)$ be a closed, oriented, foliated manifold. Suppose $m = \dim M$ and $q= \rk F \ge 2$. Assume either $TM$ is spin or $F$ is spin. Let $g^F$ be a smooth metric on $F$. Let $\Phi\colon M\to \bbs^m$ be a smooth map of nonzero degree such that $k^F\ge 2|\wedge^2\!(\di\Phi|_F)\,|_{\tr}$.  Then there exists a leaf $\call$ such that $k^F = 2|\wedge^2\!(\di\Phi|_F)\,|_{\tr}$ on $\call$.
\end{thm}
\begin{proof}
    There exists a connected component of $M$ where the restriction of $\Phi$ has nonzero degree. So we can assume $M$ is connected without loss of generality. 
    
    Recall the auxiliary manifold $M^\lozenge$ defined by \eqref{tmspin} or \eqref{fspin} depending on whether $TM$ is spin or $F$ is spin. Let
\begin{equation}
    \pi_\lozenge\colon M^\lozenge\to M
\end{equation}
denote the natural projection. The curvature term $\calr^E_q$ is defined by \eqref{end1}. Let $F$ still denote its pullback to $M^\lozenge$. Recall that there exists a $1$-Lipschitz smooth map $h\colon \bbs^m\times\bbs^1 \to \bbs^{m+1}$ of degree one\footnote{For our purpose, it is enough to require that $h$ is $1$-Lipschitz along the direction of $\bbs^m$.}. Then the map $\Phi\colon M\to \bbs^m$ can be extended to a smooth map $\Phi_\lozenge\colon M^\lozenge\to \bbs^{\dim M^\lozenge}$ such that
\begin{equation}\label{deg1}
    \deg\Phi_\lozenge = \deg\Phi.
\end{equation}

Recall the trace norm inequality $|AB|_{\tr} \leq |A|\cdot |B|_{\tr}$ for $A,B\in \en{V_0}$, where $V_0$ is a vector space and $|A|$ is the operator norm. Then we have on $M^\lozenge$ that
\begin{equation}\label{tr32}
    |\wedge^2\!(\di\Phi_\lozenge|_F)\,|_{\tr} \leq \pi_\lozenge^*\bket{|\wedge^2\!(\di\Phi|_F)\,|_{\tr}}.
\end{equation}

Let $S(T\bbs^{\dim M^\lozenge}) = S_+(T\bbs^{\dim M^\lozenge})\oplus S_-(T\bbs^{\dim M^\lozenge})$ be the Hermitian bundle of spinors over the standard even-dimensional sphere $\bbs^{\dim M^\lozenge}$. In \cref{main1}, we take $E = E_+ \oplus E_-$ to be the pullback Hermitian vector bundle
\begin{equation}\label{defe31}
E_\pm = \Phi_\lozenge^*\bket{S_\pm(T\bbs^{\dim M^\lozenge})}.
\end{equation}
By \cite[(4.6)]{ll98} and \eqref{tr32}, we have on $M^\lozenge$ that
\begin{equation}\label{cur1}
    2|\calr^E_q| \le |\wedge^2\!(\di\Phi_\lozenge|_F)\,|_{\tr}.
\end{equation}
By \eqref{tr32}, \eqref{cur1} and $k^F \ge 2|\wedge^2\!(\di\Phi|_F)\,|_{\tr}$ on $M$, we have on $M^\lozenge$ that
\begin{equation}\label{eq}
    k^F - 4|\calr^E_q| \ge k^F - 2\pi_\lozenge^*(|\wedge^2\!(\di\Phi|_F)\,|_{\tr}) \ge 0.
\end{equation} 
Moreover, by \eqref{deg1} and \eqref{defe31}, we have
\begin{align}
    \begin{split}
        \quad\binn{\widehat{A}(TM^\lozenge)&\ch(E_+-E_-),[M^\lozenge]}\\
        &= \binn{\ch\!\bket{S_+(T\bbs^{\dim M^\lozenge}) - S_-(T\bbs^{\dim M^\lozenge})},[\bbs^{\dim M^\lozenge}]} \cdot \deg\Phi_\lozenge\\
        &= 2(-1)^{\dim M^\lozenge /2}\,\deg\Phi\ne 0.
    \end{split}
\end{align}

Then \cref{main1} implies that there exists a leaf $\call\subseteq M^\lozenge$ associated with $F$ such that $k^F = 4|\calr^E_q|$ on $\call$. It follows from \eqref{eq} that $k^F = 2\pi_\lozenge^*(|\wedge^2\!(\di\Phi|_F)\,|_{\tr})$ on $\call$. The result follows since $\pi_\lozenge(\call)$ is a leaf of $(M,F)$.
\end{proof}

As a corollary, we have:
\begin{proof}[Proof of \cref{t1.7}]
    Now we have 
    \begin{equation}\label{kf1}
        k^F\ge q(q-1).
    \end{equation}
    Since $\Phi$ is area non-increasing along $F$, we have $\lambda_i\lambda_j \le 1$ ($i<j$) in \eqref{deflam}. So
    \begin{equation}\label{kf2}
        |\wedge^2\!(\di\Phi|_F)\,|_{\tr} = \sum_{1\le i<j\le q}\lambda_i\lambda_j \le \frac{q(q-1)}{2}.
    \end{equation}
    It follows from \eqref{kf1}--\eqref{kf2} that $k^F\ge 2|\wedge^2\!(\di\Phi|_F)\,|_{\tr}$ on $M$. 
    
    By \cref{main11}, there exists a leaf $\call\subseteq M$ associated with $F$ such that
    \begin{equation}\label{eq2}
        k^F = 2|\wedge^2\!(\di\Phi|_F)\,|_{\tr}\quad\text{on }\call.
    \end{equation}
    By \eqref{kf1}--\eqref{eq2}, we have $\lambda_i\lambda_j \equiv 1$ $(1\le i<j\le q)$ on $\call$. Since $q\ge 3$, we have
    \begin{equation}
        \lambda_j \equiv 1\quad \text{on } \call.
    \end{equation}
    So $(\Phi|_\call)^*g_0 = g^F|_\call$, where $g_0$ is the standard metric on the unit sphere $\bbs^m$.
\end{proof}

\section{Long Neck Foliated Problem}\label{4}

In this section, we first prove a strict leafwise distance version of \cref{t1.4}. Then \cref{t1.4} follows by using conformal deformations (see \cref{cfm}). As a corollary, we also obtain \cref{t1.2,t1.5}.
\subsection{Penrose Operator along the Leaves}\label{4.1}
    Under the settings of \cref{2.2}, recall the metric $g^{T\calm}_\bv$ defined in \eqref{met3}. Let $\{e_j\}_{j=1}^n$ be a local orthonormal frame on the Connes fibration $(\calm,g^{T\calm}_\bv)$ such that $\{e_j\}_{j=1}^q, \{e_j\}_{j=q+1}^m, \{e_j\}_{j=m+1}^n$ span the subbundles $\calf, \f, \ff$ respectively, where $q=\rk \calf$, $m=\rk(\calf\oplus \f)$ and $n=\dim \calm$. Suppose $q= \rk F\ge 2$.

    Let $\cals = \cals_+\oplus \cals_-$ be a $\bbz_2$-graded Hermitian vector bundle over $\calm$. Suppose there exists a family of Hermitian connections $\nabla^\bv$ on $\cals$ and a family of Clifford actions $c_\bv\colon T\calm \to \en{\cals}$ such that $(\cals,\nabla^\bv,c_\bv)$ is a $\bbz_2$-graded Clifford module\footnote{That is, (1) the splitting $\cals = \cals_+\oplus \cals-$ is orthogonal under the Hermitian metric on $\cals$; (2) the Hermitian connection $\nabla^\bv$ is even; (3) the Clifford action $c_\bv(X)$ is odd and skew-adjoint; (4)
    \begin{equation*}
        c_\bv(X)\,c_\bv(Y) + c_\bv(Y)\,c_\bv(X) = - 2g^{T\calm}_\bv(X,Y)
        \quad\text{and}\quad
        [\nabla^\bv_X,c_\bv(Y)] = c_\bv\bket{\nabla^{T\calm,\bv}_X Y},
    \end{equation*}
    where $X,Y\in\Gamma(T\calm)$ and $\nabla^{T\calm,\bv}$ is the Levi-Civita connection of $g^{T\calm}_\bv$.} on $(\calm,g^{T\calm}_\bv)$. Following \cite[(4.2)~and~(4.11)]{chen}, for $X\in\Gamma(\calf)$, we define the partial sum of Dirac operator and the Penrose operator along $\calf$, respectively, by
    \begin{equation}\label{dpdef}
        D_{\bv;q} = \sum_{j=1}^q c_\bv(e_j)\nabla^\bv_{e_j}
        \quad \text{and} \quad
        P^{\bv;q}_X = \nabla^\bv_X + \frac{1}{q}c_\bv(X)D_{\bv;q}.
    \end{equation}
    
    We add ``$\bv$" as a subscript or superscript to indicate geometric objects associated with the metric. The omission of ``$\bv$" corresponds to the case where $\beta = \varepsilon = 1$. We also add ``$q$" as a subscript or superscript to indicate geometric objects associated with the foliation.

    Let $A_\bv \in \Gamma(\en{\cals})$ be a family of bundle endomorphisms. 
    As in \cite[Lem.~4.3]{chen}, for any $u\in\Gamma(\cals)$, we observe that\footnote{The operator $P^{\bv;q}$ is also the Penrose operator along $\calf$ of the connection $\nabla^\bv - \frac{1}{q}c_\bv(\cdot) A_\bv$.}
    \begin{align}\label{penrose}
        \begin{split}
            \sum_{j=1}^q\babs{P^{\bv;q}_{e_j}u}^2 
            &= \sum_{j=1}^{q}
            \bbabs{\nabla^\bv_{e_j}u + \frac{1}{q}c_\bv(e_j)D_{\bv;q}u}^2\\
            &= \sum_{j=1}^{q}
            \bbabs{\bbket{\nabla^\bv_{e_j} - \frac{1}{q}c_\bv(e_j) A_\bv}u + \frac{1}{q}c_\bv(e_j)\bket{D_{\bv;q} + A_\bv}u}^2\\
            &= \sum_{j=1}^{q}\bbabs{\bbket{\nabla^\bv_{e_j} - \frac{1}{q}c_\bv(e_j) A_\bv}u}^2 - \frac{1}{q}\babs{(D_{\bv;q} + A_\bv) u}^2.
        \end{split}
    \end{align}

    Let $R^\bv = (\nabla^\bv)^2$ denote the curvature of the connection $\nabla^\bv$ on $\cals$. Recall the bundle endomorphism $\calr^\bv$ defined in \eqref{calr}. Similarly, we define
    \begin{equation}\label{calrq}
        \calr^\bv_q = \frac{1}{2}\sum_{i,j=1}^q c_\bv(e_i)\,c_\bv(e_j)\,R^\bv(e_i,e_j).
    \end{equation}
    
    If $\partial M\ne \varnothing$, further assume that the integrable subbundle $F$ is transverse to $\partial M$. We can require the chosen splitting \eqref{split} satisfies $F^\perp|_{\partial M} \subseteq T\partial M$. Then we have $\f\oplus\ff\subseteq T\partial\calm$. Let $\nu$ denote the inward unit normal  on $\partial\calm$ associated with the metric $g^{T\calm}_\bv$. Then $\nu\in\Gamma(\calf)$. Assume the chosen local orthonormal frame $\{e_j\}^{n}_{j=1}$ satisfies $e_1 = \nu$ on $\partial \calm$. 
    
    Following \cite[Lem.~2.1--2.2]{gilkey93}, we define $D^\partial_{\bv;a}\colon \Gamma(\cals|_{\partial \calm}) \to \Gamma(\cals|_{\partial \calm})$ by
    \begin{equation}
        D^\partial_{\bv;a} = \sum_{j=2}^{a}c_\bv(e_j)\,c_\bv(\nu) \bbket{\nabla^\bv_{e_j} + \frac{1}{2}c_\bv\bket{\nabla^{TM,\bv}_{e_j}\nu}\,c_\bv(\nu)}
    \end{equation}
    for $a\in\{q,n\}$; see \cite[(3.7)~and~(4.3)]{chen}. Set $D^{\partial}_\bv = D^\partial_{\bv;n}$. Recall that the (partial) mean curvature is defined by
    \begin{equation}\label{hdef}
        H^{T\calm}_{\bv;a}=-\frac{1}{a-1}\sum_{j=2}^{a}\inn{\nabla^{T\calm,\bv}_{e_j}\nu,e_j},\quad a\in\{q,n\}.
    \end{equation}

    In what follows, the integrations are taken with respect to the Riemannian measure $\di\mu_\bv$ induced by $g^{T\calm}_\bv$. We will omit $\di\mu_\bv$ for simplicity.
\begin{prop}\label{lem3}
    Let $(M,F)$ be a foliated manifold such that $F$ is transverse to $\partial M$, with $q =\rk F\ge 2$. Let $(\cals,\nabla^\bv,c_\bv)$ be a $\bbz_2$-graded Clifford module on the Connes fibration $(\calm,g^{T\calm}_\bv)$. Under above settings, further assume that $A_\bv\in\Gamma(\en{\cals})$ are odd and self-adjoint. Let $L_0\subseteq \calm$ be a compact subset. Then the following formula holds uniformly for all sufficiently small $\bv>0$ and all $u\in\Gamma_\cm(\cals)$ supported in $L_0$:
\begin{align}\label{qton}
    \begin{split}
        \int_\calm \bbac{&\babs{\bket{D_{\bv;q} + A_\bv} u}^2 - \babs{\bket{D_\bv + A_\bv} u}^2}\\
        &= \int_\calm \bbac{ O(\beta)\sum_{j=1}^{q} \babs{P^{\bv;q}_{e_j} u}^2- \sum_{j=q+1}^{n} \babs{\nabla^\bv_{e_j} u}^2} \\
        &\qquad\quad+ \int_\calm\bbinn{u,\bbket{\bket{1+O(\beta)}\calr_q^\bv-\calr^\bv + O\bket{\beta^{-1}}}u}\\
        &\qquad\quad + \int_\calm \bbac{\binn{u, \bcu{D_{\bv;q}-D_{\bv}\,, A_\bv}u} + O(1)|u||A_\bv u|}\\
        &\qquad\quad+ \int_{\partial \calm} \bbinn{u,\bbket{D^\partial_\bv - \bket{1+O(\beta)} D^\partial_{\bv;q} + O(1)}u}.        
    \end{split}
\end{align}
Here, the $O(\cdot)$ terms denote possibly complex-valued scalar quantities that may depend on $u$, with bounding constants independent of $u$.
\end{prop}
\begin{proof}
    Since $D_\bv$ and $A_\bv$ are odd, we have $[D_\bv,A_\bv] = D_\bv A_\bv + A_\bv D_\bv$. By the Green's formula (see, e.g., \cite[(3.24)]{chen}), we have
    \begin{align}\label{deform1}
        \begin{split}
            \int_\calm|(D_\bv+A_\bv)u|^2 &= \int_\calm \bbac{|D_\bv u|^2 + \binn{u,[D_\bv,A_\bv]u}+ |A_\bv u|^2}\\
            &\qquad\qquad + \int_{\partial\calm}\binn{u,c_\bv(\nu)A_\bv u}.
        \end{split}
    \end{align}
    By the Bochner--Weitzenböck formula (see, e.g., \cite[(3.14)]{chen}), we have 
    \begin{align}\label{deform2}
        \begin{split}
            \int_\calm |D_\bv u|^2
            &= \int_\calm \bbac{\sum_{j=1}^n\babs{\nabla^\bv_{e_j}u}^2 +\inn{u,\calr^\bv u}}\\
            &\qquad\qquad + \int_{\partial\calm}\bbinn{u,\bbket{\frac{n-1}{2}H^{T\calm}_\bv - D^\partial_\bv} u},
        \end{split}
    \end{align}

    Let $p\colon TM\to \calf$ and $p^\perp\colon TM\to \f\oplus\ff$ be the orthogonal projections. Recall that $\nu\in\Gamma(\calf|_{\partial\calm})$. As in \cite[(B.4)]{chen}, for $v_1,v_2\in\Gamma_\cm(S)$, we have
\begin{align}\label{greenq}
    \begin{split}
        \int_\calm\inn{D_{\bv;q} \,v_1,v_2} 
        &= \int_\calm\bbinn{v_1, c_\bv\bbket{\sum_{j=1}^{q}p^\perp \nabla^{T\calm,\bv}_{e_j}e_j - \sum_{k=q+1}^{n}p\nabla^{T\calm,\bv}_{e_k}e_k}v_2}\\
        &\qquad + \int_\calm\binn{v_1,D_{\bv;q}\,v_2} + \int_{\partial \calm} \binn{v_1,c_\bv(\nu)\,v_2}.
    \end{split}
\end{align}
It follows from \eqref{greenq} and \cite[Lem.~1.3]{zhang17} (cf.\ \cite[Lem.~4.8~and~(4.35)]{chen}) that
\begin{align}\label{deform3}
    \begin{split}
        \int_\calm\inn{D_{\bv;q}\, u,A_\bv u} &= \int_\calm\binn{u,D_{\bv;q}A_\bv u} + O(1)\int_\calm |u||A_\bv u|\\
        &\qquad + \int_{\partial \calm} \binn{u,c_\bv(\nu)A_\bv u}.
    \end{split}
\end{align}
It follows from \eqref{deform3} that
    \begin{align}\label{deform4}
        \begin{split}
            \int_\calm|(D_{\bv;q}+A_\bv)u|^2 &= \int_\calm \bbac{|D_{\bv;q}\, u|^2+ \binn{u,[D_{\bv;q},A_\bv]u} +|A_\bv u|^2}\\
            &\qquad  + O(1)\int_\calm |u||A_\bv u| + \int_{\partial\calm}\binn{u,c_\bv(\nu)A_\bv u}.
        \end{split}
    \end{align}
By \cite[(4.40)]{chen}, we have 
    \begin{align}\label{deform5}
        \begin{split}
            \int_\calm |D_{\bv;q}\,u|^2 &= \bket{1+O(\beta)}  \sum_{j=1}^{q}\int_\calm \babs{\nabla^\bv_{e_j} u}^2 + \int_\calm\binn{u, \bket{\calr^\bv_{q} + O\ket{\beta^{-1}}} u} \\ 
            &\qquad\qquad + \int_{\partial \calm} \bbinn{ u,\bbket{\frac{q-1}{2}H^{T\calm}_{\bv;q} - D^\partial_{\bv;q} + O(1)} u}.
        \end{split}
    \end{align}
    By \eqref{deform1} and \eqref{deform4}, we obtain
    \begin{align}\label{deform6}
    \begin{split}
        \int_\calm \bbac{\babs{&\bket{D_{\bv;q} + A_\bv} u}^2 - \babs{\bket{D_\bv  + A_\bv} u}^2}\\
        &=\int_\calm \bbac{\babs{D_{\bv;q}\,u}^2 - \babs{D_\bv u}^2}\\
        &\qquad + \int_\calm \bbac{\binn{u, \bcu{D_{\bv;q}-D_{\bv}\,, A_\bv}u} + O(1)|u||A_\bv u|}.
    \end{split}
    \end{align} 

    Let $H^\calf$ denote the leafwise boundary mean curvature on $\partial \calm$ associated with the metric $g^\calf$ and inward unit normal. By \cite[Prop.~4.9]{chen}, we have
    \begin{equation}\label{deform7}
        H^{T\calm}_{\bv;q} = \frac{H^\calf}{\beta} 
        \text{ \;and \;}
        H^{T\calm}_{\bv} = \frac{(q-1)H^\calf}{(n-1)\beta} + O(\beta\varepsilon^2) \quad \text{on }L_0\cap\partial \calm.
    \end{equation}
    By \eqref{deform2}, \eqref{deform5} and \eqref{deform7}, we obtain 
    \begin{align}\label{deform8}
    \begin{split}
        \int_\calm \bbac{\babs{D_{\bv;q}\,u}^2 - \babs{D_\bv u}^2}
        &= \bbbac{O(\beta)\sum_{j=1}^{q} - \sum_{j=q+1}^{n}} \int_\calm \babs{\nabla^\bv_{e_j} u}^2\\
        &\qquad+ \int_\calm\bbinn{u,\bbket{\calr_q^\bv-\calr^\bv + O\bket{\beta^{-1}}}u}\\
        &\qquad + \int_{\partial \calm} \bbinn{u,\bbket{D^\partial_\bv - D^\partial_{\bv;q} + O(1)}u}.
    \end{split}
    \end{align}
    By \cite[(4.42)]{chen}, we have
    \begin{align}\label{deform9}
        \begin{split}
            \sum_{j=1}^{q}\int_\calm \babs{\nabla^\bv_{e_j} u}^2 &= \frac{1+O(\beta)}{q-1}\int_\calm\!\bbac{q\sum_{j=1}^{q}  \babs{P^{\bv;q}_{e_j} u}^2 + \binn{u, \bket{\calr^\bv_{q} + O\ket{\beta^{-1}}} u}} \\ 
            & \quad+\frac{1+O(\beta)}{q-1}\int_{\partial \calm} \bbinn{u,\bbket{\frac{q-1}{2}H^{T\calm}_{\bv;q} - D^\partial_{\bv;q} + O(1)}u}.
        \end{split}
    \end{align}
    By \eqref{deform7}, we have $O(\beta)\cdot H^{T\calm}_{\bv;q} = O(1)$ on $L_0\cap\partial \calm$. Then \eqref{deform8}--\eqref{deform9} implies that
    \begin{align}\label{deform10}
    \begin{split}
        \int_\calm \bbac{\babs{&D_{\bv;q}\,u}^2 - \babs{D_\bv u}^2}\\
        &= O(\beta)\sum_{j=1}^{q} \int_\calm \babs{P^{\bv;q}_{e_j} u}^2 - \sum_{j=q+1}^{n} \int_\calm \babs{\nabla^\bv_{e_j} u}^2\\
        &\qquad\quad+ \int_\calm\bbinn{u,\bbket{\bket{1+O(\beta)}\calr_q^\bv-\calr^\bv + O\bket{\beta^{-1}}}u}\\
        &\qquad\quad + \int_{\partial \calm} \bbinn{u,\bbket{D^\partial_\bv - \bket{1+O(\beta)} D^\partial_{\bv;q} + O(1)}u}.
    \end{split}
    \end{align}
    The result follows from \eqref{deform6} and \eqref{deform10}.
    \end{proof}
    \begin{rem}\label{rem3}
        Since $(\calm^\lozenge,\calf)$ is almost isometric with respect to \eqref{met3} (or the metrics on $\calm^\lozenge$ in similar forms as \eqref{met5}, \eqref{met6}, \eqref{met71}), \cref{lem3} also holds for Clifford modules over $(\calm_r^\lozenge,g^{T\widetilde{M}_r^\lozenge}_\bv)$.
    \end{rem}

\subsection{Quantitative Infimum Estimate of Leafwise Scalar Curvature}\label{4.2}
\begin{defn}[\cite{gl83}; cf.\ \cite{zhang23,cz24}]
    Let $M$ be a smooth manifold with or without boundary. Let $K\subseteq \inter M$ be a compact subset in the interior of $M$. Let $E=E_+\oplus E_-$ be a $\bbz_2$-graded Hermitian vector bundle over $M$ equipped with an even Hermitian connection $\nabla^E = \nabla^{E_+} \oplus \nabla^{E_-}$. We say $E$ is a \emph{Gromov--Lawson pair} with support $K$ if there exists a smooth, parallel, unitary bundle isomorphism
    \begin{equation}\label{iso1}
        \mathfrak{t}\colon E_+|_{M\setminus K}\to E_-|_{M\setminus K}.
    \end{equation}
\end{defn}

Recall that $\sat(\cdot)$ denotes the saturation of a given subset associated with a foliation. Recall the even-dimensional manifold $M^\lozenge$ defined in \eqref{tmspin}--\eqref{fspin} and the self-adjoint bundle endomorphism $\calr^E_q$ defined in \eqref{end1}. For $a\in\bbr$, the condition
\begin{equation}
k^F + 4\calr^E_q \ge a\quad(\text{resp. }>a)
\end{equation}
means that for all  $0\ne u\in S(F)\otimes E$ or all $0\ne u\in S(TM)\otimes E$,
\begin{equation}
    \binn{u,\ket{k^F + 4\calr^E_q}u}\ge a|u|^2 \quad(\text{resp. }>a|u|^2). 
\end{equation}

The following theorem is a quantitative infimum estimate of the leafwise scalar curvature on noncompact, connected, oriented, complete, foliated manifolds.
\begin{thm}\label{leaf21}
Let $(M,g)$ be a noncompact, connected, oriented, complete Riemannian manifold without boundary. Let $F\subseteq TM$ be a smooth integrable subbundle of rank $q\ge 2$. Assume either $TM$ or $F$ is spin. Define $g^F = g|_F$ on $F$. Let $L \subseteq W$ be compact subsets of $M^\lozenge$. For $\alpha>0$ and $l\in (0,\frac{\pi}{\alpha q})$, suppose $E=E_+\oplus E_-$ is a Gromov--Lawson pair over $M^\lozenge$ with support $K$ satisfying
\begin{enumerate}[label=\textup{(\roman*)},ref=\textup{\roman*}]
    \item $L\subseteq K\subseteq W \cap \sat(L)$ and $\dist_F(K,\partial W) > l$;\label{5.4i}
    \item $\bket{k^F + 4\calr^{E}_q}\big|_{L} > 0$ and $\bket{k^F + 4\calr^{E}_q}\big|_{K\setminus L} \ge 0$;\label{5.4ii}
    \item $\bket{k^F + 4\calr^{E}_q}\big|_{W\setminus K} \geq \alpha^2 q(q-1)$ and $\bket{k^F + 4\calr^{E}_q}\big|_{M^\lozenge\setminus W} \ge -\alpha^2 q(q-1)\tan^2\ket{\frac{\alpha q l}{2}}$.\label{5.4iii}
\end{enumerate}
Then we have
\begin{equation}\label{result3}
    \binn{\widehat{A}\ket{TM^\lozenge}\ch (E_+-E_-),[M^\lozenge]} = 0.
\end{equation}
\end{thm}
\begin{proof}
We only prove the case where $TM$ is spin and of even rank in detail. Applying \cite[Thm.~4.11--4.12]{chen} and Remarks \ref{rem2}, \ref{rem3} to $(\widetilde{\calm}_r^\lozenge,\calf)$ where $\rk \calf =q$, we obtain other cases in a similar way; cf.\ Sections \ref{3.2}--\ref{3.3} and \cite[\S5.2~and~\S6.2]{chen}.

Now we have $M^\lozenge = M$. Take the orthogonal splitting $TM = F\oplus F^\perp$. As in \cref{2.1}, we have the Connes fibration $\bar{\pi}\colon \calm\to M$ and its subset $\ms$ defined in \eqref{ms}. For $r>0$, we have the smooth fibrations $\pp\colon \mm\to M$ that are double of $\ms$. As in \eqref{met4}, for $0<\beta,\varepsilon\leq 1$, we have a rescaled metric $g^{T\mm}_\bv$ defined by
\begin{align}\label{met7}
    T\mm = \hfz\oplus \hf\oplus\hff,
    \quad
    g_{\beta,\varepsilon}^{T\mm}=\beta^2 g^{\widetilde{\mathcal{F}}}\oplus \frac{g^{\widetilde{\mathcal{F}}^\perp_1}}{ \varepsilon^2}\oplus g^{\widetilde{\mathcal{F}}^\perp_2}.
\end{align}
Define $\kk = \pp^{-1}(K)$ and $\ww = \pp^{-1}(W)$. Then $\partial \ww = \pp^{-1}(\partial W)$.

\textbf{Step 1:} 
Take the metric $g^{F^\perp} = g|_{F^\perp}$ on $F^\perp$. Define
\begin{equation}\label{gbeta}
    g_\beta = g^F\oplus \beta^{-2}g^{F^\perp} \quad (0<\beta\le 1)
\end{equation} 
on $M$. Since $g_\beta\ge g$, the metric $g_\beta$ is also complete. Let $\dist_\bv$ and $\dist_\beta$ denote the distance function on $(\mm, g_{\beta,\varepsilon}^{T\mm})$ and $(M,g_\beta)$, respectively. 

Fix an arbitrary $\varpi$ such that
\begin{equation}\label{conrho}
    0<2\varpi<\min \bbbac{\frac{\pi}{\alpha q}\,,\, \dist_F(K,\partial W)} - l.
\end{equation}
By \cite[Prop.~2.1]{chen} and \eqref{conrho}, there exists $\beta_0>0$ such that
\begin{equation}\label{disbeta}
    \dist_{\beta_0} (K,\partial W) \ge l + 2 \varpi.
\end{equation}
 We define $y\colon M \to [0,l+\varpi]$ by 
\begin{equation}\label{ybeta0}
    y(x) = \frac{l+\varpi}{l+2\varpi}\cdot\min\!\bac{\!\dist_{\beta_0}(x,K)\,,\, l+2\varpi}.
\end{equation} 

Note that $\dist_{\beta_0}(K,\partial W) = \dist_{\beta_0}(K, M\setminus \inter(W))$. By \eqref{disbeta}, we can approximate $y$ by a smooth $1$-Lipschitz function $\widetilde{y}\colon (M,g_{\beta_0}) \to [0,l+\varpi]$ such that
\begin{equation}\label{ybetatilde}
    \widetilde{y}|_{U_K} \equiv 0
    \quad\text{and}\quad
    \widetilde{y}|_{M\setminus \inter W} \equiv l+\varpi,
\end{equation}
for some small open neighborhood $U_K$ of $K$; cf.\ \cite[Lem.~7.2]{cz24}. 

Following \cite[(7-3)]{cz24}, we take $\eta <1$ sufficiently close to $1$ such that
\begin{equation}\label{greaterthan}
a_0 \coloneqq \eta^2\tan^2\!\bbbket{\frac{\alpha q \eta (l+\varpi)}{2}}-\tan^2\!\bbket{\frac{{\alpha q l}}{2}} >0,
\end{equation} 
where $a_0$ is defined by the first identity. We define $\psi\colon [0,\frac{\pi}{\alpha q \eta})\to [0,\infty)$ by 
\begin{equation}\label{defpsi0}
    \psi(t) = \frac{\alpha q\eta }{2} \tan\!\bbket{\frac{\alpha q \eta  t}{2}}.
\end{equation}

Assume $\beta\in (0,\beta_0]$ and $\varepsilon>0$ are small. Define $\psi_\beta\colon \mm \to [0,\infty)$ by
\begin{equation}\label{defpsi}
    \psi_\beta=\frac{\psi\ket{\widetilde{y}\circ\pp}}{\beta}.
\end{equation}
Let $|\cdot|_\beta$ denote the fiberwise norm of $g_\beta$. Since $g_{\beta} \ge g_{\beta_0}$, we have $|\di\widetilde{y}|_\beta \le |\di\widetilde{y}|_{\beta_0}\le 1$. By \eqref{control} and \cite[(5.28)]{chen}, the map
\begin{equation}
    \pp\colon(\mm,g^{T\mm}_\bv) \to (M,g_\beta)
\end{equation} 
is $\frac{1}{\beta}$-Lipschitz when $\bv>0$ are sufficiently small. So
\begin{equation}
    |\di (\widetilde{y}\circ\pp)|_\bv\le |\di \widetilde{y}|_\beta\cdot
    |\di\pp|_\bv \le \beta^{-1}
\end{equation}
on $\mm$. Then we have
\begin{numcases}{}
    \psi_\beta|_{\mm \setminus \ww}={\frac{\alpha q\eta  }{2\beta}}\tan\!\bbbket{\frac{\alpha q\eta  (l+\varpi)}{2}} >0,&\label{conpsi1}\\
    \psi_\beta^2-|\di\psi_\beta|_\bv=\frac{\psi(\widetilde{y}\circ\pp)^2}{\beta^2}-\frac{{\psi}'(\widetilde{y}\circ\pp)\cdot|\di(\widetilde{y}\circ\pp)|_\bv}{\beta}  \geq -\frac{\alpha^2q^2\eta^2}{4\beta^2}.&\label{conpsi2}
\end{numcases}

\textbf{Step 2:} Define the pullback Hermitian vector bundle $\ee = \ee_+\oplus\ee_-$ over $\mm$ by $\ee_\pm = \pp^*E_\pm$. Let $\cals_\bv$ be the twisted Hermitian vector bundle over $\mm$ defined by 
\begin{equation}\label{sbv}
    \cals_{\bv} = S_\bv(T\mm)\,\widehat{\otimes}\,\xi\,\widehat{\otimes}\,\ee,
\end{equation}
where $\xi$ is defined in \eqref{xi}. 
Let $\nabla^\bv$, $c_\bv$ and $D_{\bv}$ denote the Hermitian connection, the Clifford action and the twisted Dirac operator on $\cals_\bv$, respectively. Note that $D_{\bv}$ is formally self-adjoint.  

For simplicity, let $|\cdot|$ denote the fiberwise norm associated with the Hermitian metric on $\cals_\bv$. Let $\|\cdot\|$ denote the $\mathrm{L}^2$-norm on $\cals_\bv$ associated with $g^{T\mm}_\bv$. In what follows, the integrations over $\mm$ are taken with respect to the Riemannian measure associated with $g^{T\mm}_\bv$.

Recall the parallel, unitary bundle isomorphism $\mathfrak{t}\colon E_+|_{M\setminus K}\to E_-|_{M\setminus K}$ given in \eqref{iso1}. We define
\begin{equation}\label{thetacale}
    \vartheta^\ee = \begin{pmatrix}0&\pp^*\mathfrak{t}^*\\ \pp^*\mathfrak{t}& 0\end{pmatrix}\colon \ee|_{\mm\setminus\kk} \to  \ee|_{\mm \setminus \kk},
\end{equation} 
where $\pp^*$ is the pullback, and $\mathfrak{t}^*$ is the adjoint of $\mathfrak{t}$.  Define the odd self-adjoint bundle endomorphisms $V\in\Gamma(\en{\cals_\bv})$ and $\vartheta \in \Gamma\bket{\mathrm{End}\bket{\cals_\bv|_{\mm\setminus \kk}}}$ as in \cite[(2.31)]{zhang17} and \cite[Ex.~2.5]{cz24}:
\begin{align}\label{end3}
    V \coloneqq \id 
    \,\widehat{\otimes}\, V^\xi \,\widehat{\otimes}\, \id,
    \quad 
    \vartheta\coloneqq\id
    \,\widehat{\otimes}\, \id
    \,\widehat{\otimes}\, \vartheta^\ee,
\end{align}
where $V^\xi$ is defined around \eqref{vxi}. Note that
\begin{equation}\label{antic}
    \vartheta^2 = \id
    \quad\text{and}\quad
    [V,\vartheta] = V\vartheta + \vartheta V = 0.
\end{equation}
Since $\mathfrak{t}$ is parallel, we know $\vartheta$ is also parallel with respect to $\nabla^\bv$.

Following \cite[Thm.~1.17]{gl83}, we fix a point $x_0\in M$, and let
\begin{equation}\label{defd}
d\colon (M,g)\to [0,\infty)
\end{equation} be a regularization of $x\mapsto\dist_g (x, x_0)$ that is $\frac{3}{2}$-Lipschitz. For $j>0$, define
\begin{equation}\label{defbj}
    B_j = \{x\in M:d(x)\leq j\}.
\end{equation}
Since $W$ is compact, there exists a sufficiently large $j_0>0$ such that
\begin{equation}
    K\subseteq W \subseteq B_{j}
\end{equation}
for all integer $j\geq j_0$. Take a smooth cutoff function $\chi_j$ on $M$ with compact support such that $\chi_j|_{B_{2j}} \equiv 1$. For any $0<\delta\le 1$, define $h_\delta \colon \mm\to [\delta,1]$ as in \eqref{hdelta}. Define
\begin{equation}
    h_{\delta,j} = (\chi_j\circ\pp) (h_\delta - 1) + 1.
\end{equation}
Then $h_{\delta,j} = h_\delta$ on $\pp^{-1}(B_{2j})$ and $h_{\delta,j} \equiv 1$ outside a compact subset.

Note that  $\psi_\beta \vartheta$ extends by zero to a smooth bundle endomorphism on $\ww$. Define
\begin{equation}\label{dd}
    \mathcal{B}_{h_{\delta,j},\psi_\beta,\beta^{-1}}=h_{\delta,j} D_{\bv} h_{\delta,j} + \psi_\beta\vartheta + \beta^{-1}V
\end{equation}
on $\mm$. The noncompact manifold $(\mm,g^{T\mm}_\bv)$ is complete since $(\mm,g^{T\mm})$ is; see \eqref{control}. By the construction of $h_{\delta,j},\psi_\beta,\vartheta$ and $V$, the operator $\mathcal{B}_{h_{\delta,j},\psi_\beta,\beta^{-1}}$ is self-adjoint as an unbounded operator on $\rml^2(\mm,\cals_\bv)$ with domain
\begin{equation}\label{domain}
    \rmh^1_{D_\bv}(\mm,\cals_\bv) = \bac{u\in \rml^2(\mm,\cals_\bv):D_\bv u\in \rml^2(\mm,\cals_\bv)}
\end{equation}
equipped with the graph norm
\begin{equation}\label{grnorm}
    \|u\|^2_{D_\bv} = \|u\|^2 + \|D_\bv u\|^2,
\end{equation}
where $D_\bv u \in \rml^2(\mm,\cals_\bv)$ is understood in the distributional sense. Moreover, the space $\Gamma_c(\cals_\bv)$ of compactly supported smooth sections is dense in \eqref{domain}. 

For any $u\in\Gamma_c(\cals_\bv)$ supported in $\mm\setminus\ww$, we have
\begin{align}\label{fredholm1}
    \begin{split}
        \|\mathcal{B}_{h_{\delta,j},\psi_\beta,\beta^{-1}}\,u\|^2 &= \|(h_{\delta,j} D_\bv h_{\delta,j} + \beta^{-1}V)u\|^2 + \|\psi_\beta u\|^2 \\
        &\geq \frac{\alpha^2 q^2\eta^2}{4\beta^2}\tan^2\!\bbbket{\frac{\alpha q\eta  (l+\varpi)}{2}} \|u\|^2
    \end{split}
\end{align}
by \eqref{conpsi1} and \eqref{antic}. So $\mathcal{B}_{h_{\delta,j},\psi_\beta,\beta^{-1}}$ is Fredholm. 

As in \cite[(2.34)]{zhang17}, it follows from the Atiyah--Singer index theorem \cite{as68} (cf. \cite[III.~Thm.~13.8]{lm89}) and the Gromov--Lawson Relative Index Theorem \cite{gl83} that
\begin{align}\label{ind123}
    \begin{split}
        \ind (\mathcal{B}_{h_{\delta,j},\psi_\beta,\beta^{-1},+})
        &= \binn{\widehat{A}(T\mm)\ch(\xi_+ - \xi_-)\ch(\ee_+ - \ee_-),[\mm]}\\
        &= \binn{\widehat{A}(TM)\ch(E_+ - E_-),[M]}.
    \end{split}
\end{align}

\textbf{Step 3:} Following \cite{gl83} and \cite[p.~115]{bl91}, let $\phi_0\colon [0,\infty)\to [0,1]$ be a smooth function such that $\phi_0 \equiv 0$ on $[0,1]$ and $\phi_0 \equiv 1$ on $[2,\infty)$. For $j\geq j_0$, let $\phi_{j,1},\phi_{j,2}:\mm \to [0,1]$ be the smooth functions defined by
    \begin{align}
        &\phi_{j,1} = \frac{1 - \phi_0\bket{\frac{d\circ \pp}{j}}}{\bbket{\phi_0\bket{\frac{d\circ \pp}{j}}^2 + \bket{1- \phi_0\bket{\frac{d\circ \pp}{j}}}^2}^{\frac{1}{2}}},
        \\
        &\phi_{j,2} = \frac{\phi_0\bket{\frac{d\circ \pp}{j}}}{\bbket{\phi_0\bket{\frac{d\circ \pp}{j}}^2 + \bket{1- \phi_0\bket{\frac{d\circ \pp}{j}}}^2}^{\frac{1}{2}}}.
    \end{align}
Then we have $\phi_{j,1}^2+\phi_{j,2}^2=1$, and for the compact subsets $B_j, B_{2j}$ defined by \eqref{defbj},
\begin{align}
    &\phi_{j,1}\equiv 1 \;\text{ on } \widetilde{\pi}^{-1}_r(B_{j}),
    \qquad 
    \phi_{j,1}\equiv 0 \;\text{ on } \mm\setminus \widetilde{\pi}^{-1}_r(B_{2j}),\\
    &\phi_{j,2}\equiv 0 \;\text{ on } \widetilde{\pi}^{-1}_r(B_{j}),
    \qquad 
    \phi_{j,2}\equiv 1 \;\text{ on } \mm\setminus \widetilde{\pi}^{-1}_r(B_{2j}).
\end{align}

As in \eqref{spe1}, for $u\in \Gamma_\cm(\cals_\bv)$, it follows from $\phi_{j,1}^2 + \phi_{j,2}^2 = 1$ that
\begin{align}\label{splitphi}
\begin{split}
    \|\mathcal{B}_{h_{\delta,j},\psi_\beta,\beta^{-1}}\,u\|^2
    &= \|\phi_{j,1} \mathcal{B}_{h_{\delta,j},\psi_\beta,\beta^{-1}}\,u\|^2+ \|\phi_{j,2} \mathcal{B}_{h_{\delta,j},\psi_\beta,\beta^{-1}}\,u\|^2\\
    &\geq 
    \frac{1}{2}\| \mathcal{B}_{h_{\delta,j},\psi_\beta,\beta^{-1}}(\phi_{j,1}\, u)\|^2+ \frac{1}{2}\| \mathcal{B}_{h_{\delta,j},\psi_\beta,\beta^{-1}}(\phi_{j,2}\, u)\|^2\\
    &\qquad\qquad -\|h_{\delta,j}^2\,c_\bv(\di\phi_{j,1})\,u\|^2 -\|h_{\delta,j}^2\,c_\bv(\di\phi_{j,2})\,u\|^2.
\end{split}
\end{align}
Define the local orthonormal frame $\{e_i\}$ as in \eqref{scale}. Since $g_\beta\ge g$ for $\beta\le 1$, the function $d\colon (M,g_\beta)\to [0,\infty)$ defined in \eqref{defd} is also $\frac{3}{2}$-Lipschitz for the metric $g_\beta$ on $M$. Together with $|\di\pp|_\bv \le \frac{1}{\beta}$, we have on $\mm$ that
\begin{equation}\label{esti51}
    c_\bv(\di\phi_{j,b}) = O\bbbket{\frac{1}{\beta j}}.
\end{equation}
By \eqref{conpsi1} and \eqref{antic}, we have
\begin{align}\label{esti52}
\begin{split}
    \|\mathcal{B}_{h_{\delta,j},\psi_\beta,\beta^{-1}}(\phi_{j,2}\,u)\|^2 &= \|\ket{h_{\delta,j} D_{\bv} h_{\delta,j} + \beta^{-1} V}\phi_{j,2}\,u\|^2+ \|\psi_\beta\,\phi_{j,2}\,u\|^2\\
    &\geq {\frac{\alpha^2q^2\eta^2}{4\beta^2}}\tan^2 \!\bbbket{\frac{\alpha q\eta  (l+\varpi)}{2}}\|\phi_{j,2}\,u\|^2.
\end{split}
\end{align}

We define $\varphi_1,\varphi_2:\mm\to [0,1]$ by \eqref{pou2} on $\calm_r$, and require that they are constant outside $\calm_{\frac{3r}{4}}$. Then $\varphi_1^2+\varphi_2^2=1$ and for $u\in \Gamma_\cm(\cals_\bv)$,
\begin{align}\label{splitvarphi}
\begin{split}
    \|\mathcal{B}_{h_{\delta,j},\psi_\beta,\beta^{-1}}(\phi_{j,1}\,u)\|^2
    &= \|\varphi_1 \mathcal{B}_{h_{\delta,j},\psi_\beta,\beta^{-1}}(\phi_{j,1}\,u)\|^2+ \|\varphi_2 \mathcal{B}_{h_{\delta,j},\psi_\beta,\beta^{-1}}(\phi_{j,1}\,u)\|^2\\
    &\geq 
    \frac{1}{2}\| \mathcal{B}_{h_{\delta,j},\psi_\beta,\beta^{-1}}(\varphi_1\,\phi_{j,1}\,u)\|^2\!-\|h_{\delta,j}^2\,c_\bv(\di\varphi_1)\,\phi_{j,1}\,u\|^2\\
    &\quad  + \frac{1}{2}\| \mathcal{B}_{h_{\delta,j},\psi_\beta,\beta^{-1}}(\varphi_2\,\phi_{j,1}\,u)\|^2\! -\|h_{\delta,j}^2\,c_\bv(\di\varphi_2)\,\phi_{j,1}\,u\|^2.
\end{split}
\end{align}
By \cite[Lem.~2.1]{zhang17}, for $b=1,2$, we have on the compact subset $\ms\cap\pp^{-1}(B_{2j})$ that
\begin{equation}\label{6.17}
    c_\bv(\di\varphi_b)= O_j\bbbket{\frac{1}{\beta r}}+O_{j,r}(1)
    \;\;\text{and}\;\;
    [D_\bv,V] = O_j\bbbket{\frac{1}{\beta r}}+O_{j,r}(1);
\end{equation}
cf.\ \eqref{spe2}. As in \eqref{spe4}, we have 
\begin{align}\label{6.19}
\begin{split}
    \|\mathcal{B}_{h_{\delta,j},\psi_\beta,\beta^{-1}}(\varphi_2\,\phi_{j,1}\,u)\|^2
    &= \int_\mm\babs{(h_{\delta,j} D_\bv h_{\delta,j} + \psi_\beta \vartheta)(\varphi_2\,\phi_{j,1}\, u)}^2\\
    &\quad + \int_{\mm}\!\bbbinn{\varphi_2 \,\phi_{j,1}\,u, \bbbket{\frac{V^2}{\beta^2} +\frac{h^2_{\delta,j}}{\beta}[D_\bv,V]}\varphi_2\,\phi_{j,1}\, u}\\
    &\geq \bbbac{\frac{1}{\beta^2}+O_j\bbbket{\frac{1}{\beta^2 r}}+O_{j,r}\bbbket{\frac{1}{\beta}}\!}\int_\ms|\varphi_2\,\phi_{j,1}\,u|^2\\
    &\quad+\bbbac{\frac{1}{\beta^2}+\delta^2 \,O_{\beta,\varepsilon,j,r}(1)}\int_{\mm\setminus\ms} |\varphi_2\,\phi_{j,1}\,u|^2. 
\end{split}
\end{align}

Let $\cale$ denote the pullback $\bbz_2$-graded Hermitian vector bundle $\bar{\pi}^*E$ over $\calm$. We define the Hermitian vector bundles $\cals_1$ and $\cals_2$ over $\calm$ by 
\begin{align}
    \cals_1 &=S_\bv(T\calm)\,\widehat{\otimes}\,S(\ff)^* \,\widehat{\otimes}\, \cale\label{s1},\\
    \cals_2 &=S_\bv(T\calm)\,\widehat{\otimes}\,(\caly \oplus \caly) \,\widehat{\otimes}\, \cale.\label{s2}
\end{align}
Then there is a unique splitting
\begin{equation}
    \varphi_1\,\phi_{j,1}\,u = u_1 + u_2,
\end{equation}
where $u_b\in \Gamma_c(\cals_b|_\ms)$ for $b=1,2$. As in \eqref{spe5}, we have
\begin{equation}\label{esti53}
\left\|\mathcal{B}_{h_{\delta,j},\psi_\beta,\beta^{-1}}\,u_2\right\|^2 = \int_{\ms}\babs{\ket{D_\bv+\psi_\beta \vartheta}u_2}^2+\beta^{-2}|u_2|^2 \geq \int_{\ms}\beta^{-2}|u_2|^2.
\end{equation}

\textbf{Step 4}: For the local orthonormal frame $\{e_i\}$ defined as in \eqref{scale}, recall that 
\begin{equation}\label{penroseop}
    P^{\bv;q}_{e_i} = \nabla^\bv_{e_i} + \frac{1}{q}c_\bv(e_i)D_{\bv;q}.
\end{equation} 
Recall the bundle endomorphisms $\calr^\bv$,$\calr^{\cale,\bv}_q$ and $\calr^\bv_q$ defined in \eqref{calr}--\eqref{calre} and \eqref{calrq}. In equalities, let the $O(\cdot)$ terms denote possibly complex-valued scalar quantities that may depend on the smooth sections like $u_1$, with bounding constants independent of the sections. By \cite[Thm.~4.11]{chen}, we have
\begin{align}\label{lich}
        \begin{split}
            \int_\ms |D_\bv \,u_1|^2
            &=
            \int_\ms\bbbac{\frac{q+O_{j,r}(\beta)}{q-1}\sum_{i=1}^{q}|P^{\bv;q}_{e_i}u_1|^2 +\sum_{i=q+1}^{n}|\nabla^\bv_{e_i}u_1|^2}\\
            &\quad+
            \int_\ms \bbbinn{u_1, \bbbket{\calr^\bv + \frac{1+ O_{j,r}(\beta)}{q-1}\calr^\bv_{q} + O_{j,r}\bbbket{\frac{1}{\beta}}} u_1},
        \end{split}
    \end{align}
    since ${\rm supp}(u_1)\subseteq \mathcal{M}_{\frac{3r}{4}}\cap \pp^{-1}(B_{2j})$. Set $\calr^\bv_{n} = \calr^\bv$. By \eqref{calr}, \eqref{calrq}, \cite[Lem.~1.3]{zhang17} and \cite[Thm.~4.12]{chen} (cf.\ \cite[(5.66)--(5.67)]{chen}), we have
    \begin{equation}\label{cur8}
        \calr^\bv_{b} = \frac{k^\calf + 4\beta^2 \calr^{\cale,\bv}_q}{4\beta^2} + O_{j,r}\bbbket{\frac{1}{\beta} + \frac{\varepsilon}{\beta^2}},\quad (b=q,n)
    \end{equation}
 on $\calm_{\frac{3r}{4}}\cap \pp^{-1}(B_{2j})$. Define the bundle endomorphism $\calq_\bv\in\Gamma(\en{\cals_1})$ on $\calm$ by
\begin{equation}\label{qbv}
    \calq_\bv = \frac{q}{q-1}\cdot\frac{k^\calf + 4\beta^2 \calr^{\cale,\bv}_q}{4\beta^2}+\psi^2_\beta +c_\bv(\di\psi_\beta)\vartheta.
\end{equation}
In inequalities, the $O(\cdot)$ terms denote real-valued scalar quantities with bounding constants independent of the smooth sections like $u_1$. Recall that $h_{\delta,j}\equiv 1$ on ${\rm supp}(u_1)$. Applying \eqref{6.17} and \eqref{lich}--\eqref{cur8}, and then taking real parts, we obtain
\begin{align}\label{u1}
    \begin{split}
        \|\mathcal{B}_{h_{\delta,j},\psi_\beta,\beta^{-1}}\,u_1\|^2
        &=\int_{\ms}|D_\bv \,u_1|^2 + \int_{\ms}\bbinn{u_1,\bbket{\psi^2_\beta +c_\bv(\di\psi_\beta)\vartheta}u_1}\\
        &\quad+\int_{\ms}\!\bbbinn{u_1,\bbbket{\frac{V^2}{\beta^2}+\frac{[D_\bv,V]}{\beta}}u_1}\\
        &\ge
        \frac{q+O_{j,r}(\beta)}{q-1}\int_{\ms}\sum_{i=1}^{q}|P^{\bv;q}_{e_i}u_1|^2\\
        &\quad+\int_{\ms}\!\bbbinn{u_1,\bbbket{\calq_\bv + O_j\bbbket{\frac{1}{\beta^2 r}}+ O_{j,r}\bbbket{\frac{1}{\beta}+ \frac{\varepsilon}{\beta^2}}}u_1}.
    \end{split}
\end{align}

Since $\vartheta$ is parallel with respect to $\nabla^\bv$, we have 
\begin{equation}\label{a1}
    [D_{\bv;q} - D_\bv,\psi_\beta \vartheta + \beta^{-1}V] = O_{j,r}\bbbket{\frac{1}{\beta} + \frac{\varepsilon}{\beta^2}}
\end{equation}
on $\calm_{\frac{3r}{4}}\cap \pp^{-1}(B_{2j})$; cf.\ \eqref{spe2}. By \eqref{cur8}, we have on $\calm_{\frac{3r}{4}}\cap \pp^{-1}(B_{2j})$ that
\begin{equation}\label{a2}
    \bket{1+O_{j,r}(\beta)}\calr^\bv_{q} - \calr^\bv = O_{j,r}\bbbket{\frac{1}{\beta} + \frac{\varepsilon}{\beta^2}}.
\end{equation}
Taking $A_\bv = \psi_\beta \vartheta + \beta^{-1}V$ in \eqref{penrose} and \eqref{qton}, we obtain from \eqref{a1}--\eqref{a2} that
\begin{align}\label{u3}
\begin{split}
    \int_{\ms}\sum_{i=1}^q \bbabs{&\bbket{\nabla^\bv_{e_i} - \frac{1}{q}c_\bv(e_i)\bket{\psi_\beta \vartheta + \beta^{-1}V}}u_1}^2\\
    &=
    \int_{\ms}\bbbac{\sum_{i=1}^q\babs{P^{\bv;q}_{e_i}u_1}^2 + \frac{1}{q}\babs{(D_{\bv;q} + \psi_\beta \vartheta + \beta^{-1}V) u_1}^2}\\
    &\le \int_{\ms}\bbbac{\bket{1+ O_{j,r}(\beta)}\sum_{i=1}^q\babs{P^{\bv;q}_{e_i}u_1}^2+ \frac{1}{q}|\calb_{h_{\delta,j},\psi_\beta,\beta^{-1}}\,u_1|^2 }\\
    &\qquad\qquad + \int_\ms O_{j,r}\bbbket{\frac{1}{\beta}+ \frac{\varepsilon}{\beta^2}}|u_1|^2.
\end{split}
\end{align}
By \eqref{u1} and \eqref{u3}, we have
\begin{align}\label{u4}
    \begin{split}
        \|\mathcal{B}_{h_{\delta,j},\psi_\beta,\beta^{-1}}\,u_1\|^2
        &\ge
        \frac{1}{2}\|\mathcal{B}_{h_{\delta,j},\psi_\beta,\beta^{-1}}\,u_1\|^2 + \frac{q+O_{j,r}(\beta)}{2(q-1)}\sum_{i=1}^{q}\int_{\ms}|P^{\bv;q}_{e_i}u_1|^2\\
        &\quad + \frac{1}{2}\int_{\ms}\bbbinn{u_1,\bbbket{\calq_\bv + O_j\bbbket{\frac{1}{\beta^2 r}}+ O_{j,r}\bbbket{\frac{1}{\beta}+ \frac{\varepsilon}{\beta^2}}}u_1}\\
        &\ge
        \frac{1}{2}\int_{\ms}\sum_{i=1}^q \bbabs{\bbket{\nabla^\bv_{e_i} - \frac{1}{q}c_\bv(e_i)\bket{\psi_\beta \vartheta + \beta^{-1}V}}u_1}^2\\
        &\quad + \frac{1}{2}\int_{\ms}\bbbinn{u_1,\bbbket{\calq_\bv + O_j\bbbket{\frac{1}{\beta^2 r}}+ O_{j,r}\bbbket{\frac{1}{\beta}+ \frac{\varepsilon}{\beta^2}}}u_1}
    \end{split}
\end{align}
when $\bv>0$ are sufficiently small.

By \eqref{5.4ii} and \eqref{ybetatilde}, there exists $a>0$ independent of $j,\bvr,\delta$ such that
\begin{equation}\label{qa0}
    \calq_\bv|_{\bar{\pi}^{-1}(L)} > 2a\beta^{-2}\quad\text{and}\quad \calq_\bv|_{\bar{\pi}^{-1}(K)}\ge 0.
\end{equation}
Recall the open subset $U_K$ chosen in \eqref{ybetatilde}. Then there exists an open subset $U\subseteq M$ independent of $j,\bvr,\delta$ such that $L\subseteq U\subseteq U_K$ and
\begin{equation}\label{qa1}
    \calq_\bv|_{\bar{\pi}^{-1}(U)} \ge a\beta^{-2}>0.
\end{equation}

We apply \cref{lem2} to $(\calm,g^{T\calm})$ and the restriction of the Hermitian connection $\nabla^\bv$ to $\cals_
1$. In \eqref{result2}, we take 
\begin{equation}\label{Adef}
    A_X = -q^{-1}c_\bv(X)\bket{\psi_\beta \vartheta + \beta^{-1}V}.
\end{equation} 
Then \eqref{acond2} holds for some $a_0>0$. 
Since $K\subseteq \sat(U)$ for the open subset $U$, there exist a compact subset $L_0\subseteq M$ and a constant $C_3>0$ both independent of $j,\bvr,\delta$ such that
\begin{align}\label{u2}
    \begin{split}
        C_3\int_{\bar{\pi}^{-1}(K)}\frac{|u_1|^2}{\beta^2}
        &\le \int_{\bar{\pi}^{-1}(L_0)}\sum_{i=1}^q \bbabs{\bbket{\nabla^\bv_{e_i} - \frac{1}{q}c_\bv(e_i)\bket{\psi_\beta \vartheta + \beta^{-1}V}}u_1}^2\\
        &\qquad\qquad+ \int_{\bar{\pi}^{-1}(U)} \frac{|u_1|^2}{\beta^2}
    \end{split}
\end{align}
holds for any $u_1\in \Gamma_c(\cals_1|_{\calm_{3r/4}})$. Note that $\psi^2_\beta +c_\bv(\di\psi_\beta)\vartheta \ge \psi^2_\beta - |\di\psi_\beta|_\bv$. By \eqref{5.4iii}, \eqref{ybetatilde}--\eqref{greaterthan} and \eqref{conpsi1}--\eqref{conpsi2}, we have
\begin{align}
    &\calq_\bv|_{\bar{\pi}^{-1}(W\setminus K)} \ge \frac{\alpha^2 q^2 (1-\eta^2)}{4\beta^2},\label{q4}\\
    &\calq_\bv|_{\bar{\pi}^{-1}(M\setminus W)} \ge \frac{\alpha^2q^2}{4\beta^2}\bbbket{\eta^2\tan^2\!\bbbket{\frac{\alpha q\eta  (l+\varpi)}{2}} - \tan^2\!\bbbket{\frac{\alpha q l}{2}}} = \frac{\alpha^2 q^2 a_0}{4\beta^2}.\label{q5}
\end{align}
By \eqref{qa0}--\eqref{qa1} and \eqref{q4}--\eqref{q5}, we have 
\begin{align}
    \frac{1}{4}\int_{\ms}\inn{u_1,\calq_\bv u_1}&\ge \int_{\bar{\pi}^{-1}(U)} \frac{a}{4\beta^2}|u_1|^2,\label{q7}\\
    \frac{1}{4}\int_{\ms}\inn{u_1,\calq_\bv u_1} &\ge \int_{\bar{\pi}^{-1}(M\setminus K)} \frac{\min\{\alpha^2q^2(1-\eta^2),\alpha^2q^2a_0\}}{16\beta^2}|u_1|^2.\label{q8}
\end{align}
Sum \eqref{q7}--\eqref{q8}, and then insert the result into \eqref{u4}. Applying \eqref{u2} to the resulting inequality, we conclude that
\begin{align}\label{u5}
    \begin{split}
        \|\mathcal{B}_{h_{\delta,j},\psi_\beta,\beta^{-1}}\,u_1\|^2
        &\ge
        \int_{\ms}\frac{\min\!\bac{4C_3,4aC_3,\alpha^2q^2(1-\eta^2),\alpha^2q^2a_0}}{16\beta^2}|u_1|^2 \\
        &\qquad+\int_{\ms}\bbbac{ O_j\bbbket{\frac{1}{\beta^2 r}}+ O_{j,r}\bbbket{\frac{1}{\beta} + \frac{\varepsilon}{\beta^2}}}|u_1|^2.
    \end{split}
\end{align}

By \eqref{splitphi}--\eqref{6.19}, \eqref{esti53}, \eqref{u5} and $h_{\delta,j}\le 1$, there exist sufficiently large $j\geq j_0$, sufficiently large $r>0$ and sufficiently small $\beta,\varepsilon,\delta,C_0>0$ such that
\begin{equation}\label{u8}
\|\mathcal{B}_{h_{\delta,j},\psi_\beta,\beta^{-1}}\,u\|\geq C_0\|u\|
\end{equation}
holds for all $u\in\Gamma_c(\cals_\bv)$; cf. \eqref{spe8}. Then 
\begin{equation}\label{u9}
{\rm ind}\left(\mathcal{B}_{h_{\delta,j},\psi_\beta,\beta^{-1},+}\right)=0.
\end{equation}
By \eqref{ind123} and \eqref{u9}, we conclude \eqref{result3}.
\end{proof}

\subsection{Proof of Theorem \ref{t1.4} and \ref{t1.5}}\label{4.3}
We first apply \cref{leaf21} to show that \cref{t1.4} holds under the slightly stronger condition $\dist_F(\supp(\di\Phi),\partial W) > l$:

\begin{thm}\label{leaf30}
Let $(M,g)$ be an $m$-dimensional, noncompact, connected, oriented, complete Riemannian manifold without boundary. Let $F\subseteq TM$ be a smooth integrable subbundle of rank $q \ge 2$. Define $g^F = g|_F$. Assume either $TM$ or $F$ is spin. Let $\Phi \colon M\to\bbs^m$ be a smooth map. Suppose $W\subseteq M$ is a codimension-$0$, compact, embedded submanifold with boundary and $L\subseteq W$ is a compact subset such that 
\begin{enumerate}[label=\textup{(\roman*)},ref=\textup{\roman*}]
    \item $L\subseteq \supp(\di\Phi) \subseteq W\cap \sat(L)$ and $\deg\Phi\ne 0$;\label{5.6i}
    \item $k^F|_L > 2\,|\wedge^2\!(\di\Phi|_F)\,|_{\tr}$ and $k^F|_{\supp(\di\Phi)\setminus L} \ge 2\,|\wedge^2\!(\di\Phi|_F)\,|_{\tr}$;\label{5.6ii}
    \item $k^F|_{W \setminus \supp(\di\Phi)} \geq \alpha^2 q(q-1)$ and $\dist_F(\supp(\di\Phi),\partial W) > l$\label{5.6iii}
\end{enumerate}
    for some $\alpha>0$ and $l\in (0,\frac{\pi}{\alpha q})$. Then $\inf_{M} k^F < -\alpha^2 q(q-1)\tan^2\ket{\frac{\alpha q l }{2}}$.
\end{thm}
\begin{proof}
    Consider the orthogonal splitting $TM = F \oplus F^\perp$ and define $g^{F^\perp} = g|_{F^\perp}$. For $0<\beta\le 1$, define the complete metrics $g_\beta = g^F \oplus \beta^{-2}g^{F^\perp}$ on $M$. Let $\dist_\beta$ denote the distance map on $(M,g_\beta)$. Suppose to the contrary that 
    \begin{equation}\label{contr1}
        \inf_M k^F \ge -\alpha^2 q(q-1)\tan^2\!\bbket{\frac{\alpha q l }{2}}.
    \end{equation}
    
    \textbf{Step 1:} Since $\dist_F(\supp(\di\Phi),\partial W)>l$, \cite[Prop.~2.1]{chen} implies that there exist small numbers $\varpi>0$ and $\beta_0>0$ such that
\begin{equation}\label{disvarrho}
    \dist_{\beta_0}\!\bket{\supp(\di\Phi),\partial W}\geq l+2\varpi.
\end{equation}
For $t\in [0,\varpi]$, define a compact subset $K_t\subseteq M$ by
\begin{equation}\label{defkr1}
    K_t=\{x\in W:\dist_{\beta_0}(x,\partial W)\geq l+2\varpi-t\}.
\end{equation}
Then we have
\begin{align}
    &\supp(\di \Phi) \subseteq K_0 \subseteq \inter K_{\varpi/2}\subseteq \inter K_{\varpi}\subseteq W\label{K0prop}\\
    \text{and \,}&\dist_F(K_\varpi,\partial W) \ge \dist_{\beta_0}(K_\varpi,\partial W) >l.\label{prop6}
\end{align}

\textbf{Step 2:} Take a bump function $\zeta_0\colon M\to [0,1]$ such that 
    \begin{equation}\label{cond11}
        \supp \zeta_0\subseteq \inter K_{\varpi}
        \quad\text{and}\quad
        \zeta_0|_{K_{\varpi/2}}\equiv 1.
    \end{equation}
    Since $\partial W$ is smooth and compact, by \eqref{defkr1}--\eqref{K0prop}, there exist finitely many distinct points $p_1,\dots,p_a\in \bbs^m$ such that $\Phi(M\setminus K_0)=\{p_1,\dots,p_a\}$. We define
\begin{equation}\label{omedef}
    \Omega_j=\Phi^{-1}(p_j) \cap (M\setminus K_0).
\end{equation}
Next, fix a basepoint $*\in \bbs^m$. Choose geodesic curves $\gamma_j\colon [0,1]\to \bbs^m$ such that $\gamma_j(1) = p_j$ and $\gamma_j(0) = *$. Let $\Theta:M\to \bbs^m$ be the smooth map defined by
\begin{equation}\label{phi0}
\Theta(x)=
\begin{cases}
\Phi(x),& x\in K_{\varpi/2},\\
\gamma_j(\zeta_0(x)),& x\in \Omega_j,\\
*,& x\in M\setminus K_{\varpi}.
\end{cases}
\end{equation}

Then $\Theta$ is constant at infinity, and we have
\begin{equation}\label{sup0}
    \supp(\di\Theta)\subseteq \inter K_{\varpi},
    \quad
    \deg \Theta = \deg\Phi.
\end{equation}
Since $\Theta|_{\Omega_j}$ factors through smooth curves, we have 
\begin{equation}\label{tr1}
    |\wedge^2\!(\di\Theta|_F)\,|_{\tr} = |\wedge^2\!(\di\Phi|_F)\,|_{\tr} \quad \text{on } M.
\end{equation}

Recall the auxiliary manifold $M^\lozenge$ defined by \eqref{tmspin} or \eqref{fspin} with natural projection $\pi_\lozenge\colon M^\lozenge\to M$. Let $F$ also denote its pullback to $M^\lozenge$ equipped with the metric $g^F$. 
We can require the chosen 1-Lipschitz smooth map $h\colon\bbs^{m}\times \bbs^1\to\bbs^{m+1}$ of degree one collapses $\{*\}\times\bbs^1$ to one point. Then the map $\Theta\colon M\to \bbs^m$ can be extended to a smooth map $\Theta_\lozenge\colon M^\lozenge\to \bbs^{\dim M^\lozenge}$ such that
\begin{equation}\label{sup1}
    \supp(\di\Theta_\lozenge) \subseteq \pi^{-1}_\lozenge(K_\varpi)
    \quad\text{and}\quad 
    \deg\Theta_\lozenge = \deg\Theta.
\end{equation}
Similar to \eqref{tr32}, we have
\begin{equation}\label{tr2}
    |\wedge^2\!(\di\Theta_\lozenge|_F)\,|_{\tr} \leq \pi_\lozenge^*\bket{|\wedge^2\!(\di\Theta|_F)\,|_{\tr}}.
\end{equation}

Let $S(T\bbs^{\dim M^\lozenge}) = S_+(T\bbs^{\dim M^\lozenge})\oplus S_-(T\bbs^{\dim M^\lozenge})$ be the Hermitian bundle of spinors over the standard even-dimensional sphere $\bbs^{\dim M^\lozenge}$. Let $E_\pm\to M^\lozenge$ be the pullback Hermitian vector bundles defined by
\begin{equation}\label{defe}
E_\pm = \Theta_\lozenge^*\bket{S_\pm\ket{T\bbs^{\dim M^\lozenge}}}
\end{equation}
equipped with the pullback Hermitian connection. Then $E = E_+\oplus E_-$ is a Gromov--Lawson pair with support $\pi_\lozenge^{-1}(K_\varpi)$. 

\textbf{Step 3:} We define 
\begin{equation}
    L' = L \cup \bket{K_\varpi \setminus \inter(\supp(\di\Phi))}.
\end{equation}
It follows from $\supp(\di\Phi)\subseteq \sat(L)$ that
\begin{equation}\label{prop7}
    L'\subseteq K_\varpi \subseteq \sat(L').
\end{equation}
Moreover, by \eqref{5.6ii}--\eqref{5.6iii}, we have 
\begin{equation}\label{prop8}
    \bket{k^F + 4\calr^E_q}\big|_{\pi_\lozenge^{-1}(L')} > 0.
\end{equation}

By \cite[(4.6)]{ll98} and \eqref{tr1}, \eqref{tr2}, we have 
\begin{equation}\label{prop3}
    2\calr^E_q \geq - |\wedge^2\!(\di\Theta_\lozenge|_F)\,|_{\tr} \geq - \pi_\lozenge^*\bket{|\wedge^2\!(\di\Phi|_F)\,|_{\tr}}   \quad\text{on } M^\lozenge.
\end{equation}
Moreover, it follows from \eqref{sup0} and \eqref{sup1} that
\begin{equation}\label{prop4}
    \calr^E_q = 0
    \quad\text{on } M^\lozenge\setminus \pi^{-1}_\lozenge(K_\varpi).
\end{equation}

By \eqref{sup0}, \eqref{sup1} and \eqref{defe}, we have
\begin{align}\label{prop5}
    \begin{split}
        \binn{\widehat{A}(T&M^\lozenge)\ch(E_+-E_-),[M^\lozenge]}\\
        &= \binn{\ch\!\bket{S_+(T\bbs^{\dim M^\lozenge})-S_-(T\bbs^{\dim M^\lozenge})},[\bbs^{\dim M^\lozenge}]} \cdot \deg\Theta_\lozenge\\
        & = 2(-1)^{\dim M^\lozenge /2}\,\deg\Phi.
    \end{split}
\end{align}

By \eqref{5.6ii}--\eqref{contr1}, \eqref{prop6} and \eqref{prop7}--\eqref{prop4}, now we can apply \cref{leaf21} to the Gromov--Lawson pair $E$ and the compact subsets
\begin{equation}
    \pi^{-1}_\lozenge(L') \subseteq \pi^{-1}_\lozenge(K_\varpi) \subseteq \pi^{-1}_\lozenge(W)
\end{equation}
of $M^\lozenge$. We conclude from \eqref{result3} and \eqref{prop5} that
\begin{equation}
    \deg\Phi = 0,
\end{equation}
which is a contradiction to the condition that $\Phi$ is of nonzero degree.
\end{proof}

To prove \cref{t1.4}, we need the following conformal deformation lemma.
\begin{lemma}\label{cfm}
    Let $(M,F)$ be a noncompact foliated manifold without boundary, with $q = \rk F\ge 2$. Let $g$ be a complete Riemannian metric on $M$. Define $g^F = g|_F$. Suppose $W\subseteq M$ is compact and
    \begin{equation}
        k^F|_{M}\ge -a \quad \text{and} \quad k^F|_W> -a
    \end{equation}
    for some $a>0$. Then there exists a complete Riemannian metric $\widetilde{g}$ on $M$, which coincides with $g$ on an open neighborhood of $W$, such that the leafwise scalar curvature $\widetilde{k}^F$ associated with $\widetilde{g}^F = \widetilde{g}|_F$ satisfies
    \begin{equation}
        \inf_M\widetilde{k}^F > - a.
    \end{equation}
\end{lemma}
\begin{proof}
    Since $W$ is compact and $k^F|_W > -a$, there exist $\varsigma>0$ and open subsets $U_1, U_2$ such that $\overline{U_2}$ is compact and
    \begin{equation}\label{4.102}
        W\subseteq U_1 \subseteq \overline{U}_1 \subseteq U_2, 
        \quad
        k^F|_{\overline{U}_2} \ge -a + 2\varsigma.
    \end{equation} 
     Take a smooth function $f\colon M\to [0,1]$ such that 
    \begin{equation}
        f|_{U_1}\equiv 0
        \quad\text{and}\quad
        f|_{M\setminus U_2}\equiv 1.
    \end{equation}

    Consider the orthogonal splitting $TM=F\oplus F^\perp$. Define $g^{F^\perp} = g|_{F^\perp}$. For $t>0$, define a family of complete metrics $g_t$ on $M$ by
    \begin{equation}
        g_t = e^{2tf}g^F\oplus g^{F^\perp}.
    \end{equation}
    Then $g_t = g$ on $U_1$. Define $g_t^F = g_t|_F = e^{2tf}g^F$. Let $\nabla^g$ denote the Levi-Civita connection on the leaves associated with $g$, and let $p\colon TM\to F$ be the orthogonal projection. For local orthonormal frame $\{e_i\}_{i=1}^q$ of $(F,g^F)$, we define 
    \begin{equation}\label{delgra}
        \Delta_Ff = \sum_{i=1}^{q}e_i(e_i(f)) - (p\nabla^g_{e_i}e_i)f
        \quad\text{and}\quad
        \di_Ff = \sum_{i=1}^{q}e_i(f)\cdot e_i
    \end{equation}
    Then they are the leafwise Laplacian and the leafwise gradient on $M$ associated with $g^F$, respectively. The leafwise scalar curvature $k_t^F$ associated with $g_t^F$ satisfies
    \begin{equation}\label{scf}
       k_t^F = e^{-2tf} \bac{k^F - 2t(q-1)\Delta_F f - t^2(q-1)(q-2)|\di_Ff|_F^2 }.
    \end{equation}
    By \eqref{4.102}, we have
    \begin{equation}\label{sc1}
        k_{t}^F|_{\overline{U}_2}\ge -a +\varsigma
    \end{equation} 
    for all sufficiently small $t>0$. Moreover, it follows from \eqref{scf} that 
    \begin{equation}\label{sc2}
        k^F_t|_{M\setminus U_2} = e^{-2t}k^F \ge -a e^{-2t}.
    \end{equation}

    For sufficiently small $t>0$, it follows from \eqref{sc1}--\eqref{sc2} that
    \begin{equation}
        \inf_M k^F_t \ge -a + \min\{\varsigma,a(1-e^{-2t})\} > -a.
    \end{equation} 
    So we can take $\widetilde{g} = g_t$ for sufficiently small $t>0$.
\end{proof}

Now, we are ready to prove \cref{t1.4}.
\begin{proof}[Proof of \cref{t1.4}]
    Suppose to the contrary that
    \begin{equation}\label{contr2}
        k^F \ge -\alpha^2 q(q-1)\tan^2\!\bbket{\frac{\alpha q l }{2}}.
    \end{equation}
    By \eqref{contr2}, \cref{cfm}, and \cref{t1.4}\,\eqref{t1.4ii}--\eqref{t1.4iii}, there exists a metric $\widetilde{g}^F$, which equals $g^F$ on a neighborhood of $W$, such that 
    \begin{equation}\label{curass}
        \widetilde{k}^F \ge -\alpha^2 q(q-1)\tan^2\!\bbket{\frac{\alpha q l_0 }{2}}
    \end{equation}
    for some $0<l_0 <l$. Since $\partial W$ is a separating hypersurface, we have
    \begin{equation}\label{dis1}
        \dist_{\widetilde{g}^F}(\supp(\di\Phi),\partial W) = \dist_F(\supp(\di\Phi),\partial W) \ge l > l_0.
    \end{equation}
    Moreover, the conditions \cref{t1.4}\,\eqref{t1.4ii}--\eqref{t1.4iii} still hold for $\widetilde{g}^F$. 
    
    By \eqref{dis1} and \cref{leaf30}, it follows that 
    \begin{equation}
        \inf_M\widetilde{k}^F < -\alpha^2 q(q-1)\tan^2\!\bbket{\frac{\alpha q l_0 }{2}},
    \end{equation}
    which is a contradiction to \eqref{curass}.
\end{proof}

\begin{proof}[Proof of \cref{t1.5}]
    By \cref{t1.5}\,\eqref{t1.8i}--\eqref{t1.8ii}, we have $k^F|_{\partial(\supp(\di\Phi))} >0$. So there exists an open neighborhood $U$ of $\supp(\di\Phi)$ with smooth boundary such that 
    \begin{equation}
        k^F|_{\overline{U}\setminus \supp(\di\Phi)} \ge \alpha^2q(q-1)
    \end{equation}
    for some $\alpha>0$, and $\overline{U}$ is compact. Note that
    \begin{equation}
        \dist_F(\supp(\di\Phi),\partial \overline{U}) \ge \dist_g(\supp(\di\Phi),\partial \overline{U})>0.
    \end{equation} 
    Applying \cref{t1.4} to $W = \overline{U}$, we obtain $\inf_{M} k^F < 0$.
\end{proof}

\subsection{Proof of Theorem \ref{t1.2}}\label{4.4}
Now we turn to \cref{t1.2}, which is a direct consequence of \cref{t1.4}.

\begin{proof}[Proof of \cref{t1.2}]
    Suppose to the contrary that $\dist_F(\supp(\di\Phi),\partial M)>\frac{\pi}{\alpha q}$. Extend $g^F$ to a Riemannian metric $g$ on $M$. 
    
    \textbf{Step 1:} We claim there exists a collar neighborhood $U_1$ around $\partial M$ such that
    \begin{equation}\label{distance3}
        \dist_F(\supp(\di\Phi),\overline{U}_1)>\frac{\pi}{\alpha q}.
    \end{equation} 
    
    The reason is as follows. Fix a smooth collar neighborhood $\varphi\colon \partial M \times [0,1) \to M$ around $\partial M$. For $\epsilon\in (0,1)$, define $U_\epsilon = \varphi(\partial M\times[0,\epsilon))$. Fix $a>0$ such that
    \begin{equation}
        \frac{\pi}{\alpha q} < a < \dist_F(\supp(\di\Phi),\partial M).
    \end{equation} 
    Suppose to the contrary that $\dist_F(\supp(\di\Phi),\overline{U}_\epsilon)\le \frac{\pi}{\alpha q}$ for all $\epsilon\in (0,1)$. Then there exist smooth curves $\gamma_\epsilon\colon [0,1]\to M$ such that $|\dot{\gamma}_\epsilon|$ is constant and
    \begin{equation}\label{curve1}
        \gamma_\epsilon(0)\in \supp(\di\Phi),\quad \gamma_\epsilon(1)\in \overline{U}_{\epsilon},\quad\dot{\gamma}_\epsilon(t)\in F_{\gamma_\epsilon(t)},\quad |\dot{\gamma}_\epsilon| < a
    \end{equation}
    for $t\in [0,1]$. By isometrically embedding $M$ into $\bbr^N$, the Arzelà--Ascoli Theorem and the Banach--Alaoglu theorem allow us to extract a subsequence $\{\gamma_{\epsilon_j}\}_{j=1}^\infty$ converging uniformly to an absolutely continuous curve $\gamma\colon [0,1]\to M$, with their derivatives converging weakly in $\rml^2([0,1],\bbr^N)$ and $\lim_{j\to\infty}\epsilon_j = 0$. Then we have $\gamma(0)\in \supp(\di\Phi)$, $\gamma(1)\in \partial M$, and $\dot{\gamma}(t)\in F_{\gamma(t)}$ almost everywhere. 
    
    Moreover, the leafwise distance $\dist_F(\supp(\di\Phi),\partial M)$ defined using smooth curve is the same as that defined using absolutely continuous curves\footnote{The reason is following. Similar to the construction of $\gamma$, there exists a length-minimizing absolutely continuous curve $\zeta\colon [0,1]\to M$ such that $\zeta(0)\in \supp(\di\Phi), \zeta(1)\in\partial M$, and $\dot{\zeta}(t)\in F_{\zeta(t)}$ almost everywhere. It follows $\zeta([0,1))\subset\inter M$. In local foliation charts inside $\inter M$, the length-minimizing curve $\zeta$ is a smooth leafwise geodesic. So $\zeta|_{[0,1)}$ can be extended smoothly to the end point $1$. So absolutely continuous curves and smooth curves define the same $\dist_F(\supp(\di\Phi),\partial M)$.}. It follows that
    \begin{align}
        \begin{split}
            \dist_F(\supp(\di\Phi),\partial M) \le \int_{0}^{1}|\dot{\gamma}| &\le \bbac{\int_{0}^{1}|\dot{\gamma}|^2}^{1/2}\\ &\le \liminf_{j\to\infty}\bbac{\int_{0}^{1}|\dot{\gamma}_{\epsilon_j}|^2}^{1/2} \le a,
        \end{split}
    \end{align}
    which is a contradiction.

    \textbf{Step 2:} Take a smooth collar neighborhood $U_1\cong\partial M\times [0,1)$ satisfying \eqref{distance3}. Take a smooth function $y\colon M\to [0,1]$ such that $y\equiv 1$ on $M\setminus U_1$ and $y|_{\partial M\times [0,\frac{1}{2}]}$ is the projection to the second factor. Then $y|_{\partial M} \equiv 0$. Define
    \begin{equation}\label{defu}
        u=-\ln y.
    \end{equation}
    
    For the metric $g$ that extends $g^F$, define $\widetilde{g}$ on $\inter M$ by
    \begin{equation}\label{defwg}
        \widetilde{g} = e^{2u}g.
    \end{equation}
    We have $\widetilde{g} = \frac{g}{y^2}$ on $\partial M\times (0,\frac{1}{2})$, which implies that $\widetilde{g}$ is complete on $\inter M$. Moreover, we have $\widetilde{g} = g$ on $M\setminus U_1$. Since $\widetilde{g}|_F = e^{2u}g^F$, the leafwise scalar curvature $\widetilde{k}^F$ on $\partial M\times(0,\frac{1}{2})$ associated with $\widetilde{g}|_F$ is given by
    \begin{align}\label{kf3}
        \begin{split}
            \widetilde{k}^F &= e^{-2u} \bac{k^F - 2(q-1)\Delta_F u - (q-1)(q-2)|\di_Fu|_g^2 }\\
            &= y^2 k^F + 2(q-1)y\Delta_F y - q(q-1)|\di_Fy|_g^2.
        \end{split}
    \end{align}
    Since $M$ is compact, the smooth functions $\Delta_F y$ and $|\di_F y|_g^2$ are bounded on $M$; see \eqref{delgra}. So the leafwise scalar curvature $\widetilde{k}^F$ on $\inter M $ is bounded from below.

    Define $L' = L \cup \partial(\supp(\di\Phi))$. By \cref{t1.2}\,\eqref{1.2ii}--\eqref{1.2iii}, we have
    \begin{equation}\label{add0}
        k^F|_{L'} > 2\,|\wedge^2\!(\di\Phi|_F)\,|_{\tr}.
    \end{equation} 
    Let $\sat_{\inter M}(\cdot)$ denote the saturation of a subset on $(\inter M,F)$ defined as in \eqref{defsat}. For each $x\in \supp(\di\Phi)$, by \cref{t1.2}\,\eqref{1.2i}, there exists a smooth curve $\gamma\colon [0,1]\to M$ such that $\gamma(0)=x$, $\gamma(1)\in L$ and $\dot{\gamma}(t)\in F_{\gamma(t)}$ for $t\in [0,1]$. If $\gamma([0,1])$ is contained in $\supp(\di\Phi)$, then we have $x\in \sat_{\inter M}(L)$. Otherwise, take the smallest number $t_0\in [0,1]$ such that $\gamma(t_0)\in \partial(\supp(\di\Phi))$, and then $x\in \sat_{\inter M}(\partial(\supp(\di\Phi)))$. So
    \begin{equation}\label{add}
    \supp(\di\Phi)\subseteq \sat_{\inter M}(L').
    \end{equation} 
    Applying \cref{t1.4} to the foliated, complete manifold $(\inter M,\widetilde{g})$ and the compact subsets $L'$, $W \coloneqq M\setminus U_1$, it follows from \eqref{distance3} that 
    \begin{equation}\label{kf4}
        \inf_{\inter M} \widetilde{k}^F \leq -\alpha^2 q(q-1)\tan^2\!\bbket{\frac{\alpha q l}{2}}
    \end{equation}
    for any $l\in (0,\frac{\pi}{\alpha q})$. So $\inf_{\inter M} \widetilde{k}^F= -\infty$, which is a contradiction.
\end{proof}
\begin{rem}\label{r4.6}
    If $F$ is further assumed to be transverse to $\partial M$ in \cref{t1.2}, then we can attach a cylinder $U_2\coloneqq\partial M\times (-1,0]$ to $M$ by gluing $\partial M \times \{0\}$ and $\partial M$. Applying the above conformal deformation argument on $U_2$ and applying \cref{t1.4} to $W=M$, we can prove \eqref{1.5}.
\end{rem}

\section{Leafwise Scalar and Mean Curvature Comparison}\label{5}

Finally, we turn to \cref{t1.2-0}. Let $(W,F)$ be a foliated manifold with integrable subbundle $F$ transverse to $\partial W$. Recall the leafwise boundary mean curvature $H^F\in C^\infty(\partial W)$ defined in \eqref{h0}. Let $W^\lozenge$ be the even-dimensional auxiliary manifold defined as in \eqref{tmspin}--\eqref{fspin}. Recall the self-adjoint bundle endomorphism $\calr_q^E$ defined as in \eqref{end1}. The following theorem generalizes \cite[Thm.~4.10]{cz23} to foliations.

\begin{thm}\label{main3}
Let $W$ be a compact, connected, oriented manifold with nonempty boundary $\partial W$. Let $F\subseteq TW$ be an integrable subbundle of rank $q\ge 2$ that is transverse to $\partial W$. Assume either $TW$ or $F$ is spin. Let $g^F$ be a smooth metric on $F$. Let $L \subseteq W^\lozenge$ be a compact subset. For $\alpha >0$ and $l\in (0,\frac{\pi}{\alpha q})$, suppose $E=E_+\oplus E_-$ is a Gromov--Lawson pair over $W^\lozenge$ with support $K$ satisfying
\begin{enumerate}[label=\textup{(\roman*)},ref=\textup{\roman*}]
    \item $L \subseteq K\subseteq \sat_{W^\lozenge}(L)$ and $\dist_F\ket{K,\partial W^\lozenge} > l$;\label{6.1i}
    \item $\bket{k^F + 4\calr^{E}_q}|_{L} > 0$ and $\bket{k^F + 4\calr^{E}_q}|_{K\setminus L} \ge 0$;\label{6.1ii}
    \item $\bket{k^F + 4\calr^{E}_q}|_{W^\lozenge\setminus K} \geq \alpha^2 q(q-1)$ and $H^F|_{\partial W^\lozenge} \geq -\alpha\tan\ket{\frac{\alpha q l}{2}}$.\label{6.1iii}
\end{enumerate}
Then we have
\begin{equation}\label{result4}
    \binn{\widehat{A}\ket{TW^\lozenge}\ch (E_+ - E_-),[W^\lozenge]} = 0.
\end{equation}
\end{thm}
\begin{proof}
    We only prove the case where $TW$ is spin and of even rank in detail. Other cases are similar as we mentioned before; cf.\ \cite[\S5.2 and \S6.2]{chen}. Now $W^\lozenge = W$.

\textbf{Step 1:} Let $W'$ be another copy of $W$ with the opposite orientation. Define
\begin{equation}
    M = W\cup_{\partial W} W'
\end{equation}
to be the smooth double of $W$. Since $F$ is transverse to $\partial W$, it extends to a smooth integrable subbundle on $M$ denoted by the same symbol. 

Take a subbundle $F^\perp\subseteq TM$ such that 
\begin{equation}\label{splitting1}
    F\oplus F^\perp = TM
    \quad\text{and}\quad
    F^\perp|_{\partial W}\subseteq T\partial W.
\end{equation}
As in \cref{2.1}, we have the Connes fibration $\bar{\pi}\colon \calm\to M$ and the fibrations $\pp\colon \mm\to M$ $(r>0)$. The rescaled metric $g^{T\mm}_\bv$ $(0<\beta,\varepsilon\leq 1)$ is defined as in \eqref{met4}. Define $\kk = \pp^{-1}(K)$ and $\ww = \pp^{-1}(W)$. Then $\partial \ww = \pp^{-1}(\partial W)$. Recall the subset $\ms$ defined in \eqref{ms}, we define $\ws = \ww\cap\ms$.

Repeating Step 1 of the proof of \cref{leaf21}, we obtain smooth functions $\psi_\beta\colon \mm\to [0,\infty)$ satisfying \eqref{conpsi1}--\eqref{conpsi2} for an arbitrarily fixed $\varpi$ satisfying \eqref{conrho}, the chosen $\eta$ in \eqref{greaterthan} and sufficiently small $\beta>0$.

By the isomorphism \eqref{iso1}, the Gromov--Lawson pair $E = E_+\oplus E_-$ can be extended to a $\bbz_2$-graded complex vector bundle over $M$ denoted by the same symbol such that $E_\pm|_{W'} = E_-$. Following Step 2 of the proof of \cref{leaf21}, we define $h_\delta$, $\cals_\bv$ and $V$, $\vartheta$ by \eqref{hdelta}, \eqref{sbv} and \eqref{end3}. Let $D_\bv$ and $c_\bv$ denote the Dirac operator and Clifford action on $\cals_\bv$. We define
\begin{equation}\label{defdd1}
    \calb_{h_\delta,\psi_\beta,\beta^{-1}} = h_\delta D_{\bv} h_\delta + \psi_\beta\vartheta + \beta^{-1}V.
\end{equation}

By \eqref{splitting1}, we have $\hf|_{\partial \ww}\subseteq T\partial\ww$ for the bundle $\hf$ in \eqref{met2}. Let $\nu$ denote the inward unit normal on $\partial \ww$ associated with the metric $g^{T\mm}_\bv|_\ww$. Then $\nu\in\Gamma(\hfz)$. Following \cite[Def.~2.3]{cz24}, we define
\begin{equation}\label{bdcon}
    \chi=c_\bv(\nu)\,\vartheta\in \Gamma(\en{\cals_\bv|_{\partial\ww}}).
\end{equation}
As in \cite{cz24}, $\chi$ is a self-adjoint, even involution (that is $\chi^2=\id$), which anticommutes with $c_\bv(\nu)$ and commutes with $c_\bv(X)$ for any $X\in\Gamma(T\partial\ww)$. Define
\begin{equation}
    \Gamma_{\vartheta}(\cals_\bv|_{\ww})=\bac{u\in \Gamma(\cals_\bv|_{\ww}): \chi(u|_{\partial \ww})=u|_{\partial \ww}}.
\end{equation}
Since $\partial \ww$ is compact, the trace theorem implies that the restriction map $\Gamma(\cals_\bv)\to \Gamma(\cals_\bv|_{\partial \ww}), u\mapsto u|_{\partial \ww}$ extends to a continuous linear operator
\begin{equation}
    \scrr\colon \rmh^1(\ww,\cals_\bv) \to \rmh^{1/2}(\partial \ww ,\cals_\bv),
\end{equation}
where $\rmh^1(\ww,\cals_\bv)$ is the Sobolev space of sections of $\cals_\bv$, and $\rmh^{1/2}(\partial \ww,\cals_\bv)$ is the fractional Sobolev space of sections on the boundary.
As in \cite{cz24} (cf. \cite[\S3.2]{chen}), we know $\chi\scrr u = \scrr u$ is a local elliptic boundary condition for the operator $B_{h_\delta,\psi_\beta,\beta^{-1}}$ defined by \eqref{defdd1} in the sense of Bär and Ballmann \cite{bb12,bb16}.

Following \cite[Thm.~3.4]{cz24}, the operator $\mathcal{B}_{h_\delta,\psi_\beta,\beta^{-1}}$ on the compact manifold $\ww$ is self-adjoint and Fredholm as an unbounded operator on $\rml^2(\ww, \cals_\bv)$ with domain
\begin{equation}\label{bdcon2}
    \rmh_{\vartheta}^1(\ww, \cals_\bv)=\bac{u\in \rmh^1(\ww, \cals_\bv) :\chi \scrr u=\scrr u}.
\end{equation}
By ellipticity, the $\rmh^1$-norm and the graph norm \eqref{grnorm} on $\rmh_{\vartheta}^1(\ww, \cals_\bv)$ are equivalent and $\Gamma_{\vartheta}(\cals_\bv|_{\ww})$ is dense in $\rmh_{\vartheta}^1(\ww, \cals_\bv)$.

Since $\ww$ is compact, the operator $\mathcal{B}_{h_\delta,\psi_\beta,\beta^{-1}}$ is homotopic to $\mathcal{B}_{1,\psi_\beta,0} = D_\bv + \psi_\beta \vartheta$ through a continuous family of self-adjoint Fredholm operators on $\rmh^1_{\vartheta}(\ww,\cals_\bv)$. Recall that $D_{\bv}$ is the twisted Dirac operator on $\cals_\bv\to \mm$. It follows that
\begin{equation}\label{ind10}
    \ind\!\bket{\mathcal{B}_{h_\delta,\psi_\beta,\beta^{-1},+}} = \ind\!\bket{\mathcal{B}_{1,\psi_\beta,0,+}} = \ind\!\bket{D_{\bv,+}|_\mm},
\end{equation}
where the last equality follows from \cite[Lem.~3.8]{cz24} and the splitting theorem \cite[Thm.~3.6]{cz24}; cf.\ \cite[Cor.~3.9]{cz24}. As in \cite[(2.34)]{zhang17}, it follows from the Atiyah--Singer index theorem \cite{as68} (cf. \cite[III.~Thm.~13.8]{lm89}) that
\begin{align}\label{ind12}
    \begin{split}
        \ind\!\bket{D_{\bv,+}|_\mm} &= \binn{\widehat{A}(T\mm)\ch(\xi_+ - \xi_-)\ch(\ee_+ - \ee_-),[\mm]}\\
        &= \binn{\widehat{A}(TM)\ch(E_+ - E_-),[M]}\\
        &=\binn{\widehat{A}(TW)\ch (E_+ - E_-),[W]}.
    \end{split}
\end{align}

\textbf{Step 2:} We define $\varphi_1,\varphi_2:\mm\to [0,1]$ by \eqref{pou2} on $\calm_r$, and require that they are constant outside $\calm_{\frac{3r}{4}}$. Similar to \eqref{splitvarphi}, we have
\begin{align}\label{splitvarphi1}
\begin{split}
    \|\mathcal{B}_{h_\delta,\psi_\beta,\beta^{-1}}\,u\|^2
    &= \|\varphi_{1} \mathcal{B}_{h_\delta,\psi_\beta,\beta^{-1}}\,u\|^2+ \|\varphi_{2} \mathcal{B}_{h_\delta,\psi_\beta,\beta^{-1}}\,u\|^2\\
    &\geq 
    \frac{1}{2}\| \mathcal{B}_{h_\delta,\psi_\beta,\beta^{-1}}(\varphi_{1} u)\|^2+ \frac{1}{2}\| \mathcal{B}_{h_\delta,\psi_\beta,\beta^{-1}}(\varphi_{2}\,u)\|^2\\
    &\qquad\qquad -\|h^2_\delta\,c_\bv({\rm d}\varphi_{1})u\|^2 -\|h^2_\delta\,c_\bv({\rm d}\varphi_{2})u\|^2.
\end{split}
\end{align}
By \cite[Lem.~2.1]{zhang17}, there exists a constant $C_1>0$ independent of $\bv,r>0$ so that
\begin{align}\label{esti11}
|c_\bv({\rm d}\varphi_b)|\leq \frac{C_1}{\beta r}+O_{r}(1) \quad\text{on }\ms
\quad\text{and}\quad
\abs{\left[D_\bv,V\right]}\leq {\frac{C_1}{\beta r}}+O_{r}(1)
\end{align}
on $\ms$ for $b=1,2$. From now on, assume $u\in \Gamma_\vartheta(\cals_\bv|_\ww)$. By \cite[(3.25)]{chen}, we have 
\begin{align}\label{3.21}
    \begin{split}
        \|\mathcal{B}_{h_\delta,\psi_\beta,\beta^{-1}}(\varphi_2 u)\|^2
        &= \int_\ww\babs{(h_\delta D_\bv h_\delta + \psi_\beta \vartheta)(\varphi_2 u)}^2\\
        &\qquad\qquad + \int_{\ww}\bbbinn{\varphi_2 u, \bbbket{\frac{V^2}{\beta^2} +\frac{h^2_\delta}{\beta}[D_\bv,V]}\varphi_2 u}\\
        &\geq \bbbac{\frac{1}{\beta^2}+O\bbbket{\frac{1}{\beta^2 r}}+O_r\bbbket{\frac{1}{\beta}}}\int_{\ws}|\varphi_2 u|^2\\
        &\qquad\qquad+\bbbac{\frac{1}{\beta^2}+\delta^2 \, O_{\beta,\varepsilon,r}(1)}\int_{\ww \setminus \ws}|\varphi_2 u|^2. 
    \end{split}
\end{align}

Define $\cale = \bar{\pi}^*E$, and define the Hermitian vector bundles $\cals_1$ and $\cals_2$ over $\ms$ by \eqref{s1}--\eqref{s2}.
Then $B_{h_\delta,\psi_\beta,\beta^{-1}}$ preserves the orthogonal splitting $\cals_\bv|_{\ms} = \cals_1|_{\ms} \oplus \cals_2|_{\ms}$. For $u\in \Gamma_\vartheta(\cals_\bv|_\ww)$, we have the orthogonal splitting
\begin{equation}
    \varphi_1 u = u_1 + u_2,
\end{equation}
where $u_b\in \Gamma_\vartheta(\cals_b|_\ww)$ for $b=1,2$. Then \eqref{esti53} still holds with $h_{\delta,j}$ replaced by $h_\delta$.

The formula \eqref{cur8} still holds with $O_{j,r}$ replaced by $O_r$. Recall the operator \eqref{penroseop}. Since ${\rm supp}(u_1)\subseteq \mathcal{W}_{\frac{3r}{4}}$, by \cite[(3.22),~(4.5)~and~Thm.~4.11]{chen}, we have
\begin{align}\label{3.16}
        \begin{split}
            \int_\ws |D_\bv \,u_1|^2
            &=
            \int_\ws\bbbac{\frac{q+O_r(\beta)}{q-1}\sum_{i=1}^{q}|P^{\bv;q}_{e_i}u_1|^2 +\sum_{i=q+1}^{n}|\nabla^\bv_{e_i}u_1|^2}\\
            &\quad+
            \int_{\ws} \bbbinn{u_1,\bbbket{\frac{q}{q-1}\cdot\frac{k^\calf + 4\beta^2 \calr^{\cale,\bv}_q}{4\beta^2} + O_{r}\bbbket{\frac{1}{\beta} + \frac{\varepsilon}{\beta^2}}}u_1}\\
            &\qquad+ \int_{\ws\,\cap\,\partial\ww} \bbbac{\frac{q H^{\calf}}{2\beta} + O_r(1)}|u_1|^2.
        \end{split}
    \end{align}
    Moreover, \eqref{a1}--\eqref{a2} still hold with $O_{j,r}$ replaced by $O_{r}$. By \eqref{penrose}, \eqref{qton} and \cite[(3.22)~and~(4.5)]{chen}, we have
\begin{align}\label{u10}
\begin{split}
    \int_{\ws}\sum_{i=1}^q \bbabs{&\bbket{\nabla^\bv_{e_i} - \frac{1}{q}c_\bv(e_i)\bket{\psi_\beta \vartheta + \beta^{-1}V}}u_1}^2\\
    &\le \int_{\ws}\bbbac{\bket{1+ O_{r}(\beta)}\sum_{i=1}^q\babs{P^{\bv;q}_{e_i}u_1}^2+ \frac{1}{q}|\calb_{h_\delta,\psi_\beta,\beta^{-1}}\,u_1|^2 }\\
    &\qquad\qquad + \int_\ws O_{r}\bbbket{\frac{1}{\beta}+ \frac{\varepsilon}{\beta^2}}|u_1|^2 + \int_{\partial \ww} O_r(1)\,|u_1|^2;
\end{split}
\end{align}
cf.\ \eqref{u3}. By \cite[(3.26)]{chen} and $h_\delta = 1$ on $\supp(u_1)$, we have
\begin{align}\label{u11}
    \begin{split}
        \|\mathcal{B}_{h_\delta,\psi_\beta,\beta^{-1}}\,u_1\|^2
        &=\int_{\ws}|D_\bv \,u_1|^2+\int_{\ws\,\cap\, \partial \ww} \psi_\beta | u_1|^2\\
        &\quad+\int_{\ws}\bbbinn{u_1,\bbbket{\psi_\beta^2 +c_\bv({\rm d}\psi_\beta)\vartheta+\frac{V^2}{\beta^2}+\frac{[D_\bv,V]}{\beta}}u_1},
    \end{split}
\end{align}
where we used $\inn{u_1,c_\bv(\nu)Vu_1} = 0$ on $\partial \ww$ since $u_1\in \Gamma_{\vartheta}(\cals_1)$; cf.\ \cite[(3.23)]{chen}. Recall the bundle endomorphism $\calq_\bv$ defined in \eqref{qbv}. By \eqref{3.16}--\eqref{u11}, we have
\begin{align}\label{3.17}
    \begin{split}
        \|\mathcal{B}_{h_\delta,\psi_\beta,\beta^{-1}}\,u_1\|^2
        &\ge \frac{1}{2}\int_{\ws}\sum_{i=1}^q \bbabs{\bbket{\nabla^\bv_{e_i} - \frac{1}{q}c_\bv(e_i)\bket{\psi_\beta \vartheta + \beta^{-1}V}}u_1}^2\\
        &\quad+ \frac{1}{2}\int_{\ws} \bbbinn{u_1,\bbbket{\calq_\bv + O\bbbket{\frac{1}{\beta^2 r}} + O_{r}\bbbket{\frac{1}{\beta} + \frac{\varepsilon}{\beta^2}}}u_1}\\
        &\qquad + \frac{1}{2}\int_{\ws\,\cap\,\partial\ww} \bbbac{\frac{qH^{\calf}}{2\beta} + \psi_\beta + O_r(1)}|u_1|^2;
    \end{split}
\end{align}
cf.\ \eqref{u4}. Let $\sat_{\inter W}(\cdot)$ denote the saturation of subsets on $(\inter W,F)$. We define 
\begin{equation}
    L' = L \cup \partial K.
\end{equation}
Similar to \eqref{add0}--\eqref{add}, we have
\begin{equation}\label{control2}
    \bket{k^F + 4\calr^{E}_q}|_{L'} > 0
    \quad\text{and}\quad 
    K \subseteq \sat_{\inter W}(L').
\end{equation} 

By \eqref{ybetatilde} and \eqref{control2}, there also exist $a>0$ independent of $\bvr,\delta$ and an open subset $U\subseteq W$ containing $L'$ such that
\begin{equation}\label{q6}
    \calq_\bv|_{\bar{\pi}^{-1}(U)} \ge a\beta^{-2}>0\quad\text{and}\quad \calq_\bv|_{\bar{\pi}^{-1}(K)}\ge 0.
\end{equation}
By \eqref{6.1iii}, the formula \eqref{q4} still holds. By \eqref{control2}, we have
\begin{equation}
K \subseteq \sat_{\inter W}(U)
\end{equation}
holds for the open subset $U$. By \cref{lem2}, there exist a compact subset $L_0\subseteq \inter W$ and a constant $C_3>0$ both independent of $\bvr,\delta$ such that
\begin{align}\label{poin}
    \begin{split}
        C_3\int_{\bar{\pi}^{-1}(K)}\frac{|u_1|^2}{\beta^2}
        &\le \int_{\bar{\pi}^{-1}(L_0)}\sum_{i=1}^q \bbabs{\bbket{\nabla^\bv_{e_i} - \frac{1}{q}c_\bv(e_i)\bket{\psi_\beta \vartheta + \beta^{-1}V}}u_1}^2\\
        &\qquad\qquad+ \int_{\bar{\pi}^{-1}(U)} \frac{|u_1|^2}{\beta^2}.
    \end{split}
\end{align}
As in \eqref{u5}, it follows from \eqref{q4}, \eqref{3.17}, \eqref{q6} and \eqref{poin} that
\begin{align}\label{u51}
    \begin{split}
        \|\mathcal{B}_{h_\delta,\psi_\beta,\beta^{-1}}\,u_1\|^2
        &\ge
        \frac{\min\!\bac{4C_3,4aC_3,\alpha^2q^2(1-\eta^2)}}{16\beta^2}\int_{\ws}|u_1|^2 \\
        &\qquad+\int_{\ws}\bbbac{ O\bbbket{\frac{1}{\beta^2 r}}+ O_r\bbbket{\frac{1}{\beta} + \frac{\varepsilon}{\beta^2}}}|u_1|^2\\
        &\qquad\qquad 
        +\frac{1}{2}\int_{\ws\,\cap\,\partial\ww} \bbbac{\frac{qH^{\calf}}{2\beta} + \psi_\beta + O_r(1)}|u_1|^2.
    \end{split}
\end{align}

By \eqref{6.1iii}, \eqref{greaterthan} and \eqref{conpsi1}, we have on $\ws\cap\partial\ww$ that
\begin{equation}\label{h1}
    \frac{qH^{\calf}}{2\beta} + \psi_\beta \ge \frac{\alpha q}{2\beta}\bbbac{\eta\tan\!\bbbket{\frac{\alpha q \eta ( l+ \varpi)}{2}}- \tan\bbbket{\frac{\alpha q l}{2}}} > 0.
\end{equation}
By \eqref{esti53}, \eqref{splitvarphi1}--\eqref{3.21}, \eqref{u51}--\eqref{h1} and $h_\delta \le 1$, there exist sufficiently large $r>0$ and sufficiently small $\beta = \varepsilon,\delta,C_0>0$ such that
\begin{align}\label{ind13}
\|\mathcal{B}_{h_\delta,\psi_\beta,\beta^{-1}}\,u\|\geq C_0\|u\| 
\end{align}
holds for all $u\in \Gamma_\vartheta(\cals_\bv|_\ww)$. It follows that
\begin{align}\label{ind14}
{\rm ind}\bket{\mathcal{B}_{h_\delta,\psi_\beta,\beta^{-1},+}}=0.
\end{align}
By \eqref{ind10}--\eqref{ind12} and \eqref{ind14}, we obtain the formula \eqref{result4}.
\end{proof}

Similar to the proof of \cref{leaf30}, we know \cref{main3} implies \cref{t1.2-0}.

\vspace{0.5em}
\noindent\textbf{AI Disclosure}. After completing the initial manuscript, the authors used AI tools for proofreading and for extending \cref{leaf30} to \cref{t1.4} by weakening the strict leafwise distance condition. In particular, \cref{cfm} was suggested by AI. The authors wrote and verified the final manuscript and take full responsibility for its content.

\vspace{0.5em}
\noindent\textbf{Acknowledgments.} 
The authors are grateful to Professor Weiping Zhang for many helpful suggestions. H. Chen was partially supported by National Key R\&D Program of China (2024YFA1013202), National Natural Science Foundation of China (Grant No.\ 11931007 and 12621005), and Nankai Zhide Foundation. G. Su was partially supported by National Natural Science Foundation of China (Grant No.\ 12425106 and 12621005), Natural Science Foundation of Tianjin, China (Grant No.\ 25JCZDJC01110), Fundamental Research Funds for the Central Universities, Grant No.\ 050-63261176, and Nankai Zhide Foundation.

\end{document}